\documentclass[a4paper,11pt]{article}

\usepackage{authblk}
\usepackage{parskip}
\usepackage{mathtools}
\usepackage{amssymb}
\usepackage{amsmath}
\usepackage{latexsym}
\usepackage{mathrsfs}  
\usepackage{bbm}       
\usepackage{amsthm}    
\usepackage{relsize}   
\usepackage{graphicx}
\usepackage{thmtools}  
\usepackage{mathrsfs}  
\usepackage{paralist}  
\usepackage{lipsum}    
\usepackage[all]{xy}   

\numberwithin{equation}{section}

\usepackage{titlesec}   

\titleformat{\section}[block]{\bfseries\filcenter}
{{\upshape\thesection\enspace}}{.5em}{}

\titleformat{\subsection}[block]{\filcenter}
{{\upshape\thesubsection\enspace}}{.5em}{} 

\usepackage{enumitem}  
\setlist{nosep}  
\setitemize[0]{leftmargin=*}
\setenumerate[0]{leftmargin=*}
\setenumerate[1]{label={(\arabic*)}} 

\newcommand{\N}{\mathbb{N}}     
\newcommand{\R}{\mathbb{R}}     
\newcommand{\Prob}{\mathbb{P}}  
\newcommand{\Exp}{\mathbb{E}}   
\newcommand*{\backin}{\rotatebox[origin=c]{-180}{ $\in$ }}%
\newcommand{\st}{\,:\,}         
\newcommand{\goth}[1]{\mathfrak{#1}} 
\newcommand{\ind}[2]{\mathbbm{1}_{#1}\left( #2 \right)}          
\newcommand{\inner}[2]{\left\langle #1 \, , \, #2 \right\rangle} 
\newcommand{\norm}[1]{\left|\left|#1\right|\right|}              
\newcommand{\triplet}[3]{\left( #1, #2, #3 \right) }             
\newcommand{\ProbSpace}{\triplet{\Omega}{\mathscr{F}}{\Prob}}    
\newcommand{\abs}[1]{\left| #1 \right|}                          
\renewcommand{\qedsymbol}{$\square$}                       

\newcommand{\defeq}{\mathrel{\mathop:}=}                         
\newcommand\restr[2]{{
  \left.\kern-\nulldelimiterspace 
  #1 
  \vphantom{\big|} 
  \right|_{#2} 
  }}
  
\theoremstyle{plain} 
\newtheorem{theo}{Theorem}[section]    
\newtheorem{prop}[theo]{Proposition} 
\newtheorem{coro}[theo]{Corollary}
\newtheorem{lemm}[theo]{Lemma}
\newtheorem{assu}[theo]{Assumption}
\newtheorem{rema}[theo]{Remark}

\theoremstyle{definition} 
\newtheorem{defi}[theo]{Definition}
\newtheorem{exam}[theo]{Example}

\newtheorem{nota}[theo]{Notation}

\declaretheoremstyle[%
  spaceabove=-5pt,%
  spacebelow=6pt,%
  headfont=\normalfont\itshape,%
  postheadspace=1em,%
  qed=\qedsymbol%
]{mystyle} 
\declaretheorem[name={Proof},style=mystyle,unnumbered,
]{prf}

 \title{Stochastic Integration and Stochastic PDEs Driven by Jumps on the Dual of a Nuclear Space.}
\author{C. A. Fonseca-Mora}
\affil{  Escuela de Matem\'{a}tica, Universidad de Costa Rica, San Jos\'{e}, 11501-2060, Costa Rica. \\

\noindent E-mail:  christianandres.fonseca@ucr.ac.cr }

\date{}

\begin{document}

 \maketitle

\abstract{We develop a novel theory of weak and strong stochastic integration for cylindrical martingale-valued measures taking values in the dual of a nuclear space. This is applied to develop a theory of SPDEs with rather general coefficients. In particular, we can then study SPDEs driven by general L\'{e}vy processes in this context. }

\smallskip



\section{Introduction}

The aim of this paper is to introduce new foundations for stochastic analysis in duals of nuclear spaces. Our main objective is to study the following abstract stochastic Cauchy problem
\begin{equation} \label{generalSEEIntro}
\begin{cases}
d X_{t}= (A'X_{t}+ B (t,X_{t})) dt+\int_{U} F(t,u,X_{t}) M (dt,du), \quad \mbox{for }t \geq 0, \\
X_{0}=Z_{0}.
\end{cases}
\end{equation}
Here, $M$ is a cylindrical martingale-valued measure taking values in the dual $\Phi'$ of a locally convex space $\Phi$, $A'$ is the dual operator of the generator $A$ of a suitable semigroup on a nuclear space $\Psi$, and the coefficients $B$ and $F$ satisfies appropriate conditions to be described in greater detail below.  

Motivated by the study of the solutions of the equation \eqref{generalSEEIntro}, in this article we introduce a novel theory of stochastic integration for operator-valued processes with respect to cylindrical martingale-valued measures which is suitable for use in the study of stochastic differential equations and in particular to stochastic evolution equations of the form  \eqref{generalSEEIntro}.  

Roughly speaking, a cylindrical martingale-valued measure is a family $M=(M(t,A): t \geq 0, A \in \mathcal{R})$ such that $(M(t,A): t \geq 0)$ is a cylindrical martingale in $\Phi'$ for each $A \in \R$ and $M(t,\cdot)$ is finitely additive on $\mathcal{R}$ for each $t \geq 0$. This concept generalizes to locally convex spaces the martingale-valued measures introduced by Walsh \cite{Walsh:1986} for the finite dimensional setting and then extended to infinite dimensional settings such as Hilbert spaces \cite{Applebaum:2006} and duals of nuclear Fr\'{e}chet spaces \cite{Xie:2001}. We will investigate some classes of cylindrical martingale-valued measure whose second moments are determined by a family of continuous Hilbertian semi-norms $\{ q_{r,u}: r \in \R_{+}, u \in U\}$ on $\Phi$. When $\Phi$ is a nuclear space, examples of $\Phi'_{\beta}$-valued processes that define cylindrical martingale-values measures having such a second moment structure are the generalized Wiener process introduced in \cite{BojdeckiGorostiza:1986, BojdeckiJakubowski:1990} and the martingale part of the L\'{e}vy-It\^{o} decomposition of a L\'{e}vy process taking values in the strong dual $\Phi'_{\beta}$ of a nuclear space. Such decompositions and other properties of  $\Phi'_{\beta}$-valued L\'{e}vy processes were studied by the author in \cite{FonsecaMora:2017}.  

Our next task is to develop our theory of stochastic integration with respect to a cylindrical martingale-valued measure $M$ as described in the last paragraph. We first introduce a theory of integration for vector-valued maps $X=\{X(r,\omega,u): r \in \R_{+}, \omega \in \Omega, u \in U \}$. In particular $\Phi$-valued integrands are included in this class. The resulting stochastic integral $\int_{0}^{t} \int_{U} \, X(r,u)\, M(dr,du)$, called the weak stochastic integral, is a real-valued c\`{a}dl\`{a}g martingale. Our construction of the weak integral is simple and is very general in the sense that $\Phi$ is only required to be locally convex. This in particular offers an alternative simpler approach to the It\^{o} stochastic integral in quasi-complete locally convex spaces introduced in \cite{MikuleviciusRozovskii:1998}.  We will show some of the basic properties of the weak integral and in particular we prove a stochastic Fubini's theorem  that will be of importance for the study of the solutions to equation \eqref{generalSEEIntro}.   

For the second step we will introduce a theory of stochastic integration with respect to $M$ for operator-valued families $R=\{R(r,\omega,u): r \in [0,T], \omega \in \Omega, u \in U \}$ taking values in the strong dual $\Psi'_{\beta}$ of a quasi-complete, bornological, nuclear space $\Psi$. In particular these integrands includes families of linear operators from $\Phi'_{\beta}$ into $\Psi'_{\beta}$. Then, the stochastic integral for these families, called the strong integral, is constructed using a new approach that uses the weak stochastic integral and the regularization theorems for cylindrical processes developed by the author in \cite{FonsecaMora:2016} as building blocks. The constructed process $\int_{0}^{t} \int_{U} \, R(r,u)\, M(dr,du)$ is a $\Phi'_{\beta}$-valued c\'{a}dl\'{a}g martingale.   

Our theory improves on previous studies of stochastic integration for operator-valued processes in the dual of a nuclear space \cite{BojdeckiJakubowski:1989,BojdeckiJakubowski:1990,Ding:1999,Ito,KorezliogluMartias:1988}, in the following two directions: 
\begin{enumerate}
\item We consider a more general integrator, in particular, we can define stochastic integrals with respect to completely general L\'{e}vy processes and to the extent of our knowledge this is the first work that considered these processes as integrators. Also, in contrast to previous works we do not assume our integrator has a version in some Hilbert space contained in $\Phi'_{\beta}$.  
\item We have considered what seems to be the largest class of integrands in the existing literature on the subject. In particular, contrary to all the previous works we do not require our integrands to be families of Hilbert-Schmidt maps from some Hilbert spaces continuously included in $\Phi'_{\beta}$ into a fixed Hilbert space continuously included in $\Psi'_{\beta}$. Moreover, we only require very weak moment conditions for our integrands and these are implied by the stronger conditions satisfied by the integrands in all the works cited above.  
\end{enumerate}
After introducing our theory of stochastic integration, we proceed to study the solutions to \eqref{generalSEEIntro}. Stochastic evolution equations in the dual of a nuclear space have been studied by many authors, for example see \cite{BojdeckiGorostiza:1986, BojdeckiJakubowski:1999, DawsonGorostiza:1990, Ding:1999, Ito, KallianpurMitoma:1992, KallianpurPerezAbreu:1988, PerezAbreuTudor:1994}. However, with the exception of \cite{Ding:1999}, only equations with additive noise have been considered and to the extent of our knowledge no equations with such a general multiplicative noise have been discussed. 
In particular, we are not aware of any work that deals with the general L\'{e}vy noise case. 

Previous studies of stochastic evolution equations in the dual of a nuclear space were strongly motivated by specific applications, such as modelling of the dynamics of nerve signals  \cite{KallianpurWolpert:1984}, environmental pollution \cite{KallianpurXiong},  statistical filtering \cite{KorezliogluMartias:1984}, infinite particles systems \cite{BojdeckiGorostiza:1986}. It is our hope that the more general theory developed here will provide the tools to enable many more applications to be developed. 

We will study the equation \eqref{generalSEEIntro} by assuming that $A'$ is the dual operator to the generator $A$ of a $(C_{0},1)$-semigroup $\{S(t)\}_{t \geq 0}$ on $\Psi$. This class of semigroups, that contains the equicontinuous semigroups, were introduced by Babalola in \cite{Babalola:1974} and has been used previously on the study of stochastic evolution equations in duals of nuclear spaces  \cite{Ding:1999, KallianpurPerezAbreu:1988}.  

In this article we will focus on showing the existence and uniqueness of the so called ``weak" and ``mild" solutions to \eqref{generalSEEIntro} (see Definitions \ref{defiWeakSolution} and \ref{defiMildSolution}). We start our investigation by providing sufficient conditions for the equivalence between weak and mild solutions. Later, as the main result in this article we show the existence and uniqueness of weak and mild solutions to \eqref{generalSEEIntro}. To do so, we will assume that the coefficients $B$ and $F$ satisfy the following 
(cylindrical) growth conditions   
\begin{align*}
\abs{B(r,g)[\psi]} & \leq  a(\psi,r)(1+\abs{g[\psi]}), \\
\int_{U} q_{r,u}(F(r,u,g)'\psi)^{2} \mu(du) & \leq b(\psi,r)^{2}(1+\abs{g[\psi]})^{2}, 
\end{align*} 
for $r \in \R_{+}$, $g \in \Psi'$, and the following (cylindrical) Lipschitz conditions  
\begin{align*}
\abs{B(r,g_{1})[\psi]-B(r,g_{2})[\psi]} & \leq  a(\psi,r) \abs{g_{1}[\psi]-g_{2}[\psi]}, \\
\int_{U} q_{r,u}(F(r,u,g_{1})'\psi-F(r,u,g_{2})'\psi)^{2} \mu(du)  & \leq b(\psi,r)\abs{g_{1}[\psi]-g_{2}[\psi]}^{2}, 
\end{align*}
for $r \in \R_{+}$, $g_{1}, g_{2} \in \Psi'$,  with $a, b: \Psi \times \R_{+} \rightarrow \R_{+}$ satisfying  $
 \int_{0}^{T} \sup_{\psi \in K} ( a(\psi,r)^{2} +  b(\psi,r)^{2}) dr < \infty$ for each $T>0$ and $K \subseteq \Psi$ bounded. Under the above conditions, we will demonstrate by means of a fixed point argument in locally convex spaces that there exists a unique  $\Psi'_{\beta}$-valued predictable process $X=\{X_{t}\}_{t \geq 0}$  that is a weak and a mild solution to \eqref{generalSEEIntro} (see Theorem \ref{theoExistenceAndUniquenessMildSolutions}). Moreover, for every $T>0$, there exists a Hilbert space contained in $\Psi'_{\beta}$ such that $\{X_{t}\}_{t \in [0,T]}$ takes values in this space and has uniformly bounded second moments.  

To the extent of our knowledge the above growth and Lipschitz type conditions have not been considered previously in the literature of stochastic differential equations in duals of nuclear spaces. Indeed, these works  (e.g. see \cite{Ding:1999, Ito, Jakubowski:1995, KallianpurXiong}) always introduce an assumption that the coefficients satisfy growth and Lipschitz type conditions on a prior selected Hilbert space contained in the dual of the nuclear space and hence the existence and uniqueness of solutions reduces to the standard fixed point argument in Hilbert spaces (e.g. as in  \cite{DaPratoZabczyk}). This is not the case in our work because our growth and Lipschitz type conditions   allows the consideration of a larger class of integrands and hence we will need to use a different argument.  

Finally, we apply the above theory to prove the existence and uniqueness for stochastic evolution equations driven by general L\'{e}vy processes taking values in the dual of a nuclear space and whose coefficients satisfying the above growth and Lipschitz type conditions. 

The organization of the paper is the following. In Section \ref{sectionPrelim} we list some important notions on nuclear spaces and their duals, and also properties of cylindrical and stochastic processes, martingales and L\'{e}vy processes in duals of nuclear spaces. In Section \ref{SectionNMVM} we introduce the classes of cylindrical martingale-valued measures that we will use as integrators. Section \ref{SubsectionWSI} is devoted to the construction of the weak stochastic integral. The strong stochastic  integral is developed in Section \ref{SubsectionSSI}. In Section \ref{sectionSEEDNS} we study the existence and uniqueness of the solutions of the equation \eqref{generalSEEIntro}. Finally, in Section \ref{subSectionALNC} we apply our theory to the study of stochastic evolution equations driven by L\'{e}vy processes.

\section{Preliminaries} \label{sectionPrelim}

\subsection{Nuclear Spaces And Their Strong Duals} \label{subsectionNuclSpace}

In this section we introduce our notation and review some of the key concepts on nuclear spaces and its dual space that we will need throughout this paper. For more information see \cite{Schaefer, Treves}. All vector spaces in this paper are real. 

Let $\Phi$ be a locally convex space. If each bounded and closed subset of $\Phi$ is complete, then $\Phi$ is said to be \emph{quasi-complete}. The space $\Phi$ called a \emph{barrelled} space if every  convex, balanced, absorbing and closed subset of $\Phi$ (i.e. a barrel) is a neighborhood of zero. 

If $p$ is a continuous semi-norm on $\Phi$ and $r>0$, the closed ball of radius $r$ of $p$ given by $B_{p}(r) = \left\{ \phi \in \Phi: p(\phi) \leq r \right\}$ is a closed, convex, balanced neighborhood of zero in $\Phi$. A continuous semi-norm (respectively a norm) $p$ on $\Phi$ is called \emph{Hilbertian} if $p(\phi)^{2}=Q(\phi,\phi)$, for all $\phi \in \Phi$, where $Q$ is a symmetric, non-negative bilinear form (respectively inner product) on $\Phi \times \Phi$. Let $\Phi_{p}$ be the Hilbert space that corresponds to the completion of the pre-Hilbert space $(\Phi / \mbox{ker}(p), \tilde{p})$, where $\tilde{p}(\phi+\mbox{ker}(p))=p(\phi)$ for each $\phi \in \Phi$. The quotient map $\Phi \rightarrow \Phi / \mbox{ker}(p)$ has an unique continuous linear extension $i_{p}:\Phi \rightarrow \Phi_{p}$.   

Let $q$ be another continuous Hilbertian semi-norm on $\Phi$ for which $p \leq q$. In this case, $\mbox{ker}(q) \subseteq \mbox{ker}(p)$. Moreover, the inclusion map from $\Phi / \mbox{ker}(q)$ into $\Phi / \mbox{ker}(p)$ is linear and continuous, and therefore it has a unique continuous extension $i_{p,q}:\Phi_{q} \rightarrow \Phi_{p}$. Furthermore, we have the following relation: $i_{p}=i_{p,q} \circ i_{q}$. 

We denote by $\Phi'$ the topological dual of $\Phi$ and by $f[\phi]$ the canonical pairing of elements $f \in \Phi'$, $\phi \in \Phi$. We denote by $\Phi'_{\beta}$ the dual space $\Phi'$ equipped with its \emph{strong topology} $\beta$, i.e. $\beta$ is the topology on $\Phi'$ generated by the family of semi-norms $\{ \eta_{B} \}$, where for each bounded $B \subseteq \Phi'$ we have $\eta_{B}(f)=\sup \{ \abs{f[\phi]}: \phi \in B \}$ for all $f \in \Phi'$.  If $p$ is a continuous Hilbertian semi-norm on $\Phi$, then we denote by $\Phi'_{p}$ the Hilbert space dual to $\Phi_{p}$. The dual norm $p'$ on $\Phi'_{p}$ is given by $p'(f)=\sup \{ \abs{f[\phi]}:  \phi \in B_{p}(1) \}$ for all $ f \in \Phi'_{p}$. Moreover, the dual operator $i_{p}'$ corresponds to the canonical inclusion from $\Phi'_{p}$ into $\Phi'_{\beta}$ and it is linear and continuous. 

Let $p$ and $q$ be continuous Hilbertian semi-norms on $\Phi$ such that $p \leq q$.
The space of continuous linear operators (respectively Hilbert-Schmidt operators) from $\Phi_{q}$ into $\Phi_{p}$ is denoted by $\mathcal{L}(\Phi_{q},\Phi_{p})$ (respectively $\mathcal{L}_{2}(\Phi_{q},\Phi_{p})$) and the operator norm (respectively Hilbert-Schmidt norm) is denote by $\norm{\cdot}_{\mathcal{L}(\Phi_{q},\Phi_{p})}$ (respectively $\norm{\cdot}_{\mathcal{L}_{2}(\Phi_{q},\Phi_{p})}$). We employ an analogous notation for operators between the dual spaces $\Phi'_{p}$ and $\Phi'_{q}$. 
  
Let us recall that a (Hausdorff) locally convex space $(\Phi,\mathcal{T})$ is called \emph{nuclear} if its topology $\mathcal{T}$ is generated by a family $\Pi$ of Hilbertian semi-norms such that for each $p \in \Pi$ there exists $q \in \Pi$, satisfying $p \leq q$ and the canonical inclusion $i_{p,q}: \Phi_{q} \rightarrow \Phi_{p}$ is Hilbert-Schmidt. Other equivalent definitions of nuclear spaces can be found in \cite{Pietsch, Treves}. 

Let $\Phi$ be a nuclear space. If $p$ is a continuous Hilbertian semi-norm  on $\Phi$, then the Hilbert space $\Phi_{p}$ is separable (see \cite{Pietsch}, Proposition 4.4.9 and Theorem 4.4.10, p.82). Now, let $\{ p_{n} \}_{n \in \N}$ be an increasing sequence of continuous Hilbertian semi-norms on $(\Phi,\mathcal{T})$. We denote by $\theta$ the locally convex topology on $\Phi$ generated by the family $\{ p_{n} \}_{n \in \N}$. The topology $\theta$ is weaker than $\mathcal{T}$. We denote by $\Phi_{\theta}$ the space $(\Phi,\theta)$. The space $\Phi_{\theta}$ is a separable pseudo-metrizable (not necessarily Hausdorff) locally convex space and its dual space satisfies $\Phi'_{\theta}=\bigcup_{n \in \N} \Phi'_{p_{n}}$ (see \cite{FonsecaMora:2016}, Proposition 2.4). We denote the completion of $\Phi_{\theta}$ by  $\widetilde{\Phi_{\theta}}$ and its strong dual by $(\widetilde{\Phi_{\theta}})'_{\beta}$.

\subsection{Cylindrical and Stochastic Processes} \label{subSectionCylAndStocProcess}

Unless otherwise specified, in this section $\Phi$ will always denote a nuclear space. 
 
Let $\ProbSpace$  be a complete probability space and consider a filtration $\left\{ \mathcal{F}_{t} \right\}_{t \geq 0}$ on $\ProbSpace$ that satisfies the \emph{usual conditions}, i.e. it is right continuous and $\mathcal{F}_{0}$ contains all sets of $\mathcal{F}$ of $\Prob$-measure zero. We denote by $L^{0} \ProbSpace$ the space of equivalence classes of real-valued random variables defined on $\ProbSpace$. We always consider the space $L^{0} \ProbSpace$ equipped with the topology of convergence in probability and in this case it is a complete, metrizable, topological vector space. We denote by 
$\mathcal{P}_{\infty}$ the predictable $\sigma$-algebra on $[0, \infty) \times \Omega$ and for any $T>0$, we denote by $\mathcal{P}_{T}$ the restriction of $\mathcal{P}_{\infty}$ to $[0,T] \times \Omega$.  


A \emph{cylindrical random variable}\index{cylindrical random variable} in $\Phi'$ is a linear map $X: \Phi \rightarrow L^{0} \ProbSpace$ (see \cite{Ito}). If $X$ is a cylindrical random variable in $\Phi'$, we say that $X$ is \emph{$n$-integrable} if $ \Exp \left( \abs{X(\phi)}^{n} \right)< \infty$, $\forall \, \phi \in \Phi$, and has \emph{zero mean} if $ \Exp \left( X(\phi) \right)=0$, $\forall \phi \in \Phi$. 

Let $X$ be a $\Phi'_{\beta}$-valued random variable, i.e. $X:\Omega \rightarrow \Phi'_{\beta}$ is a $\mathscr{F}/\mathcal{B}(\Phi'_{\beta})$-measurable map. We denote by $\mu_{X}$ the distribution of $X$, i.e. $\mu_{X}(\Gamma)=\Prob \left( X \in  \Gamma \right)$, $\forall \, \Gamma \in \mathcal{B}(\Phi'_{\beta})$, and it is a Borel probability measure on $\Phi'_{\beta}$. For each $\phi \in \Phi$ we denote by $X[\phi]$ the real-valued random variable defined by $X[\phi](\omega) \defeq X(\omega)[\phi]$, for all $\omega \in \Omega$. Then, the mapping $\phi \mapsto X[\phi]$ defines a cylindrical random variable. We will say that a $\Phi'_{\beta}$-valued random variable $X$ is \emph{$n$-integrable} if the cylindrical random variable defined by $X$ is \emph{$n$-integrable}. 
 
Let $J=\R_{+} \defeq [0,\infty)$ or $J=[0,T]$ for  $T>0$. We say that $X=\{ X_{t} \}_{t \in J}$ is a \emph{cylindrical process} in $\Phi'$ if $X_{t}$ is a cylindrical random variable for each $t \in J$. Clearly, any $\Phi'_{\beta}$-valued stochastic processes $X=\{ X_{t} \}_{t \in J}$ defines a cylindrical process under the prescription: $X[\phi]=\{ X_{t}[\phi] \}_{t \in J}$, for each $\phi \in \Phi$. We will say that it is the \emph{cylindrical process associated} to $X$.

If $X$ is a cylindrical random variable in $\Phi'$, a $\Phi'_{\beta}$-valued random variable $Y$ is a called a \emph{version} of $X$ if for every $\phi \in \Phi$, $X(\phi)=Y[\phi]$ $\Prob$-a.e. A $\Phi'_{\beta}$-valued processes $Y=\{Y_{t}\}_{t \in J}$ is said to be a $\Phi'_{\beta}$-valued \emph{version} of the cylindrical process $X=\{X_{t}\}_{t \in J}$ on $\Phi'$ if for each $t \in J$, $Y_{t}$ is a $\Phi'_{\beta}$-valued version of $X_{t}$.  

For a $\Phi'_{\beta}$-valued process $X=\{ X_{t} \}_{t \in J}$ terms like continuous, c\`{a}dl\`{a}g, $\{ \mathcal{F}_{t}\}$-adapted, predictable, etc. have the usual (obvious) meaning.

Throughout this work we will make the following convention: whenever we prove that a process has a predictable/adapted/c\'{a}dl\'{a}g version, we always replace the original process with these version, without further comment. 

A $\Phi'_{\beta}$-valued random variable $X$ is called \emph{regular} if there exists a weaker countably Hilbertian topology $\theta$ on $\Phi$ such that $\Prob( \omega: X(\omega) \in \Phi'_{\theta})=1$. Furthermore, a $\Phi'_{\beta}$-valued process $Y=\{Y_{t}\}_{t \in J}$ is said to be \emph{regular} if $Y_{t}$ is a regular random variable for each $t \in J$. 

The following result contains some useful properties of $\Phi'_{\beta}$- valued regular processes. The proof is standard so we omit it. 

\begin{prop}\label{propCondiIndistingProcess} Let $X=\left\{ X_{t} \right\}_{t \in J}$ and $Y=\left\{ Y_{t} \right\}_{t \in J}$ be $\Phi'_{\beta}$- valued regular stochastic processes such that for each $\phi \in \Phi$, $X[\phi]=\left\{ X_{t}[\phi] \right\}_{t \in J}$ is a version of $Y=\left\{ Y_{t}[\phi] \right\}_{t \in J}$. Then $X$ is a version of $Y$. Furthermore, if $X$ and $Y$ are right-continuous then they are indistinguishable processes. 
\end{prop}

We will need the following important result several times in this article. 

\begin{theo}[\cite{FonsecaMora:2016}, Theorem 4.3] \label{theoExistenceCadlagContVersionHilbertSpaceUniformBoundedMoments}
Let $X=\{X_{t} \}_{t \geq 0}$ be a cylindrical process in $\Phi'$ satisfying:
\begin{enumerate}
\item For each $\phi \in \Phi$, the real-valued process $X(\phi)=\{ X_{t}(\phi) \}_{t \geq 0}$ has a continuous (respectively c\`{a}dl\`{a}g) version.
\item There exists $n \in \N$ and a continuous Hilbertian semi-norm $\varrho$ on $\Phi$ such that for all $T>0$ there exists $C(T)>0$ such that 
\begin{equation} \label{uniformBoundMomentsByHilbertSeminorm}
\Exp \left( \sup_{t \in [0,T]} \abs{X_{t}(\phi)}^{n} \right) \leq C(T) \varrho(\phi)^{n}, \quad \forall \, \phi \in \Phi.
\end{equation} 
\end{enumerate}
Then, there exists a continuous Hilbertian semi-norm $q$ on $\Phi$, $\varrho \leq q$, such that $i_{\varrho,q}$ is Hilbert-Schmidt and there exists a $\Phi'_{q}$-valued continuous (respectively c\`{a}dl\`{a}g) process $Y=\{ Y_{t} \}_{t \geq 0}$, satisfying:
\begin{enumerate}[label=(\alph*)]
\item For every $\phi \in \Phi$, $Y[\phi]= \{ Y_{t}[\phi] \}_{t \geq 0}$ is a version of $X(\phi)= \{ X_{t}(\phi) \}_{t \geq 0}$, 
\item For every $T>0$, $\Exp \left( \sup_{t \in [0,T]} q'(Y_{t})^{n} \right) < \infty$.   
\end{enumerate} 
Furthermore, $Y$ is a $\Phi'_{\beta}$-valued continuous (respectively c\`{a}dl\`{a}g)regular version of $X$ that is unique up to indistinguishable versions.
\end{theo}

\subsection{Martingales in the Strong Dual of a Nuclear Space} \label{subSectionMDNS} 

Let $T>0$ and $n\in \N$. We denote by $\mathcal{M}_{T}^{n}(\R)$  the space of real-valued 
$\{ \mathcal{F}_{t}\}$-adapted  zero-mean $n$-th integrable c\`{a}dl\`{a}g martingales defined on $[0,T]$ and by $\mathcal{M}^{n,loc}_{T}(\R)$ the space of all the real-valued processes defined on $[0,T]$ that are locally in  $\mathcal{M}_{T}^{n}(\R)$. Recall that $\mathcal{M}_{T}^{n}(\R)$ is a Banach space when equipped with the norm $\norm{\cdot}_{\mathcal{M}_{T}^{n}(E)}$ defined by (see \cite{DaPratoZabczyk}, Proposition 3.9, p.79):
\begin{equation}\label{defNormSpaceOfSquareIntegrableCadlagMartingales}
\norm{M}^{n}_{\mathcal{M}_{T}^{n}(\R)}= \Exp \left( \sup_{t\in [0,T]} \abs{M_{t}}^{n} \right), \quad \forall M \in \mathcal{M}_{T}^{n}(\R).
\end{equation}
On the other hand, we equip the space $\mathcal{M}^{n,loc}_{T}(\R)$ with the vector topology $\mathcal{T}_{n,loc}$ generated by the local base of neighbourhoods of zero $\{ O_{\epsilon, \delta}: \epsilon > 0, \delta >0 \}$, where $O_{\epsilon, \delta}$ is given by 
\begin{equation} \label{defiLocalBaseNeigZeroSpaceLocalSquareIntegRealMarting}
 O_{\epsilon, \delta} = \left\{ M \in \mathcal{M}^{n,loc}_{T}(\R): \Prob \left(  \omega \in \Omega: \sup_{t \in [0,T]} \abs{M_{t}(\omega)}^{n} > \epsilon  \right) < \delta  \right\}.
\end{equation}
Then, equipped with the topology $\mathcal{T}_{n,loc}$ the space $\mathcal{M}^{n,loc}_{T}(\R)$ is a complete, metrizable topological vector space. 

We will say that a cylindrical process $M=\{ M_{t} \}_{t \geq 0}$ in $\Phi'$ is a  \emph{cylindrical zero mean $n$-th integrable c\`{a}dl\`{a}g martingale}  in $\Phi'$ (respectively a  \emph{cylindrical locally zero mean $n$-th integrable c\`{a}dl\`{a}g martingale}  in $\Phi'$) if for each $\phi \in \Phi$,  $M(\phi)= \{ M_{t}(\phi) \}_{t \geq 0}$ belongs to $\mathcal{M}_{T}^{n}(\R)$ (respectively to $\mathcal{M}_{T}^{n, loc}(\R)$).  


We denote by $\mathcal{M}^{n}_{T}(\Phi'_{\beta})$ (respectively by $\mathcal{M}^{n,loc}_{T}(\Phi'_{\beta})$) the linear space of all  $\Phi'_{\beta}$-valued regular c\`{a}dl\`{a}g processes such that the associated cylindrical process is a cylindrical zero-mean $n$-th integrable c\`{a}dl\`{a}g martingale (respectively a cylindrical locally zero mean $n$-th integrable c\`{a}dl\`{a}g martingale) in $\Phi'$. 

The following result is based on Theorems 3.1 and 5.2 of \cite{FonsecaMora:2016}. It will play a fundamental role in our construction of the (strong) stochastic integral. 

\begin{theo}  \label{propertiesMartingales}
Let $M=\left\{ M_{t} \right\}_{t \in [0,T]}$ be a cylindrical zero-mean $n$-th integrable c\`{a}dl\`{a}g martingale (respective a cylindrical locally zero-mean $n$-th integrable c\`{a}dl\`{a}g martingale) in $\Phi'$ such that for each $t \in [0,T]$ the map $M_{t}: \Phi \rightarrow L^{0} \ProbSpace$ is continuous. Then, there exists $\widetilde{M}=\{ \widetilde{M}_{t} \}_{t\in [0,T]}$ in $\mathcal{M}^{n}_{T}(\Phi'_{\beta})$ (respectively in $\mathcal{M}^{n,loc}_{T}(\Phi'_{\beta})$) that is a version of $M$. 

Moreover, if $\widetilde{M} \in \mathcal{M}^{n}_{T}(\Phi'_{\beta})$ for $n \geq 2$, then for each $T>0$ there exists a continuous Hilbertian semi-norm $q$ on $\Phi$ such that $\{ \widetilde{M}_{t} \}_{t \in [0, T] }$ is a $\Phi'_{q}$-valued zero-mean c\`{a}dl\`{a}g martingale satisfying $\Exp \left( \sup_{t \in [0, T] } q' ( \widetilde{M}_{t} )^{n} \right) < \infty$.

Furthermore, if for each $ \phi \in \Phi$ the real-valued process $\{ M_{t}(\phi) \}_{t \in [0,T]}$ has a continuous version, then $\widetilde{M}$ can be chosen to be continuous. 
\end{theo}

In what follows we will introduce some topologies on the space $\mathcal{M}^{n}_{T}(\Phi'_{\beta})$. We will need the following result:

\begin{prop} \label{dualNuclearIntegrableMartingalesAsOperatorsToRealIntegrableMartingales}
Let $n \in \N$. The mapping from $\mathcal{M}^{n}_{T}(\Phi'_{\beta})$ into $\mathcal{L}(\Phi,\mathcal{M}^{n}_{T}(\R))$ given by 
\begin{equation} \label{isomorphismDualNuclearMartgAndContLineOperators}
M \mapsto \left(\phi \mapsto M[\phi]=\{ M_{t}[\phi]\}_{t \in [0,T]} \right),
\end{equation}
is a linear isomorphism. 
\end{prop}
\begin{prf} It is easy to check that the map  \eqref{isomorphismDualNuclearMartgAndContLineOperators} is well-defined and that it is linear. Moreover, it is also injective because its kernel only contains the zero vector of $\mathcal{M}^{n}_{T}(\Phi'_{\beta})$. 

Now, let $A\in \mathcal{L}(\Phi,\mathcal{M}^{n}_{T}(\R))$. Then, $A$ defines a cylindrical process in $\Phi'$ such that for each $\phi \in \Phi$, $A\phi=\{(A\phi)_{t} \}_{t \in [0,T]} \in \mathcal{M}^{n}_{T}(\R)$ and such that for each $t \in [0,T]$ the map $\phi \mapsto (A\phi)_{t}$ from $\Phi$ into $L^{0}\ProbSpace$ is continuous. Therefore, from Theorem \ref{propertiesMartingales} there exists $M \in \mathcal{M}^{n}_{T}(\Phi'_{\beta})$ that is a version of the cylindrical process defined by $A$, i.e. for each $t \in [0,T]$, $\Prob$-a.e. we have $M_{t}[\phi]=(A\phi)_{t}$, for all $\phi \in \Phi$. Hence, the map \eqref{isomorphismDualNuclearMartgAndContLineOperators} is  surjective. 
\end{prf}

Now we proceed to introduce vector topologies on the space $\mathcal{M}^{n}_{T}(\Phi'_{\beta})$. First,  we identify each $M$ in $\mathcal{M}^{n}_{T}(\Phi'_{\beta})$ with the corresponding element $\phi \mapsto M[\phi]$ in $\mathcal{L}(\Phi,\mathcal{M}^{n}_{T}(\R))$ given by \eqref{isomorphismDualNuclearMartgAndContLineOperators} (Proposition \ref{dualNuclearIntegrableMartingalesAsOperatorsToRealIntegrableMartingales}). Then, on $\mathcal{M}^{n}_{T}(\Phi'_{\beta})$ we define the topology of bounded (respectively simple) convergence as the locally convex topology generated by the following family of semi-norms:
\begin{equation} \label{semiNormsBoundedConvTopoDualNuclearMartingales}
 M \rightarrow \sup_{ \phi \in B} \norm{ M[\phi]}_{\mathcal{M}_{T}^{n}(\R)}= \sup_{\phi \in B} \Exp \left( \sup_{t\in [0,T]} \abs{M_{t}[\phi]}^{n} \right)^{1/n},
\end{equation}
where $B$ runs over the bounded (respectively finite) subsets of $\Phi$. Hence, the topology of bounded (respectively simple) convergence on $\mathcal{M}^{n}_{T}(\Phi'_{\beta})$ is the topology of bounded (respectively simple) convergence on $\mathcal{L}(\Phi,\mathcal{M}^{n}_{T}(\R))$ defined on $\mathcal{M}^{n}_{T}(\Phi'_{\beta})$ via the isomorphism \eqref{isomorphismDualNuclearMartgAndContLineOperators}.

The next result follows from the corresponding properties of the topologies of bounded and simple convergence of the space $\mathcal{L}(\Phi,\mathcal{M}^{n}_{T}(\R))$ (See \cite{KotheII}, Chapter 39).

\begin{prop}\label{propertiesTopologiesSpaceDualNuclearMartingales} Let $\Phi$ be a barrelled nuclear space. Then, the space $\mathcal{M}^{n}_{T}(\Phi'_{\beta})$ is quasi-complete equipped with either the topology of bounded convergence or the topology of simple convergence. If additionally $\Phi$ is bornological, then $\mathcal{M}^{n}_{T}(\Phi'_{\beta})$ is complete when equipped with the topology of bounded convergence. 
\end{prop}

\subsection{L\'{e}vy Processes in the Dual of a Nuclear Space} \label{subsectionLPDNS}

In this section we review basic properties of L\'{e}vy processes taking values in the dual of a nuclear space. For further details see \cite{FonsecaMora:2017}.  

Let $\Phi$ be a barrelled nuclear space. A $\Phi'_{\beta}$-valued process $L=\left\{ L_{t} \right\}_{t\geq 0}$ is called a \emph{L\'{e}vy process} if \begin{inparaenum}[(i)] \item  $L_{0}=0$ a.s., 
\item $L$ has \emph{independent increments}, i.e. for any $n \in \N$, $0 \leq t_{1}< t_{2} < \dots < t_{n} < \infty$ the $\Phi'_{\beta}$-valued random variables $L_{t_{1}},L_{t_{2}}-L_{t_{1}}, \dots, L_{t_{n}}-L_{t_{n-1}}$ are independent,  
\item L has \emph{stationary increments}, i.e. for any $0 \leq s \leq t$, $L_{t}-L_{s}$ and $L_{t-s}$ are identically distributed, and  
\item For every $t \geq 0$ the distribution $\mu_{t}$ of $L_{t}$ is a Radon measure and the mapping $t \mapsto \mu_{t}$ from $\R_{+}$ into the space $\goth{M}_{R}^{1}(\Phi'_{\beta})$ of Radon probability measures on $\Phi'_{\beta}$ is continuous at $0$ when $\goth{M}_{R}^{1}(\Phi'_{\beta})$  is equipped with the weak topology. \end{inparaenum}

Recall that a \emph{Wiener process} is a $\Phi'_{\beta}$-valued continuous L\'{e}vy process $W=\left\{ W_{t} \right\}_{t\geq 0}$. Every Wiener process $W$ is a Gaussian process and for such a process $W$ there exists $\goth{m} \in \Phi'_{\beta}$ and a continuous Hilbertian semi-norm $\mathcal{Q}$ on $\Phi$, called respectively the \emph{mean} and the \emph{covariance functional} of $W$, such that 
\begin{equation}
\Exp \left( W_{t} [\phi] \right) = t \goth{m} [\phi], \quad  \forall \, \phi \in \Phi, \, t \geq 0.  \label{meanWienerProcess}
\end{equation}  
\begin{equation}
\Exp \left( \left(W_{t}-t \goth{m} \right)[\phi] \left(W_{s}-s \goth{m} \right)[\varphi] \right) = ( t \wedge s ) \mathcal{Q} (\phi, \varphi), \quad \forall \, \phi, \varphi \in \Phi, \, s, t \geq 0.  \label{covarianceFunctWienerProcess}
\end{equation}
where in \eqref{covarianceFunctWienerProcess} $\mathcal{Q}(\cdot,\cdot)$ corresponds to the continuous, symmetric, non-negative bilinear form on $\Phi \times \Phi$ associated to $\mathcal{Q}$ (see \cite{Ito}, Theorem 2.7.1).

\begin{theo}[\cite{FonsecaMora:2017}, Corollary 3.11] \label{theoExistenceCadlagVersionLevyProc}
Let $L=\{ L_{t} \}_{t \geq 0}$ be a $\Phi'_{\beta}$-valued L\'{e}vy process. Then, $L$ has a $\Phi'_{\beta}$-valued, regular, c\`{a}dl\`{a}g version $\tilde{L}=\{ \tilde{L}_{t} \}_{t \geq 0}$ that is also a L\'{e}vy process. Moreover, there exists a weaker countably Hilbertian topology $\vartheta_{L}$ on $\Phi$ such that $\tilde{L}$ is a $(\widetilde{\Phi_{\vartheta_{L}}})'_{\beta}$-valued c\`{a}dl\`{a}g process. 
\end{theo} 

From now on we will always identify a L\'{e}vy process $L$ with its c\`{a}dl\`{a}g version $\tilde{L}$ given in Theorem \ref{theoExistenceCadlagVersionLevyProc}. We also assume that $L$ is $\{ \mathcal{F}_{t} \}$-adapted and we strengthen the property of independent increments of $L$ by assuming that $L_{t}-L_{s}$ is independent of $\mathcal{F}_{s}$ for all $0\leq s <t$. 

For the L\'{e}vy process $L=\left\{ L_{t} \right\}_{t\geq 0}$, we define by $\Delta L_{t} \defeq L_{t}-L_{t-}$ the \emph{jump} of the process $L$ at the time $t\geq 0$. The fact that $L$ is a $(\widetilde{\Phi_{\vartheta_{L}}})'_{\beta}$-valued c\`{a}dl\`{a}g process shows that $\Delta L= \{ \Delta L_{t} \}_{t \geq 0}$ is a stationary Poisson point processes on $\left( \Phi'_{\beta} \setminus \{ 0 \}, \mathcal{B}(\Phi'_{\beta} \setminus \{ 0\}) \right)$. Then $N=\{N(t,A): \, t \geq 0, A \in \mathcal{B}(\Phi'_{\beta} \setminus \{ 0\})\}$, defined by the prescription: 
$$ N(t,A)=\# \left\{ 0 \leq s \leq t \st \Delta L_{s} \in A \right\} = \sum_{0 \leq s \leq t} \ind{A}{\Delta L_{s}}, \quad \forall \, t \geq 0, \, A \in \mathcal{B}( \Phi'_{\beta} \setminus \{ 0\}),$$
is the \emph{Poisson random measure} associated to $\Delta L$ with respect to the ring $\mathcal{A}$ of all the subsets of $\Phi'_{\beta} \setminus \{0\}$ that are \emph{bounded below} (i.e. $A \in \mathcal{A}$ if $0 \notin \overline{A}$, where $\overline{A}$ is the closure of $A$).  
 
Let $\nu$ be the \emph{characteristic measure} of $\Delta L$, i.e. the Borel measure on $\Phi'_{\beta}$ with $\nu( \{ 0 \})=0$ and that satisfies:
$$ \Exp \left( N(t,\Gamma)\right) = t \nu(\Gamma), \quad \forall \, t \geq 0, \, \Gamma \in \mathcal{B}\left(\Phi'_{\beta} \setminus \{0\}\right). $$
Clearly, $\nu(A)< \infty$ for every  $A \in \mathcal{A}$. Moreover, $\nu$  is a \emph{L\'{e}vy measure} on $\Phi'_{\beta}$ in the following sense (see \cite{FonsecaMora:2017}, Theorem 4.12):
\begin{enumerate}
\item $\nu (\{ 0 \})=0$, 
\item for each neighborhood of zero $U \subseteq \Phi'_{\beta}$, the  restriction $\restr{\nu}{U^{c}}$ of $\nu$ on the set $U^{c}$ belongs to the space $\goth{M}^{b}_{R}(\Phi'_{\beta})$ of bounded Radon measures on $\Phi'_{\beta}$,    
\item there exists a continuous Hilbertian semi-norm $\rho$ on $\Phi$ such that 
\begin{equation} \label{integrabilityPropertyLevyMeasure}
\int_{B_{\rho'}(1)} \rho'(f)^{2} \nu (df) < \infty,  \quad \mbox{and} \quad  \restr{\nu}{B_{\rho'}(1)^{c}} \in \goth{M}^{b}_{R}(\Phi'_{\beta}), 
\end{equation}
where we recall that $B_{\rho'}(1) \defeq  \{f \in \Phi': \rho'(f) \leq 1\} = B_{\rho}(1)^{0}$. 
\end{enumerate}
Furthermore, because $\nu$  is a L\'{e}vy measure on $\Phi'_{\beta}$ it follows that $\nu$ is 
 a $\sigma$-finite Radon measure (see \cite{FonsecaMora:2017}, Proposition 4.9).
We will refer to $\nu$ as  \emph{the L\'{e}vy measure of the L\'{e}vy process} $L$. 

Let $A \in \mathcal{B}(\Phi'_{\beta})$ with $\nu(A)< \infty$. For each $t \geq 0$ the \emph{Poisson integral with respect to $N$} is defined by 
\begin{equation} \label{poissonIntegralForIdenity}
\int_{A} \, f \, N(t,df) =  \sum_{0 \leq s \leq t} \Delta L_{s} \ind{A}{\Delta L_{s}}.    
\end{equation}
The Poisson integral defined above is a $\{ \mathcal{F}_{t} \}$-adapted $\Phi'_{\beta}$-valued c\`{a}dl\`{a}g L\'{e}vy process (see \cite{FonsecaMora:2017}, Proposition 4.2). If $\int_{A} \abs{f[\phi]} \nu(df) < \infty$ for each $\phi \in \Phi$, then for each $t \geq 0$ we define the \emph{compensated Poisson integral with respect to} $N$ by 
\begin{equation} \label{compensatedPoissonIntegralForIdenity}
\int_{A} \, f \, \widetilde{N}(t,df)[\phi] = \int_{A} f N(t,df)[\phi] - t \int_{A} f[\phi] \nu(df), \quad \forall  \, \phi \in \Phi.      
\end{equation}
The process  $\left\{ \int_{A} \, f \, \widetilde{N}(t,df): t \geq 0 \right\}$ is a $\Phi'_{\beta}$-valued, zero-mean, square integrable, $\{ \mathcal{F}_{t} \}$-adapted, c\`{a}dl\`{a}g L\'{e}vy process (in particular it belongs to $\mathcal{M}^{2}_{T}(\Phi'_{\beta})$). Moreover, for each $t \geq 0$, if $\int_{A} \abs{f[\phi]}^{2} \nu(df) < \infty$, for each $\phi \in \Phi$, then
\begin{equation} \label{secondMomentCompensatedPoissonIntegral}
\Exp \left( \abs{ \int_{A} \, f \, \widetilde{N}(t,df)[\phi] }^{2} \right) = t \int_{A} \abs{f[\phi]}^{2} \nu(df),  \quad \forall \, \phi \in \Phi. 
\end{equation}

\begin{theo}[L\'{e}vy-It\^{o} decomposition; \cite{FonsecaMora:2017}, Theorem 4.18] \label{levyItoDecompositionTheorem}
Let $L=\left\{ L_{t} \right\}_{t\geq 0}$ be a $\Phi'_{\beta}$-valued L\'{e}vy process. Then, for each $t \geq 0$ it has the following representation
\begin{equation} \label{levyItoDecomposition}
L_{t}=t\goth{m}+W_{t}+\int_{B_{\rho'}(1)} f \widetilde{N} (t,df)+\int_{B_{\rho'}(1)^{c}} f N (t,df)
\end{equation}
where 
\begin{enumerate}
\item $\goth{m} \in \Phi'_{\beta}$, 
\item $\rho$ is a continuous Hilbertian semi-norm on $\Phi$ such that the L\'{e}vy measure $\nu$ of $L$ satisfies \eqref{integrabilityPropertyLevyMeasure} and $B_{\rho'}(1) \defeq \{f \in \Phi'_{\beta}: \rho'(f) \leq 1\} \subseteq \Phi'_{\beta}$ is bounded, closed, convex and balanced, 
\item $\{ W_{t} \}_{t \geq 0}$ is a $\Phi'_{\beta}$-valued Wiener process with mean-zero and covariance functional $\mathcal{Q}$,
\item $\left\{ \int_{B_{\rho'}(1)} f \widetilde{N} (t,df): t\geq 0 \right\}$ is a $\Phi'_{\beta}$-valued mean-zero, square integrable, c\`{a}dl\`{a}g L\'{e}vy process with second moments given by 
\begin{equation*} \label{secondMomentSmallJumpsLID}
\Exp \left( \abs{ \int_{B_{\rho'}(1)} f \widetilde{N} (t,df)[\phi] }^{2}\right) = t \int_{B_{\rho'}(1)} \abs{f[\phi]}^{2} \nu (df), \quad \forall \, t \geq 0, \, \phi \in \Phi, 
\end{equation*} 
\item $\left\{ \int_{B_{\rho'}(1)^{c}} f N (t,df): t\geq 0 \right\}$ is a $\Phi'_{\beta}$-valued c\`{a}dl\`{a}g L\'{e}vy process defined  by means of a Poisson integral with respect to the Poisson random measure $N$ of $L$ on the set $B_{\rho'}(1)^{c}$. 
\end{enumerate}
All the random components of the decomposition \eqref{levyItoDecomposition} are independent.     
\end{theo}

\section{Cylindrical Martingale-Valued Measures} \label{SectionNMVM}

\begin{assu}
Throughout this article $\Phi$ is a locally convex space and $\Psi$ is a quasi-complete, bornological, nuclear space.
\end{assu}  

\begin{defi} \label{martingValuedMeasureOnDualSpace} 
Let $U$ be a topological space and consider a ring $\mathcal{R}\subseteq \mathcal{B}(U)$ that generates $\mathcal{B}(U)$.  A \emph{cylindrical martingale-valued measure} on $\R_{+} \times \mathcal{R}$ is a collection $M=(M(t,A): t \geq 0, A \in \mathcal{R})$ of cylindrical random variables in $\Phi'$ such that:
\begin{enumerate}
\item $\forall \, A \in \mathcal{R}$, $M(0,A)(\phi)= 0$ $\Prob$-a.s., $\forall \phi \in \Phi$.
\item $\forall t \geq 0$, $M(t,\emptyset)(\phi)= 0$ $\Prob$-a.s. $\forall \phi \in \Phi$ and if $A, B \in \mathcal{R}$ are disjoint then 
$$M(t,A \cup B)(\phi)= M(t,A)(\phi) + M(t,B)(\phi) \, \Prob \mbox{-a.s.}, \quad \forall \phi \in \Phi.$$
\item $\forall \, A \in \mathcal{R}$, $(M(t,A): t \geq 0)$ is a cylindrical zero-mean square integrable c\`{a}dl\`{a}g martingale. 
\item For disjoint $A, B \in \mathcal{R}$, $\Exp \left( M(t,A)(\phi) M(s,B)(\varphi) \right)=0$, for each $t,s \geq 0$, $\phi, \varphi \in \Phi$. 
\end{enumerate}
Moreover, we say that $M$ has \emph{independent increments} if whenever $0\leq s < t$, $M ( (s, t], A)(\phi) \defeq (M(t,A)- M(s,A))(\phi)$ is independent of $\mathcal{F}_{s}$, for all $A \in \mathcal{R}$, $\phi \in \Phi$. 
\end{defi}

We will construct stochastic integrals with respect to the following class of cylindrical martingale-valued measures: 

\begin{defi}\label{nuclearMartingaleValMeasDualSpace} A cylindrical martingale-valued measure $M$ on $\R_{+} \times \mathcal{R}$ with independent increments is said to be \emph{nuclear} if for each $A \in \mathcal{R}$ and $0 \leq s < t$, 
\begin{equation} \label{covarianceFunctionalNuclearMartValuedMeasure}
\Exp \left( \abs{ M((s,t],A)(\phi)}^{2} \right) = \int_{s}^{t} \int_{A} q_{r,u}(\phi)^{2} \mu(du) \lambda (dr) , \quad \forall \, \phi \in \Phi.
\end{equation} 
where 
\begin{enumerate}
	\item $\mu$ is a $\sigma$-finite measure on $(U, \mathcal{B}(U))$ satisfying $\mu(A)< \infty$, $\forall \, A \in \mathcal{R}$,
	\item $\lambda$ is a $\sigma$-finite measure on $(\R_{+},\mathcal{B}(\R_{+}))$, finite on bounded intervals,
	\item $\{q_{r,u}: r \in \R_{+}, \, u \in U \}$ is a family of  continuous Hilbertian semi-norms on $\Phi$, such that for each $\phi$, $\varphi$ in $\Phi$, the map $(r,u) \mapsto q_{r,u}(\phi,\varphi)$ is $\mathcal{B}(\R_{+}) \otimes \mathcal{B}(U)/ \mathcal{B}(\R_{+})$-measurable and  bounded on $[0,T] \times U$ for all $T>0$. Here, $q_{r,u}(\cdot,\cdot)$ denotes the positive, symmetric, bilinear form associated to the Hilbertian semi-norm $q_{r,u}$.  
\end{enumerate}
\end{defi}


\begin{exam} \label{examGeneralizedWienerProcesses} 
Let $\Phi$ be a nuclear space.  A $\Phi'_{\beta}$-valued continuous zero-mean Gaussian process $W=\{ W_{t} \}_{t \geq 0}$ is called a \emph{generalized Wiener process} if (see \cite{BojdeckiGorostiza:1986, BojdeckiJakubowski:1990}): 
\begin{enumerate}
\item $W$ is $\{ \mathcal{F}_{t} \}$-adapted, 
\item $W_{t}-W_{s}$ is independent of $\mathcal{F}_{s}$, for $0 \leq s < t$, 
\item 
\begin{equation} \label{covarianceFunctionalGeneralizedWiener}
\Exp \left( W_{t}[\phi] W_{s}[\varphi] \right) = \int_{0}^{t \wedge s} q_{r}(\phi,\varphi)dr , \quad \forall \, t, s \in R_{+}, \, \phi \in \Phi. 
\end{equation} 
\end{enumerate} 
where $\{q_{r}: r \in \R_{+}\}$ is a family of continuous Hilbertian semi-norms on $\Phi$, such that the map $r \mapsto q_{r}(\phi,\varphi)$ is Borel measurable and bounded on finite intervals, for each $\phi$, $\varphi$ in $\Phi$. As in Definition \ref{nuclearMartingaleValMeasDualSpace}, $q_{r}(\cdot,\cdot)$ denotes the positive, symmetric, bilinear form associated to the Hilbertian semi-norm $q_{r}$.  It is clear that every $\Phi'_{\beta}$-valued Wiener process $W$ is a generalized Wiener process and that if $\mathcal{Q}$ is the covariance functional of $W$, one has \eqref{covarianceFunctionalGeneralizedWiener} with $q_{r}=\mathcal{Q}$, for all $r \in \R_{+}$. 

Is easy to see from the definition of $W$ and from Definition \ref{nuclearMartingaleValMeasDualSpace} that if we take $M$ given by 
\begin{equation} \label{generalizedWienerMartVM}
M(t,A)=W_{t} \delta_{0}(A), \quad \forall \, t \in \R_{+}, \, A \in \mathcal{B}(\{ 0 \}), 
\end{equation}
then $M$ defines a cylindrical martingale-valued measure with independent increments. Moreover, for each $0 \leq s < t$, we have:
\begin{equation} \label{covarianceGeneralizedWienerMVM}
\Exp \left( \abs{ M((s,t],\{ 0 \})[\phi]}^{2} \right) = \int_{s}^{t} q_{r}(\phi)^{2} dr , \quad \forall \, \phi \in \Phi, 
\end{equation} 
and hence $M$ is nuclear, where with respect to the notation in Definition \ref{nuclearMartingaleValMeasDualSpace} we have: \begin{inparaenum}[(i)] \item $U=\{ 0 \}$,  $\mathcal{R}=\mathcal{B}(\{ 0 \})$ and $\mu = \delta_{0}$. 
\item $\lambda$ is the Lebesgue measure on $(\R_{+}, \mathcal{B}(\R_{+}))$, and \item $q_{r,0}=q_{r}$, where $\{q_{r}: r \in \R_{+}\}$ are as given above. \end{inparaenum} 
\end{exam}

\begin{exam} \label{examLevyMartValuedMeasure}
Let $\Phi$ be a barrelled nuclear space and let $L$ be a $\Phi'_{\beta}$-valued L\'{e}vy process with L\'{e}vy-It\^{o} decomposition \eqref{levyItoDecomposition}. Let $U \in \mathcal{B}(\Phi'_{\beta})$ be such that $0 \in U$ and $\int_{U} \, \abs{u[\phi]}^{2} \nu(du)< \infty$ for every $\phi \in \Phi$. Take  $\mathcal{R}= \{  U \cap \Gamma: \Gamma \in \mathcal{A} \} \cup \{0\}$, and $M=(M(t,A): r \geq 0, A \in \mathcal{R})$ be given by
\begin{equation} \label{levyMartValuedMeasExam} 
M(t,A) = W_{t} \delta_{0}(A) + \int_{A \backslash \{0 \}} u \widetilde{N}(t,du), \quad \mbox{ for } \, t \geq 0, \, A \in \mathcal{R}. 
\end{equation}
For each $t \geq 0$, $A \in \mathcal{R}$, because $\int_{A} \, \abs{u[\phi]}^{2} \nu(du)< \infty$ for every $\phi \in \Phi$, it follows that $M(t,A)$ is well-defined. Then, it is easy to check that $M$ is a nuclear cylindrical martingale-valued measure with independent increments. Moreover, for each $0 \leq s < t$, $A \in \mathcal{R}$ we have:
\begin{equation} \label{secondMomentLevyNuclearMartValuedMeasu}
\Exp \left( \abs{M((s,t],A)[\phi]}^{2} \right)=(t-s) \left[ \mathcal{Q}(\phi)^{2} + \int_{A \backslash \{0 \}} \abs{u[\phi]}^{2}  \nu(du) \right], \quad \forall \, \phi \in \Phi,
\end{equation}
and hence $M$ is also nuclear. With respect to the notation in Definition \ref{nuclearMartingaleValMeasDualSpace} we have: \begin{inparaenum}[(i)]
\item $\mu = \delta_{0}+ \restr{\nu}{U}$,  
\item $\lambda$ is the Lebesgue measure on $(\R_{+}, \mathcal{B}(\R_{+}))$, 
\item $\{q_{r,f}: r \in \R_{+}, u \in U \}$ given by 
\begin{equation}\label{defiSemiNormsLevyMartValuedMeas}
q_{r,u}(\phi)= \begin{cases} \mathcal{Q}(\phi), & \mbox{if } u=0, \\  \abs{u[\phi]}, & \mbox{if } u \in U \setminus \{0\}. \end{cases}
\end{equation}
\end{inparaenum}
We will call $M$ defined in \eqref{levyMartValuedMeasExam} a \emph{L\'{e}vy martingale-valued measure}.  
\end{exam}

Additionally to the properties of the family of semi-norms $\{q_{r,u}: r \in \R_{+}, \, u \in U \}$ given in Definition \ref{nuclearMartingaleValMeasDualSpace}, we will assume they satisfy the following:
 
\begin{assu} \label{assumpDenseSubsetSemiNormsMartValuedMeasures}
For each $T>0$ there exists a countable subset $D$ of $\Phi$ that is dense in $\Phi_{q_{r,u}}$ for each $r \in [0,T]$, $u \in U$. 
\end{assu}

\begin{prop}\label{conditionsAssumpDenseSubsetSemiNormsMVM} The Assumption \ref{assumpDenseSubsetSemiNormsMartValuedMeasures} is satisfied if either $\Phi$ is separable or if $\Phi$ is barrelled.  
\end{prop}
\begin{prf}
If $\Phi$ is separable then Assumption \ref{assumpDenseSubsetSemiNormsMartValuedMeasures} is an immediate consequence of the fact that $\Phi$ is dense in each $\Phi_{q_{r,u}}$. If $\Phi$ is barrelled, similar arguments to those used in the proof of Theorem 4.2 in \cite{BojdeckiJakubowski:1989} show that Assumption \ref{assumpDenseSubsetSemiNormsMartValuedMeasures} is satisfied. 
\end{prf}

An important consequence of Assumption \ref{assumpDenseSubsetSemiNormsMartValuedMeasures} is given in the following proposition. The proof follows from the same arguments as those used in the proof of Proposition 1.8 in  \cite{BojdeckiJakubowski:1990}.

\begin{prop} \label{compatibilityPredictableMapsAndSeminorms}
Let $\{ q_{r,u} \}$ satisfy Assumption \ref{assumpDenseSubsetSemiNormsMartValuedMeasures}. Let the functions $(r,\omega,u) \mapsto f(r,\omega,u) \in \Phi_{q_{r,u}}$ and $(r,\omega,u) \mapsto g(r,\omega,u) \in \Phi_{q_{r,u}}$ be such that for each $\phi \in \Phi$, the functions $(r,u,\omega) \mapsto q_{r,u}(f(r,\omega,u),\phi)$ and $(r,\omega,u) \mapsto q_{r,u}(g(r,\omega,u),\phi)$ are $\mathcal{P}_{T} \otimes \mathcal{B}(U)/ \mathcal{B}(\R_{+})$-measurable. Then, the function $(r,\omega,u) \mapsto q_{r,u}(f(r,\omega,u),g(r,\omega,u))$ is $\mathcal{P}_{T} \otimes \mathcal{B}(U)/ \mathcal{B}(\R_{+})$-measurable.
\end{prop}

\begin{nota}
Throughout this article and unless otherwise stated, $M$ will denote a nuclear cylindrical martingale valued measure on $\R_{+} \times \mathcal{R}$ and satisfying \eqref{covarianceFunctionalNuclearMartValuedMeasure} for $\mu$, $\lambda$ and $\{ q_{r,u} \}$ as in Definition \ref{nuclearMartingaleValMeasDualSpace}. Also, the family of semi-norms $\{ q_{r,u} \}$ will satisfy Assumption \ref{assumpDenseSubsetSemiNormsMartValuedMeasures}. Furthermore we fix an arbitrary $T>0$.
\end{nota}

\section{The Weak Stochastic Integral} \label{SubsectionWSI}

\subsection{The Weak Stochastic Integral for Integrands with Square Moments}

We start by introducing the space of integrands. 

\begin{defi} \label{integrandsWeakIntegSquareMoments}
Let $\Lambda^{2}_{w}(M;T)$ denote the collection  of families $X=\{X(r,\omega,u): r \in [0,T], \omega \in \Omega, u \in U \}$ of Hilbert space-valued maps satisfying the following: 
\begin{enumerate}
\item $X(r,\omega,u) \in \Phi_{q_{r,u}}$, for all $r \in [0, T]$, $\omega \in \Omega$, $u \in U$, 
\item $X$ is \emph{$q_{r,u}$-predictable}, i.e. for each $\phi \in \Phi$, the mapping $[0,T] \times \Omega \times U \rightarrow \R_{+}$ given by $(r,\omega,u) \mapsto q_{r,u}(X(r,\omega,u), \phi)$ is $\mathcal{P}_{T} \otimes \mathcal{B}(U)/\mathcal{B}(\R_{+})$-measurable.
\item 
\begin{equation} \label{finiteSecondMomentIntegrandsWeakIntg}
\norm{X}_{w,M;T}^{2} \defeq \Exp \int_{0}^{T} \int_{U} q_{r,u}(X(r,u))^{2} \mu(du) \lambda(dr) < \infty.
\end{equation} 
\end{enumerate}
\end{defi}

\begin{rema} \label{normOfIntegrandsIsPredictable}
Note that Proposition \ref{compatibilityPredictableMapsAndSeminorms} guaranties that the map $(r,\omega,u) \mapsto q_{r,u}(X(r,\omega,u))^{2}$ is $\mathcal{P}_{T} \otimes \mathcal{B}(U)$-measurable and hence the integral in \eqref{finiteSecondMomentIntegrandsWeakIntg} is well defined. 
\end{rema}

When there is no necessity to give emphasis to the dependence of the space $\Lambda^{2}_{w}(M;T)$ with respect to $M$, we will denote $\Lambda^{2}_{w}(M;T)$ and $\norm{\cdot}_{w,M;T}$ by $\Lambda^{2}_{w}(T)$ and $\norm{\cdot}_{w,T}$ respectively. We will keep using the shorter notation for the remainder of this section. 

With some minor changes, the proof of the following proposition can be carried out following similar arguments to those in the proof of Proposition 2.4 in \cite{BojdeckiJakubowski:1990}. 

\begin{prop} \label{spaceIntegrandsWeakIntegralIsHilbertSpace} $ \Lambda^{2}_{w}(T)$ is a Hilbert space when equipped with the inner product $\inner{\cdot}{\cdot}_{w,T}$ corresponding to the Hilbertian norm $\norm{X}_{w,M;T}$.
\end{prop}

Now, we define a class of simple families of random variables contained in $ \Lambda^{2}_{w}(T)$. 

\begin{defi} \label{classSimpleHilbertValuedRV}
Let $\mathcal{S}_{w}(T)$ be the collection of all the families $X=\{ X(r,\omega,u): r \in [0,T],  \omega \in \Omega, u \in U\}$ of Hilbert space valued maps of the form:
\begin{equation} \label{equationSimpleIntegrandsWeakIntegral}
X(r,\omega,u)= \sum_{i=1}^{n} \sum_{j=1}^{m}  \ind{]s_{j}, t_{j}]}{r} \ind{F_{j}}{\omega} \ind{A_{i}}{u} i_{q_{r,u}} \phi_{i,j}, 
\end{equation}
for all $r \in [0,T]$, $\omega \in \Omega$, $u \in U$, where $m$, $n \in \N$, and for $i=1, \dots, n$, $j=1, \dots, m$, $0\leq s_{j}<t_{j} \leq T$, $F_{j} \in \mathcal{F}_{s_{j}}$, $A_{i} \in \mathcal{R}$ and $\phi_{i,j} \in \Phi$.
\end{defi}

It is easy to check that $\mathcal{S}_{w}(T)$ is a subspace of $\Lambda^{2}_{w}(T)$. Moreover, we have the following: 

\begin{prop} \label{simpleIntegrandsDenseInSquareIntegIntegrandsWeakIntegral}
$S_{w}(T)$ is dense in $\Lambda_{w}^{2}(T)$. 
\end{prop}
\begin{prf}
Let $C_{w}(T)$ be the collection of all families of Hilbert space valued maps $Y=\{ Y(r,\omega,u): r \in [0,T],  \omega \in \Omega, u \in U\}$ taking the simple form
\begin{equation} \label{simpleFamiliesProofDenseSupspaceIntegrands}
Y(r,\omega,u)= \ind{]s,t]}{r} \ind{F}{\omega} \ind{A}{u} i_{q_{r,u}} \phi, \quad \forall \, t \in [0,T], \, \omega \in \Omega, \, u \in U,
\end{equation}  
where $0 \leq s < t \leq T$, $F \in \mathcal{F}_{s}$, $A \in \mathcal{R}$ and $\phi \in \Phi$. 

Is clear from \eqref{equationSimpleIntegrandsWeakIntegral} and \eqref{simpleFamiliesProofDenseSupspaceIntegrands} that $C_{w}(T)$ spans $S_{w}(T)$. Our objective is then to prove that the only element of $\Lambda^{2}_{w}(T)$ that is orthogonal to $C_{w}(T)$ is the zero family (to be precise, its equivalence class). This will imply that $S_{w}(T)$ is dense in $\Lambda^{2}_{w}(T)$. 

To do this, let $X \in \Lambda^{2}_{w}(T)$. If $Y \in S_{w}(T)$ is of the form \eqref{simpleFamiliesProofDenseSupspaceIntegrands}, then we have that 
\begin{equation} \label{innerProductProofDenseSubpaceIntegrands}
\inner{X}{Y}_{w,T}= \int_{F} \int_{s}^{t} \int_{A} q_{r,u}(X(r,\omega,u),i_{q_{r,u}} \phi) \mu(du) \lambda(dr) \Prob(d \omega). 
\end{equation}
 Assume that $X \in C_{w}(T)^{\perp}$, where $C_{w}(T)^{\perp}$ denotes the orthogonal complement of $C_{w}(T)$ in $\Lambda^{2}_{w}(T)$. Hence, it follows from \eqref{innerProductProofDenseSubpaceIntegrands} that $X$ satisfies:
\begin{equation} \label{orthogonalityConditionProofDenseSubspaceIntegrands}
\int_{F} \int_{s}^{t} \int_{A} q_{r,u}(X(r,\omega,u),i_{q_{r,u}} \phi) \mu(du) \lambda(dr) \Prob(d \omega)=0, 
\end{equation}
 for all $0 \leq s < t \leq T$, $F \in \mathcal{F}_{s}$, $A \in \mathcal{R}$ and $\phi \in \Phi$. 
 
 Moreover, as $\mathcal{P}_{T} \otimes \mathcal{B}(U)$ is generated by the family of all subsets of $[0,T] \times \Omega \times U$ of the form $G=]s,t] \times F \times A$, where $0 \leq s < t \leq T$, $F \in \mathcal{F}_{s}$, $A \in \mathcal{R}$; then \eqref{orthogonalityConditionProofDenseSubspaceIntegrands} and the Fubini theorem implies that $q_{r,u}(X(r,\omega,u),i_{q_{r,u}} \phi)=0$ $\lambda \otimes \Prob \otimes \mu$-a.e., for all $\phi \in \Phi$. Furthermore, as for each $(r,u) \in [0,T] \times U$, $i_{q_{r,u}}(\Phi)$ is dense in $\Phi_{q_{r,u}}$, then it follows that $X(r,\omega,u)=0$ $\lambda \otimes \Prob \otimes \mu$-a.e. Thus,   
$C_{w}(T)^{\perp}=\{0\}$ and hence $S_{w}(T)$ is dense in $\Lambda^{2}_{w}(T)$.   
\end{prf}

Now we define the weak stochastic integral for the elements of $\mathcal{S}_{w}(T)$. Let $X \in \mathcal{S}_{w}(T)$ be of the form \eqref{equationSimpleIntegrandsWeakIntegral}. We can always assume (taking a smaller partition if needed) that: 
\begin{equation} \label{separationPartitionIntervalsSimpleForm}
\mbox{for } k \neq j, \quad  ]s_{k},t_{k}]  \, \cap  \,  ]s_{j},t_{j}] \neq \emptyset \quad \Rightarrow  \quad ]s_{k},t_{k}]= \, ]s_{j},t_{j}] \mbox{ and } F_{k} \cap F_{j}= \emptyset.
\end{equation}
Then, for such an $X$ we define   
\begin{equation} \label{weakStochasticIntegralSimpleIntegrands}
I^{w}_{T}(X)= \sum_{i=1}^{n} \sum_{j=1}^{m} \mathbbm{1}_{F_{j}} M((s_{j}, t_{j}],A_{i})(\phi_{i,j}).
\end{equation}

It is easy to see from the finite-additivity of $M$ on $\mathcal{R}$ and the linearity on $\Phi$ of $M(t,A)$ for any $t \geq 0$, $A \in \mathcal{R}$, that $I^{w}_{T}(X)$ is independent (up to modifications) of the representation of $X \in \mathcal{S}_{w}(T)$ (i.e. of the expression of $X$ as in \eqref{equationSimpleIntegrandsWeakIntegral}). Furthermore, we have the following:

\begin{theo} \label{theoItoIsometrySimpleIntegrandsWeakIntegral}
For every $X \in \mathcal{S}_{w}(T)$,   
\begin{equation}\label{itoIsometrySimpleIntegrandsWeakIntegral}
\Exp \left( I^{w}_{T}(X) \right)=0, \quad \Exp \left( \abs{ I^{w}_{T}(X) }^{2} \right)= \norm{X}^{2}_{w,T}. 
\end{equation}
Moreover, the map $I^{w}_{T}: \mathcal{S}_{w}(T) \rightarrow L^{2} \ProbSpace$, $X \mapsto I^{w}_{T}(X)$, is a linear isometry.  
\end{theo}
\begin{prf}
Let $X \in \mathcal{S}_{w}(T)$ be of the form \eqref{equationSimpleIntegrandsWeakIntegral} and satisfying \eqref{separationPartitionIntervalsSimpleForm}. From the definition of $I^{w}_{T}(X)$ in \eqref{weakStochasticIntegralSimpleIntegrands}, the independent increments of $M$ and Definition \ref{martingValuedMeasureOnDualSpace}(3), we have 
$$\Exp \left( I^{w}_{T}(X) \right) = \sum_{i=1}^{n} \sum_{j=1}^{m} \Prob (F_{j}) \Exp \left(  M((s_{j}, t_{j}],A_{i})(\phi_{i,j}) \right)=0. $$ 

Now, note that from the orthogonality of $M$ on the ring $\mathcal{A}$ (Definition \ref{martingValuedMeasureOnDualSpace}(4)) we have that 
$$ \Exp \left( M((s_{j} , t_{j} ],A_{i})(\phi_{i,j}) \cdot M((s_{l} , t_{l} ],A_{k})(\phi_{k,l})  \right) = 0. $$
for each $i,k=1, \dots, n$, $j,l=1, \dots, m$, $i \neq k$, $j \neq l$.  Then, 
\begin{eqnarray} 
\Exp \left( \abs{ I^{w}_{T}(X) }^{2} \right)
& = & \sum_{i,k=1}^{n} \sum_{j,l=1}^{m}  \Exp \left( \mathbbm{1}_{F_{j}} M((s_{j}, t_{j}],A_{i})(\phi_{i,j}) \mathbbm{1}_{F_{l}} M((s_{l}, t_{l}],A_{k})(\phi_{k,l})  \right) \hspace{25pt} \label{secondMomentWeakIntegralSimpleFamilies} \\
& = & \sum_{i=1}^{n} \sum_{j=1}^{m} \Prob (F_{j}) \Exp \left( \abs{ M((s_{j}, t_{j}],A_{i})(\phi_{i,j}) }^{2} \right)  \nonumber\\
& = & \sum_{i=1}^{n} \sum_{j=1}^{m} \Prob (F_{j}) \int_{s_{j}}^{t_{j}} \int_{A_{i}} q_{r,u}(\phi_{i,j} )^{2} \mu (du) \lambda (dr) \nonumber \\
& = & \Exp \int_{0}^{T} \int_{U} q_{r,u}(X(r,u))^{2} \mu(du) \lambda(dr) \nonumber \\
& = & \norm{X}_{w,T}^{2} \nonumber.  
\end{eqnarray}
The linearity of the map $I^{w}_{T}: \mathcal{S}_{w}(T) \rightarrow L^{2} \ProbSpace$ follows from the properties (2) and (3) of $M$ in Definition \ref{martingValuedMeasureOnDualSpace}. Finally, that $I^{w}_{T}$ is an isometry is a consequence of \eqref{itoIsometrySimpleIntegrandsWeakIntegral}.  
\end{prf}

Now, from Proposition \ref{simpleIntegrandsDenseInSquareIntegIntegrandsWeakIntegral} and Theorem \ref{theoItoIsometrySimpleIntegrandsWeakIntegral}, the map $I^{w}_{T}$ extends to a linear isometry from $\Lambda^{2}_{w}(T)$ into $L^{2} \ProbSpace$, that we still denote by $I^{w}_{T}$. Moreover, from \eqref{itoIsometrySimpleIntegrandsWeakIntegral} we have the following \emph{It\^{o} isometry}
\begin{equation}\label{itoIsometryWeakIntegral}
\Exp \left( \abs{ I^{w}_{T}(X) }^{2} \right)= \norm{X}^{2}_{w,T}, \quad \forall \, X \in \Lambda^{2}_{w}(T). 
\end{equation}
For $0 \leq t \leq T$, $X \in \Lambda^{2}_{w}(T)$, it is clear that $\mathbbm{1}_{[0,t]}X \in \Lambda^{2}_{w}(T)$ and hence we can define a real-valued process $I^{w}(X)=\{ I^{w}_{t}(X) \}_{t \geq 0}$ by means of the prescription 
\begin{equation} \label{defiWeakIntegralAsProcess}
I^{w}_{t}(X) \defeq I^{w}_{T}(\mathbbm{1}_{[0,t]} X), \quad \forall \, t \in [0,T]. 
\end{equation}
The process $I^{w}(X)$ will be called the \emph{weak stochastic integral} of $X$, and sometimes we denote it by $\left\{ \int^{t}_{0} \int_{U} X (r,u) M (dr, du): t \in [0,T] \right\}$.  Some of the properties of the weak stochastic integral are given in the following theorem. 

\begin{theo} \label{propertiesWeakIntegralSquareIntegIntegrands}
For each $X \in \Lambda^{2}_{w}(T)$, $I^{w}(X)=\{ I^{w}_{t}(X) \}_{t \in [0,T]}$ is a real-valued zero-mean, square integrable, c\`{a}dl\`{a}g martingale with second moments given by
\begin{equation} \label{secondMomentWeakStochasticIntegral}
\Exp \left( \abs{I^{w}_{t}(X)}^{2} \right) = \Exp \int^{t}_{0} \int_{U} q_{r,u}(X(r,u))^{2} \mu(du) \lambda(dr), \quad \forall \, t \in [0,T]. 
\end{equation}
Moreover, $I^{w}(X)$ is mean square continuous and it has a predictable version. Furthermore, the mapping $I^{w}: \Lambda^{2}_{w}(T) \rightarrow \mathcal{M}^{2}_{T}(\R)$ given by $X \mapsto I^{w}(X)=\{ I^{w}_{t}(X) \}_{t \in [0,T]}$ is linear and continuous. 
\end{theo}	
\begin{prf}
Let $X \in \mathcal{S}_{w}(T)$ be of the form \eqref{equationSimpleIntegrandsWeakIntegral} and satisfying \eqref{separationPartitionIntervalsSimpleForm}. From \eqref{weakStochasticIntegralSimpleIntegrands},  \eqref{defiWeakIntegralAsProcess}, the independent increments of $M$, Definition \ref{martingValuedMeasureOnDualSpace}(3) and \eqref{itoIsometrySimpleIntegrandsWeakIntegral} it follows that  $I^{w}(X) \in \mathcal{M}^{2}_{T}(\R)$. Moreover, similar calculations to those used in \eqref{secondMomentWeakIntegralSimpleFamilies} shows that $I^{w}(X)$ satisfies \eqref{secondMomentWeakStochasticIntegral}.  

Now, if $X \in \Lambda^{2}_{w}(T)$ there exist a sequence $\{ X_{k} \}_{k \in \N} \subseteq \mathcal{S}_{w}(T)$ that converges to $X$ in $\Lambda^{2}_{w}(T)$ (Proposition \ref{simpleIntegrandsDenseInSquareIntegIntegrandsWeakIntegral}). Then, from the linearity of the weak integral,  \eqref{itoIsometryWeakIntegral}, Doob's inequality and the completeness of $\mathcal{M}^{2}_{T}(\R)$, we have that $\{ I^{w}(X_{k}) \}_{k \in \N}$ converges to $I^{w}(X)$ in $\mathcal{M}^{2}_{T}(\R)$, and hence $I^{w}(X) \in \mathcal{M}^{2}_{T}(\R)$. Moreover, because each $X_{k}$ satisfies \eqref{secondMomentWeakStochasticIntegral}, the fact that  $\{ I^{w}_{t}(X_{k}) \}_{k \in \N}$ converges to $I^{w}_{t}(X)$ in $L^{2}\ProbSpace$ for each $t \in [0,T]$ implies that $I^{w}_{t}(X)$ satisfies \eqref{secondMomentWeakStochasticIntegral}.  Then for every $X \in \Lambda^{2}_{w}(T)$ it follows from Doob's inequality and \eqref{itoIsometryWeakIntegral} that  
\begin{equation} \label{weakIntegralMapIsContinuousInequ}
\norm{I^{w}(X)}^{2}_{\mathcal{M}^{2}_{T}(\R)} = \Exp \left( \sup_{t\in [0,T]} \abs{I^{w}_{t}(X)}^{2} \right) \leq 4 T \, \Exp \left( \abs{I^{w}_{T}(X)}^{2} \right)= 4T \, \norm{X}^{2}_{w,T}, 
\end{equation} 
and so the linear map $I^{w}: \Lambda^{2}_{w}(T) \rightarrow \mathcal{M}^{2}_{T}(\R)$, $X \mapsto I^{w}(X)=\{ I^{w}_{t}(X) \}_{t \in [0,T]}$, is continuous.

To prove the mean square continuity property, note that if $X \in \Lambda^{2}_{w}(T)$, then it follows from \eqref{secondMomentWeakStochasticIntegral} that for any $0 \leq s \leq t \leq T$ we have:
$$ \Exp \left( \abs{I^{w}_{t}(X)-I^{w}_{s}(X)}^{2} \right) = \Exp \int^{t}_{s} \int_{U} q_{r,u}(X(r,u))^{2} \mu(du) \lambda(dr) \leq \norm{X}_{w,T}^{2},$$
and hence from an application of the dominated convergence theorem we have 
$$ \Exp \left( \abs{I^{w}_{t}(X)-I^{w}_{s}(X)}^{2} \right) \rightarrow 0 \quad \mbox{ as } s \rightarrow t, \mbox{ or } t \rightarrow s. $$
Thus, $I^{w}(X)$ is mean square continuous. Now, as $I^{w}(X)$ is $\{ \mathcal{F}_{t}\}$-adapted and stochastically continuous it has a predictable version (see \cite{PeszatZabczyk}, Proposition 3.21, p.27). 
\end{prf}

\begin{defi}
We call the map $I^{w}$ defined in Theorem \ref{propertiesWeakIntegralSquareIntegIntegrands} the \emph{weak stochastic integral} mapping.
\end{defi}

\begin{prop}\label{propWeakStochIntegContIfCylinMartIsCont}
If for each $A \in \mathcal{R}$ and $\phi \in \Phi$, the real-valued process $(M(t,A)(\phi): t \geq 0)$ is continuous, then for each $X \in \Lambda^{2}_{w}(T)$ the stochastic integral $I^{w}(X)$ is a continuous process.  
\end{prop}
\begin{prf}
The result follows clearly from the definition of $I^{w}(X)$ for $X \in \mathcal{S}_{w}(T)$ and this can be extended by the denseness of $\mathcal{S}_{w}(T)$ in $\Lambda^{2}_{w}(T)$ to any $X \in \Lambda^{2}_{w}(T)$. 
\end{prf}

\subsection{Properties of the Weak Stochastic Integral}\label{subSectionPWSI}

In this section we obtain some properties of the weak stochastic integral. 
The following result can be proven using similar arguments to those in the proof of Lemma 4.9 in \cite{DaPratoZabczyk}, p.94-5.

\begin{prop} \label{propStoppedIntegralWeakCase}
Let $X \in \Lambda^{2}_{w}(T)$ and $\sigma$ be an $\{\mathcal{F}_{t}\}$-stopping time such that $\Prob (\sigma \leq T)=1$. Then, $\Prob$-a.e. 
\begin{equation} \label{stoppedIntegralWeakCase}
I^{w}_{t}(\mathbbm{1}_{[0,\sigma]} X )= I^{w}_{t \wedge \sigma}(X), \quad \forall \, t \in [0,T].
\end{equation}
\end{prop}

\begin{prop} \label{propWeakIntegralInSubintervalAndRandomSubset}
Let $0 \leq s_{0} < t_{0} \leq T$ and $F_{0} \in \mathcal{F}_{s_{0}}$. Then, for every $X \in \Lambda^{2}_{w}(T)$, $\Prob$-a.e. we have   
\begin{equation} \label{weakIntegralInSubintervalAndRandomSubset}
I^{w}_{t}(\mathbbm{1}_{]s_{0},t_{0}]\times F_{0}} X)= \mathbbm{1}_{F_{0}} \left( I^{w}_{t \wedge t_{0}}(X)-I^{w}_{t \wedge s_{0}}(X) \right), \quad \forall \, t \in [0,T]. 
\end{equation}
\end{prop}
\begin{prf}
Let $X$ be of the simple form:
\begin{equation} \label{simpleFormWeakIntegrandsInSubintervalsAndRandomSubsets}
X(r,\omega,u)=\ind{]s_{1},t_{1}]}{r} \ind{F_{1}}{\omega} \ind{A}{u} i_{q_{r,u}} \phi, \quad \forall \, r \in [0,T], \, \omega \in \Omega, \, u \in U, 
\end{equation} 
where $0 \leq s_{1} < t_{1} \leq T$, $\mathcal{F}_{1} \in \mathcal{F}_{s_{1}}$, $A \in \mathcal{R}$ and $\phi \in \Phi$. Then, for such simple $X$ one can easily check that $\mathbbm{1}_{]s_{0},t_{0}] \times F_{0}}X$ belongs to $S_{w}(T)$ and that \eqref{weakIntegralInSubintervalAndRandomSubset} is satisfied.  The linearity of the integral implies that \eqref{weakIntegralInSubintervalAndRandomSubset} is valid for any $X \in S_{w}(T)$. Moreover, by the denseness of $S_{w}(T)$ in $\Lambda^{2}_{w}(T)$ and the continuity of the weak stochastic integral mapping $I^{w}$ (Theorem \ref{propertiesWeakIntegralSquareIntegIntegrands}), it follows that \eqref{weakIntegralInSubintervalAndRandomSubset} is satisfied for every $X \in \Lambda^{2}_{w}(T)$.    
\end{prf}
 
We finalize this section with the following result that will be of great importance our study in Section \ref{subSectionALNC} of SPDEs driven by general L\'{e}vy noise. We omit its proof as it can be proven by following similar artuments to those in the proof of  Proposition 3.7 in \cite{BojdeckiJakubowski:1990}.

\begin{prop} \label{propDecompWeakIntegralSumIndepMartValMeasu}
Let $N_{1}$, $N_{2}$ be two nuclear cylindrical martingale-valued measures with independent increments on $\R_{+} \times \mathcal{R}$, each with covariance structure as in \eqref{covarianceFunctionalNuclearMartValuedMeasure} determined by the family $\{ p_{r,u}^{j} \}_{r,u}$ of continuous Hilbertian semi-norms on $\Phi$ and measures $\lambda_{j}=\lambda$, $\mu_{j}=\mu$, for $j = 1,2$; all of them satisfying the conditions given in Definition \ref{nuclearMartingaleValMeasDualSpace}. Assume furthermore that for all $A, B \in \mathcal{R}$ and all $\phi$, $\varphi \in \Phi$, the real valued processes $\{ N_{1}(t,A)(\phi) \}_{t \geq 0}$ and $\{ N_{2}(t,B)(\varphi) \}_{t \geq 0}$ are independent. Let $M= (M(t,A): r \geq 0, A \in \mathcal{R})$ be given by the prescription:
\begin{equation*} \label{defiSumIndepMartValuedMeasures}
M(t,A) \defeq N_{1}(t,A)+N_{2}(t,A), \quad \forall \, t \in \R_{+}, \, A \in \mathcal{R}. 
\end{equation*} 
Then, $M$ is also a nuclear cylindrical martingale-valued measure with independent increments on $\R_{+} \times \mathcal{R}$, with covariance structure determined by $\lambda$, $\mu$ and the family of continuous Hilbertian semi-norms $\{ q_{r,u} \}_{r,u}$ satisfying $q_{r,u}(\phi)^{2}=p^{1}_{r,u}(\phi)^{2}+ p^{2}_{r,u}(\phi)^{2}$ for all $r \geq 0$, $u \in U$, $\phi \in \Phi$. Moreover, if $X \in \Lambda^{2}_{w}(M;T)$ we have: 
\begin{enumerate}
\item For each $j=1,2$, $\{ i_{p_{r,u}^{j},q_{r,u}} X(r,\omega,u): r \in [0,T], \omega \in \Omega, u \in U \}  \in \Lambda^{2}_{w}(N_{j};T)$. 
\item $\Prob$-a.e., for all $t \in [0,T]$ we have, 
\begin{eqnarray*}
\int_{0}^{t}\int_{U} X(r,u) M(dr,du) & = & \int_{0}^{t}\int_{U} i_{p_{r,u}^{1},q_{r,u}} X(r,u) N_{1}(dr,du) \label{decompWeakIntegralSumIndpMartValMeasu} \\
& {} & + \int_{0}^{t}\int_{U} i_{p_{r,u}^{2},q_{r,u}} X(r,u) N_{2}(dr,du) \nonumber.  
\end{eqnarray*}
\end{enumerate}
\end{prop}

\subsection{An Extension of The Class of Integrands}\label{subSectionAECIWI}

The final step in our construction of the weak stochastic integral is to extend it to the following families of random variables with only almost sure second moments:  

\begin{defi} \label{integrandsWeakIntegAlmostSureSquareMoments}
Let $\Lambda^{2,loc}_{w}(M;T)$ denote the collection of families $X=\{X(r,\omega,u): r \in [0,T], \omega \in \Omega, u \in U \}$ of Hilbert space-valued maps satisfying the following conditions:
\begin{enumerate}
\item $X(r,\omega,u) \in \Phi_{q_{r,u}}$, for all $r \in [0, T]$, $\omega \in \Omega$, $u \in U$, 
\item $X$ is \emph{$q_{r,u}$-predictable}, i.e. for each $\phi \in \Phi$, the mapping $[0,T] \times \Omega \times U \rightarrow \R_{+}$ given by $(r,\omega,u) \mapsto q_{r,u}(X(r,\omega,u), \phi)$ is $\mathcal{P}_{T} \otimes \mathcal{B}(U)$-measurable.
\item 
\begin{equation} \label{almostSureSecondMomentIntegrandsWeakIntg}
\Prob \left( \omega \in \Omega: \int_{0}^{T} \int_{U} q_{r,u}(X(r,\omega,u))^{2} \mu(du) \lambda(dr) < \infty \right)=1.
\end{equation} 
\end{enumerate}
\end{defi}

As before, we will sometimes denote $\Lambda^{2,loc}_{w}(M;T)$ by $\Lambda^{2,loc}_{w}(T)$ when is clear to which cylindrical martingale-valued measure $M$ we are referring. 

We equip the linear space $\Lambda^{2,loc}_{w}(T)$ with the vector topology $\mathcal{T}_{2,loc}^{M}$ generated by the local base of neighbourhoods of zero $\{ \Gamma_{\epsilon, \delta}: \epsilon > 0, \delta >0 \}$, where $\Gamma_{\epsilon, \delta}$ is given by 
$$ \Gamma_{\epsilon, \delta} = \left\{ X \in \Lambda^{2,loc}_{w}(T): \Prob \left( \omega \in \Omega: \int_{0}^{T} \int_{U} q_{r,u}(X(r,\omega,u))^{2} \mu(du) \lambda(dr) > \epsilon \right) \leq \delta \right\}. $$

\begin{prop}\label{extendedClassWeakIntegrandsInMetrizable}
$(\Lambda^{2,loc}_{w}(T),\mathcal{T}_{2,loc}^{M})$ is a complete, metrizable topological vector space.  
\end{prop}
\begin{prf}
On $\Lambda^{2,loc}_{w}(T)$, we introduce the translation invariant metric $d_{\Lambda}$ given by 
$$ d_{\Lambda}(X,Y)= \Exp \left[ G \left( \int_{0}^{T} \int_{U} q_{r,u}(X(r,u)-Y(r,u))^{2} \mu(du) \lambda(dr)  \right) \right], \quad \forall \, X, Y \in \Lambda^{2,loc}_{w}(T),$$
where $G:\R \rightarrow \R$ is given by $G(x)=\frac{x}{1+x}$, for each $x \in \R$. It is easy to show that $d_{\Lambda}$ generates a vector topology equivalent to $\mathcal{T}^{M}_{2,loc}$. Therefore, $(\Lambda^{2,loc}_{w}(T),\mathcal{T}_{2,loc}^{M})$ is a metrizable topological vector space. The proof of the completeness can be carried out by following similar arguments to those used in the proof of Proposition 2.4 in \cite{BojdeckiJakubowski:1990}. 
\end{prf}

\begin{rema} \label{spaceExtendedWeakIntegrandsNonLocallyConvex}
If $\Prob$ is an atomless measure (see \cite{BogachevMT}, Definition 1.12.7, p.55) we can show using similar arguments to those in used in Remarque 1 of Badrikian \cite{Badrikian} p.2, that every convex neighbourhood of zero in $\Lambda^{2,loc}_{w}(T)$ is identical to it, and hence $\Lambda^{2,loc}_{w}(T)$ is not locally convex. 
\end{rema}

The extension of the weak stochastic integral to integrands in $\Lambda^{2,loc}_{w}(T)$ will utilise the following result. The proof follows from standard arguments (see Section 4.2 of \cite{DaPratoZabczyk}) by using the properties of the weak stochastic integral obtained in Theorem \ref{propertiesWeakIntegralSquareIntegIntegrands} and Proposition \ref{propStoppedIntegralWeakCase}.  

\begin{theo} \label{existenceExtendedStochasticIntegralWeakCase}
Let $X \in \Lambda^{2,loc}_{w}(T)$. Then, 
\begin{enumerate}
\item There exists an increasing sequence $\{ \tau_{n} \}_{n \in \N}$ of $\{\mathcal{F}_{t}\}$-stopping times satisfying $\lim_{n \rightarrow \infty} \tau_{n} = T$ $\Prob$-a.e. and such that for each $n \in \N$, $\mathbbm{1}_{[0,\tau_{n}]} X \in \Lambda^{2}_{w}(T)$. 
\item There exists a unique  $\hat{I}^{w}(X)=\{ \hat{I}^{w}_{t}(X) \}_{t \in [0,T]} \in \mathcal{M}^{2,loc}_{T}(\R)$ such that for any sequence of $\{\mathcal{F}_{t}\}$-stopping times $\{ \sigma_{n} \}_{n \in \N}$ satisfying $\lim_{n \rightarrow \infty} \sigma_{n} = T$ $\Prob$-a.e. and $\mathbbm{1}_{[0,\sigma_{n}]} X \in \Lambda^{2}_{w}(T)$ for each $n \in \N$, the process $\hat{I}^{w}(X)$ satisfies: 
\begin{equation} \label{compatibilityCondExtendedStochasticIntegralWeakCase}
\hat{I}^{w}_{t \wedge \sigma_{n} }(X) = I^{w}_{t}(\mathbbm{1}_{[0,\sigma_{n}]} X), \quad \forall \, t \in [0,T], \, n \in \N,
\end{equation}
where the process on the right-hand side of \eqref{compatibilityCondExtendedStochasticIntegralWeakCase} is the weak stochastic integral of $\mathbbm{1}_{[0,\sigma_{n}]} X$.  
\end{enumerate}
\end{theo}

\begin{defi} For every $X \in \Lambda^{2,loc}_{w}(T)$, we will call the process $\hat{I}^{w}(X)$ given in Theorem \ref{existenceExtendedStochasticIntegralWeakCase} the \emph{weak stochastic integral} of $X$ and denote it by $\left\{ \int^{t}_{0} \int_{U} X (r,u) M (dr, du): t \in [0,T] \right\}$.  
\end{defi}

The property \eqref{compatibilityCondExtendedStochasticIntegralWeakCase} allow us to ``transfer'' the properties of the weak stochastic integral for integrands in $\Lambda^{2}_{w}(T)$ (see Section \ref{subSectionPWSI}) to those in $\Lambda^{2,loc}_{w}(T)$. We summarize this in the following result:

\begin{prop} \label{propPropertWeakStochIntegExtenToIntegLocalSecoMomen}
Let $X \in \Lambda^{2,loc}_{w}(T)$. Then, all the assertions in Propositions \ref{propStoppedIntegralWeakCase}, \ref{propWeakIntegralInSubintervalAndRandomSubset} and \ref{propDecompWeakIntegralSumIndepMartValMeasu} are valid for the weak stochastic integral $\hat{I}^{w}(X)$ of $X$. 
\end{prop} 

As was shown for the weak stochastic integral for integrands in $\Lambda^{2}_{w}(T)$, we can also prove that the \emph{extended weak stochastic integral} map $\hat{I}^{w}: \Lambda^{2,loc}_{w}(T) \rightarrow \mathcal{M}^{2,loc}_{T}(\R)$, $X \mapsto \hat{I}^{w}(X)$, is linear and continuous.

The linearity of the map $\hat{I}^{w}$ follows from \eqref{compatibilityCondExtendedStochasticIntegralWeakCase} and the corresponding linearity of the map $I^{w}: \Lambda^{2}_{w}(T) \rightarrow \mathcal{M}^{2}_{T}(\R)$. The continuity follows from the following estimate that can by proved by similar arguments to those used in the proof of Proposition 4.16 in  \cite{DaPratoZabczyk}, p.104-5. 

\begin{prop} \label{estimateContinuityWeakIntegralMap}
Assume $X \in \Lambda^{2,loc}_{w}(T)$. Then, for arbitrary $a >0$, $b >0$, 
$$ \Prob \left( \sup_{t \in [0,T]} \abs{\hat{I}^{w}_{t}(X)} > a \right) \leq \frac{b}{a^{2}} + \Prob \left( \int_{0}^{T} \int_{U} q_{r,u}(X(r,u))^{2} \mu(du) \lambda(dr) > b \right). $$ 
\end{prop}   

\begin{prop} \label{extendedWeakIntegralMapIsLinearContinuous} The extended weak stochastic integral mapping $\hat{I}^{w}: \Lambda^{2,loc}_{w}(T) \rightarrow \mathcal{M}^{2,loc}_{T}(\R)$ is linear and continuous. 
\end{prop}
\begin{prf}
As the map $\hat{I}^{w}$ is linear, we need only to show its continuity. Let $\{ X_{n} \}_{n \in \N}$ be a sequence converging to $X$ in $\Lambda^{2,loc}_{w}(T)$. As both $\Lambda^{2,loc}_{w}(T)$ and $ \mathcal{M}^{2,loc}_{T}(\R)$ are metrizable, it is sufficient to prove that $\{ \hat{I}^{w}(X_{n}) \}_{n \in \N}$ converges to $\hat{I}^{w}(X)$ in $ \mathcal{M}^{2,loc}_{T}(\R)$. 

Let $\epsilon, \delta >0$. As $\{ X_{n} \}_{n \in \N}$ converges to $X$ in $\Lambda^{2,loc}_{w}(T)$, then there exists some $N_{\epsilon, \delta} \in \N$ such that for all $n \geq N_{\epsilon, \delta}$,
\begin{equation} \label{convergSequenceProofContinuityWeakIntegralExtendedClass}
\Prob \left( \int_{0}^{T} \int_{U} q_{r,u}(X(r,u)-X_{n}(r,u))^{2} \mu(du) \lambda(dr) > \frac{\delta \epsilon^{2}}{2} \right) \leq \frac{\delta}{2}. 
\end{equation}  
By linearity of the integral map, Proposition \ref{estimateContinuityWeakIntegralMap} and \eqref{convergSequenceProofContinuityWeakIntegralExtendedClass}, for all $n \geq N_{\epsilon, \delta}$, we have 
\begin{multline*}
\Prob \left( \sup_{t \in [0,T]} \abs{\hat{I}^{w}_{t}(X)-\hat{I}^{w}_{t}(X_{n})} > \epsilon \right) \\ 
\leq  \frac{\delta}{2} + \Prob \left( \int_{0}^{T} \int_{U} q_{r,u}(X(r,u)-X_{n}(r,u))^{2} \mu(du) \lambda(dr) > \frac{\delta \epsilon^{2}}{2} \right) \leq \delta.
\end{multline*}
And hence (see \eqref{defiLocalBaseNeigZeroSpaceLocalSquareIntegRealMarting}) $\{ \hat{I}^{w}(X_{n}) \}_{n \in \N}$ converges to $\hat{I}^{w}(X)$ in $ \mathcal{M}^{2,loc}_{T}(\R)$. 
\end{prf}

\subsection{The Stochastic Fubini Theorem}\label{subSectionSFTWI}

The final topic in our study of the properties of the weak stochastic integral is the stochastic Fubini theorem that we introduce and prove below.  We start by describing the class of integrands for which the theorem is valid. 

\begin{defi} \label{defiClassIntegrandsStochFubiniTheo} Let $(E, \mathcal{E}, \varrho)$ be a $\sigma$-finite measure space. We denote by $\Xi^{1,2}_{\omega}(T,E)$\index{.x X12TE@$\Xi^{1,2}_{\omega}(T,E)$} the linear space of all (equivalence classes of) families $X=\{X(r,\omega,u,e): r \in [0,T], \omega \in \Omega, u \in U, e \in E\}$
of Hilbert space-valued maps satisfying the following conditions:
\begin{enumerate}
\item $X(r, \omega, u, e) \in \Phi_{q_{r,u}}$, $\forall r \in [0,T]$, $\omega \in \Omega$, $u \in U$, $e \in E$.
\item The map $[0,T] \times \Omega \times U \times E \rightarrow \R_{+}$ given by $(r, \omega, u, e)\mapsto q_{r,u} (X(r, \omega, u, e), \phi)$ is $\mathcal{P}_{T} \otimes \mathcal{B}(U) \otimes \mathcal{E}$-measurable, for every $\phi \in \Phi$.
\item 
$$ \abs{\norm{X}}_{w,T,E} \defeq \int_{E} \norm{X(\cdot,\cdot,\cdot,e)}_{w,T} \varrho (de) < \infty. $$

\end{enumerate}
\end{defi}

It is easy to see that $\Xi^{1,2}_{\omega}(T,E)$ is a Banach space when equipped with the norm $\abs{\norm{\cdot}}_{w,T,E}$. 
We will denote by $\Xi^{2,2}_{\omega}(T,E)$ the subspace of $\Xi^{1,2}_{\omega}(T,E)$ comprising all $X=\{X(r,\omega,u,e): r \in [0,T], \omega \in \Omega, u \in U, e \in E\}$ satisfying:
$$  \abs{\norm{X}}^{2}_{w,2,T,E} \defeq \int_{E} \norm{X(\cdot,\cdot,\cdot,e)}^{2}_{w,T} \varrho (de) < \infty.$$
The space $\Xi^{2,2}_{\omega}(T,E)$ is a Hilbert space when equipped with the Hilbertian norm $ \abs{\norm{\cdot}}_{w,2,T,E}$. 

\begin{rema}
Properties (1)-(3) of Definition \ref{defiClassIntegrandsStochFubiniTheo} together with Fubini's Theorem imply that the map $e\mapsto \norm{X(\cdot,\cdot,\cdot,e)}^{2}_{w,T}$ is $\mathcal{E}$-measurable. Hence, the map $e\mapsto X(\cdot,\cdot,\cdot,e) \in \Lambda^{2}_{w}(T)$ is $\mathcal{E}/\mathcal{B}(\Lambda^{2}_{w}(T))$-measurable.
Thus, $\Xi^{1,2}_{w}(T,E)$ is a subspace of $L^{1}(E, \mathcal{E}, \varrho; \Lambda^{2}_{w}(T))$ and  $\Xi^{2,2}_{w}(T,E)$ is a subspace of $L^{2}(E, \mathcal{E}, \varrho; \Lambda^{2}_{w}(T))$.
\end{rema}


\begin{lemm}\label{lemmaSequenceSquareIntegApproxIntegFubini} Let $X \in \Xi^{1,2}_{w}(T)$. There exists a sequence $\{X_{n}\}_{n \in \N} \subseteq \Xi^{2,2}_{w}(T)$ such that $\varrho$-a.e. $\norm{X_{n}(\cdot,\cdot,\cdot,e)}_{w,T} \leq \norm{X_{n+1}(\cdot,\cdot,\cdot,e)}_{w,T}$, $\forall n \in \N$, and $$ \lim_{n \rightarrow \infty} \abs{\norm{X-X_{n}}}_{w,T,E}=0.$$
\end{lemm}
\begin{prf}
First, from Definition \ref{defiClassIntegrandsStochFubiniTheo}(3), there exists $E_{0} \subseteq E$ with $\varrho (E \setminus E_{0})=0$ such that $\forall e \in  E_{0}$, $\norm{X(\cdot,\cdot,\cdot,e)}_{w,T} < \infty$.

Let $\{G_{n}\}_{n \in \N}$ be an increasing sequence on $\mathcal{E}$ such that $E_{0}= \bigcup_{n \in \N} G_{n}$ and such that $\forall n \in \N$, $\varrho (G_{n})< \infty$. For each $n \in \N$, let $X_{n}=\{X_{n}(r,\omega,u,e)\}$ be the family of bounded random variables defined by:
\begin{eqnarray} \label{defiApproxSquareMomentIntegrandsFubiniTheorem}
X_{n}(r,\omega ,u,e) & = & \frac{n X(r,\omega,u,e)}{\norm{X(\cdot,\cdot,\cdot,e)}_{w,T}} \ind{\{e \in G_{n}: \norm{X(\cdot,\cdot,\cdot,e)}_{w,T}> n\}}{e} \\
& { } & +X(r,\omega,u,e) \ind{\{e \in G_{n}: \norm{X(\cdot,\cdot,\cdot,e)}_{w,T} \leq n \}}{e}. \nonumber
\end{eqnarray}

The properties (1)-(3) of Definition \ref{defiClassIntegrandsStochFubiniTheo} for $X$ imply that $X_{n}$ satisfies properties (1) and (2) of Definition \ref{defiClassIntegrandsStochFubiniTheo}. Moreover, \eqref{defiApproxSquareMomentIntegrandsFubiniTheorem} implies that $\abs{\norm{X_{n}}}^{2}_{w,2,T,E}  \varrho (de) \leq n^{2} \varrho (G_{n})< \infty$ and therefore $X_{n} \in \Xi^{2,2}_{w} (T,E)$.

Now, observe that from the definition of $E_{0}$ and of the $G_{n}$s we have:
\begin{equation} \label{limitSetsGnWeakMild}
\lim_{n \rightarrow \infty} \ind{\{e \in E_{0}: \norm{X(\cdot,\cdot,\cdot,e)}_{w,T}> n\}}{e}=0, \quad \lim_{n \rightarrow \infty} \ind{E_{0} \setminus G_{n}}{e}=0, \quad \forall e \in E.
\end{equation}
Hence, from \eqref{defiApproxSquareMomentIntegrandsFubiniTheorem}, \eqref{limitSetsGnWeakMild}, Definition \ref{defiClassIntegrandsStochFubiniTheo}(3) and the dominated convergence theorem:
\begin{flalign*}
&  \abs{\norm{X-X_{n}}}_{w,T,E} \\
& =  \int_{E_{0}} \ind{E_{0} \setminus G_{n}}{e}  \norm{X(\cdot,\cdot,\cdot,e)}_{w,T}\varrho (de) \\
& \hspace{10pt} +  \int_{E_{0}} \ind{\{e \in G_{n}: \norm{X(\cdot,\cdot,\cdot,e)}_{w,T}> n \}}{e} \abs{1- \frac{n}{\norm{X(\cdot,\cdot,\cdot,e)}_{w,T}}} \norm{X(\cdot,\cdot,\cdot,e)}_{w,T}\varrho (de) \\
& \leq   \int_{E_{0}} \ind{E_{0} \setminus G_{n}}{e}  \norm{X(\cdot,\cdot,\cdot,e)}_{w,T}\varrho (de)  \\ 
& \hspace{10pt} + 2 \int_{E_{0}} \ind{\{e \in E_{0}: \norm{X(\cdot,\cdot,\cdot,e)}_{w,T}> n \}}{e} \norm{X(\cdot,\cdot,\cdot,e)}_{w,T}\varrho (de) \\
& \rightarrow  0 \quad \mbox{as } n \rightarrow \infty .
\end{flalign*}
Finally, the fact that for every $e \in E_{0}$, $\norm{X_{n}(\cdot,\cdot,\cdot,e)}_{w,T} \leq \norm{X_{n+1}(\cdot,\cdot,\cdot,e)}_{w,T}$, $\forall n \in \N$, follows from \eqref{defiApproxSquareMomentIntegrandsFubiniTheorem}.   
\end{prf}

A proof of the following result can be carried out using similar arguments to those in the proof of Proposition \ref{simpleIntegrandsDenseInSquareIntegIntegrandsWeakIntegral}.

\begin{lemm} \label{lemmaDenseSimpleProcessInSquareInegFubini} Let $ S_{w}(T,E)$ denotes  the collection of all families $ X= \{X(r,\omega,u,e): r \in [0,T], \omega \in \Omega, u \in U, e \in E \}$ of Hilbert space-valued maps of the form:
\begin{equation} \label{defiClassSimpleIntegrandsStochFubiniTheo}
 X(r,\omega,u,e)= \sum^{p}_{l=1} \sum^{n}_{i=1} \sum^{m}_{j=1} \ind{]S_{j},t_{j}]}{r} \ind{F_{j}}{w} \ind{A_{i}}{u} \ind{D_{l}}{e} i_{q_{r,u}} \phi_{i,j,l},
\end{equation}
for all $r \in [0,T]$, $\omega \in \Omega$, $u \in U$, $e \in E$, where $m$, $n$, $p \in \N$, and for $l=1, \dots, p$, $i=1, \dots, n$, $j=1, \dots,m$, $0\leq s_{j} < t_{j} \leq T$, $F_{j} \in \mathcal{F}_{j}$, $A_{i} \in \mathcal{R}$, $D_{l} \in \mathcal{E}$ and $\phi_{i,j,l} \in \Phi$. Then, $S_{w}(T,E)$ is dense in $\Xi^{2,2}_{w}(T,E)$.
\end{lemm}

\begin{theo}[Stochastic Fubini's Theorem]\label{stochasticFubiniTheorem} Let $X \in \Xi^{1,2}_{w}(T,E)$. Then,
\begin{enumerate}
\item For a.e. $(r,\omega, u) \in [0,T] \times \Omega \times U$, the mapping $E \backin e \mapsto X(r,\omega,u,e) \in \Phi_{q_{r,u}}$ is Bochner integrable. Moreover,  
$$ \int_{E} X(\cdot,\cdot,\cdot,e)\, \varrho \, (de)= \left\{ \int_{E} X(r, \omega, u, e) \, \varrho (de): r \in [0,T], \omega \in \Omega,u \in U \right\} \in \Lambda^{2}_{w}(T) $$ 
\item The mapping $E \backin e \mapsto  I^{w} (X(\cdot,\cdot,\cdot,e)) \in  \mathcal{M}^{2}_{T}(\R)$ is Bochner integrable. Furthermore, 
\begin{equation} \label{defiIntegSecondPartStochFubini}
\left( \int_{E} I^{w} (X(\cdot,\cdot,\cdot,e)) \, \varrho \, (de) \right)_{t} = \int_{E} I^{w}_{t} (X(\cdot,\cdot,\cdot,e)) \, \varrho \, (de), \quad \forall \, t \geq 0.
\end{equation}
\item The following equality holds $\Prob$-a.e. 
\begin{equation} \label{identityFubiniTheoremWeakCase}
I^{w}_{t} \left( \int_{E} X (\cdot,\cdot,\cdot,e) \, \varrho  (de)\right) = \int_{E} I^{w}_{t} (X(\cdot,\cdot,\cdot,e))\, \varrho (de), \quad \forall \, t \in [0,T].
\end{equation}
\end{enumerate}
\end{theo}
\begin{prf}
Assume $X \in \Xi^{1,2}_{w} (T,E)$. For convenience we divide the proof into three parts. 

\textbf{Proof of (1).} First, from Definition \ref{defiClassIntegrandsStochFubiniTheo}(2) the mapping $(r,\omega,u,e)\mapsto q_{r,u} (X(r,\omega,u,e))$ is $P_{T} \otimes U \otimes \mathcal{E}$-measurable, then from the Minkowski inequality for integrals it follows that: 
\begin{flalign}
& \int_{[0,T] \times \Omega \times U} \left( \int_{E} \abs{q_{r,u}(X(r,\omega,u,e))}\varrho (de)\right) ^{2} (\lambda \otimes \Prob \otimes \mu) (d (r,\omega,u)) \label{finiteSquareMomentFubini} \\
& \leq \left(\int_{E} \left( \int_{[0,T] \times \Omega \times U}q_{r,u} (X(r,\omega,u,e))^{2}(\lambda \otimes \Prob \otimes \mu) (d (r,\omega,u))\right) ^{\frac{1}{2}} \varrho (de)\right)^{2} \nonumber \\
& = \abs{\norm{X}}^{2}_{W,T,E} < \infty \nonumber
\end{flalign}
Therefore, we conclude from \eqref{finiteSquareMomentFubini} that
\begin{equation} \label{aeFiniteMomentFubini}
\int_{E} q_{r,u} (X(r,\omega,u,e))\varrho (de)< \infty, \quad \mbox{for} \quad \lambda \otimes \Prob \otimes \mu \mbox{-a.e.} (r,\omega,u) \in [0,T] \times \Omega \times U. 
\end{equation}
Now, as for fixed $(r,\omega,u)$ the map $e \mapsto q_{r,u}(X(r,\omega,u,e))$ is $\mathcal{E}$-measurable, then the map $e \mapsto X(r,\omega,u,e)$ is $\mathcal{E}/ \mathcal{B} (\Phi_{q_{r,u}})$-measurable. Hence, because the Hilbert space $\Phi_{q_{r,u}}$ is separable it follows that the map $e \mapsto X(r,\omega,u,e)$ is strongly measurable. Moreover, \eqref{aeFiniteMomentFubini} implies that for almost every $(r,\omega,u)$ the map $e \mapsto X(r,\omega,u,e)$ is Bochner integrable.

To prove the second statement, let $\Gamma_{0} \subseteq [0,T] \times \Omega \times U$ be such that \eqref{aeFiniteMomentFubini} is satisfied. Then, for every $(r,\omega,u)\in \Gamma_{0}$ the Bochner integral $\int_{E} X(r,\omega,u,e) \varrho (de) \in \Phi_{q_{r,u}}$ exists. For $(r,\omega,u) \in \Gamma_{0}^{c}$ we define $\int_{E} X(r,\omega,u,e) \varrho (de)=0$. Thus, the family $\int_{E}X (\cdot,\cdot,\cdot,e) \varrho (de)$ defined in this way satisfies Definition \ref{defiClassIntegrandsStochFubiniTheo}(1).

Now, for each $\phi \in \Phi$ as the map $(r,\omega,u,e)\mapsto q_{r,u} (X(r,\omega,u,e)),\phi)$ is $\mathcal{P}_{T} \otimes \mathcal{B}(U)\otimes \mathcal{E}$-measurable and $\varrho$-integrable for all $(r,\omega,u) \in \Gamma_{0}$, then by Fubini's theorem the map 
$$(r,\omega,u)\mapsto q_{r,u} \left( \int_{E} X(r,\omega,u,e) \varrho (de),\phi\right)=\int_{E} q_{r,u} (X(r,\omega,u,e)),\phi) \varrho (de)$$ 
is $ \mathcal{P}_{T} \times \mathcal{B}(U)$-measurable. Hence Definition \ref{defiClassIntegrandsStochFubiniTheo}(2) is satisfied.

Finally, from  \eqref{finiteSquareMomentFubini} and Fubini's theorem, it follows that
\begin{flalign}
& \Exp \int^{T}_{0} \int_{U} q_{r,u}\left(\int_{E} X (\cdot,\cdot,\cdot,e) \varrho (de)\right) ^{2} \mu (du) \lambda (dr) 
\label{finiteNormSpaceIntegrandsBochnerIntegral} \\
& \leq \Exp \int^{T}_{0} \int_{U} \abs{\int_{E}q_{r,u}( X (r,\omega,u,e)) \varrho (de)}^{2} \mu (du) \lambda (dr) \nonumber \\
& \leq \int_{[0,T]\times \Omega \times U} \left( \int_{E} \abs{q_{r,u} (X(r,\omega,u,e))} \varrho (de)\right) ^{2} (\lambda \otimes \Prob \otimes \mu) (d(r,\omega,u)) < \infty . \nonumber 
\end{flalign}
Thus, $\norm{\int_{E} X (\cdot,\cdot,\cdot,e) \varrho (de)}_{w,T} < \infty$, and hence $\int_{E} X (\cdot,\cdot,\cdot,e) \varrho (de) \in \Lambda^{2}_{w}(T)$.

\textbf{Proof of (2).} First, note that from Definition \ref{defiClassIntegrandsStochFubiniTheo}(3), there exists some $E_{1} \subseteq E$ such that $\varrho(E \setminus E_{1})=0$ and such that $\norm{X (\cdot,\cdot,\cdot,e)}_{w,T}< \infty$, $\forall \, e \in E_{1}$. Hence, by redefining a version of $X$ to be equal to $X$ whenever $e \in E_{1}$ and to be $0$ whenever  $e \in E \setminus E_{1}$, if we call this version again by $X$, we have that for every  $e \in E$, $X(\cdot,\cdot,\cdot,e) \in \Lambda^{2}_{w}(T)$. Therefore, for every $e \in E$ the stochastic integral $I^{w}(X(\cdot,\cdot,\cdot,e)) \in \mathcal{M}^{2}_{T}(\R)$ exists.

Now we prove that the map $e \mapsto I^{w}(X(\cdot,\cdot,\cdot,e))$ is strongly measurable. First,  by an application of Lemmas \ref{lemmaSequenceSquareIntegApproxIntegFubini} and \ref{lemmaDenseSimpleProcessInSquareInegFubini}, there exists a sequence $\{X_{k}\}_{k \in \N}$ of families of the simple form \eqref{defiClassSimpleIntegrandsStochFubiniTheo}  such that 
\begin{equation} \label{convergenceSimpleProcessToXSTochFubini}
\lim_{k \rightarrow \infty} \abs{\norm{X-X_{k}}}_{w,T,E}=0.
\end{equation}
Note that if $X_{k}$ is of the form \eqref{defiClassSimpleIntegrandsStochFubiniTheo}, then for $e \in E$ its stochastic integral takes the form: 
$$I^{w}_{t}(X_{k}(\cdot,\cdot,\cdot,e))=\sum_{l=1}^{p} \sum^{n}_{i=1} \sum^{m}_{j=1} \ind{D_{l}}{e}\mathbbm{1}_{F_{j}}M((s_{j} \wedge t,t_{j} \wedge t],A_{i})[\phi_{i,j,l}], \quad \forall t \in [0,T],$$
(to simplify the notation, above we have omitted the dependence on $k$ of the components of \eqref{defiClassSimpleIntegrandsStochFubiniTheo}). Therefore, for each $k \in \N$ the map $e \mapsto I^{w}(X_{k}(\cdot,\cdot,\cdot,e))$ from $E$ into $\mathcal{M}^{2}_{T}(\R)$ is simple. Moreover, from the linearity of the map $I^{w}$, Doob's inequality, and \eqref{weakIntegralMapIsContinuousInequ}, it follows that:
\begin{multline} \label{convergPart1BochIntegralsStochIntegralsFubini}  
 \lim_{k\rightarrow \infty} \int_{E} \norm{I^{w}(X(\cdot,\cdot,\cdot,e))-I^{w}(X_{k}(\cdot,\cdot,\cdot,e))}_{\mathcal{M}^{2}_{T}(\R)} \varrho (de) \\
\leq  2 \sqrt{T} \lim_{k\rightarrow \infty} \int_{E} \norm{X(\cdot,\cdot,\cdot,e)-X_{k}(\cdot,\cdot,\cdot,e)}_{w,T} \varrho (de) 
 =  \lim_{k \rightarrow \infty} \abs{\norm{X-X_{k}}}_{w,T,E}.
\end{multline}
Then, it follows from \eqref{convergenceSimpleProcessToXSTochFubini},  \eqref{convergPart1BochIntegralsStochIntegralsFubini} and a standard use of the Chebyshev inequality and the Borel-Cantelli lemma that there exists a set $E_{2}\subseteq E$ with $\varrho(E \setminus E_{2})=0$ and a subsequence $\{X_{k_{q}}\}_{q \in \N}$ such that
$$\lim_{q \rightarrow \infty} \norm{I^{w}(X(\cdot,\cdot,\cdot,e))-I^{w}(X_{k_{q}}(\cdot,\cdot,\cdot,e))}_{\mathcal{M}^{2}_{T}(\R)}=0, \quad \forall e \in E_{2}. $$

In particular, this implies that the map $e \mapsto I^{w}(X(\cdot,\cdot,\cdot,e))$ is strongly measurable. Moreover, by a similar calculation to that in \eqref{convergPart1BochIntegralsStochIntegralsFubini} we have
$$\int_{E} \norm{I^{w}(X(\cdot,\cdot,\cdot,e))}_{\mathcal{M}^{2}_{T}(\R)} \varrho (de) \leq 2 \sqrt{T} \abs{\norm{X}}_{w,T,E}< \infty,$$
hence the mapping $e \mapsto I^{w}(X(\cdot,\cdot,\cdot,e))$ is Bochner integrable and furthermore, 
\begin{equation} \label{convergence2FubiniTheorem}
\norm{\int_{E} I^{w}(X(\cdot,\cdot,\cdot,e))\varrho (de)-\int_{E}I^{w} (X_{k_{q}}(\cdot,\cdot,\cdot,e))\varrho (de)}_{\mathcal{M}^{2}_{T}(\R)} \rightarrow 0, \quad \mbox{ as } q \rightarrow \infty
\end{equation}
 Finally, \eqref{defiIntegSecondPartStochFubini} follows from \eqref{convergence2FubiniTheorem} and the fact that \eqref{defiIntegSecondPartStochFubini} is satisfied for every $X_{k_{q}}$ due to their simple form.

\textbf{Proof of (3).} 
Let $\{X_{k_{q}}\}_{q \in \N}$ be the sequence to simple families as defined  in the proof of (2). For each $q \in \N$, the simple form of $X_{k_{q}}$ (see \eqref{defiClassSimpleIntegrandsStochFubiniTheo}) implies that it satisfies \eqref{identityFubiniTheoremWeakCase}.

Now, from Doob's inequality, \eqref{weakIntegralMapIsContinuousInequ}, \eqref{finiteNormSpaceIntegrandsBochnerIntegral} and \eqref{convergenceSimpleProcessToXSTochFubini}, and the linearity of both the weak stochastic integral and the Bochner integral, it follows that
\begin{flalign}
& \norm{I^{w}\left(\int_{E} X(\cdot,\cdot,\cdot,e)\varrho (de)\right)-I^{w}\left(\int_{E} X_{k_{q}}(\cdot,\cdot,\cdot,e)\varrho (de)\right)}_{\mathcal{M}^{2}_{T}(\R)} \label{convergence1FubiniTheorem} \\
& \leq  2 \sqrt{T} \norm{\int_{E} (X(\cdot,\cdot,\cdot,e)-X_{k_{q}}(\cdot,\cdot,\cdot,e))\varrho (de)}_{w,T}  \nonumber \\
& \leq  2 \sqrt{T} \abs{\norm{X-X_{k_{q}}}}_{w,T,E} \rightarrow 0, \quad  \mbox{ as } q \rightarrow \infty. \nonumber
\end{flalign}
Then, it follows from \eqref{convergence2FubiniTheorem} and  \eqref{convergence1FubiniTheorem}, and the fact that \eqref{identityFubiniTheoremWeakCase} is valid for each $X_{k_{q}}$, that the processes $I^{w} \left(\int_{E}(X(\cdot,\cdot,\cdot,e)\varrho (de)\right) $ and $\int_{E}I^{w} (\cdot,\cdot,\cdot,e)\varrho (de)$ are equal as elements of ${\mathcal{M}^{2}_{T}(\R)}$ and hence this implies that $X$ satisfies \eqref{identityFubiniTheoremWeakCase}.
\end{prf}

\section{The Strong Stochastic Integral} \label{SubsectionSSI}

\subsection{Construction of the Strong Stochastic Integral} \label{subsectionSSICABP}

In this section we proceed to construct the strong stochastic integral. We start by introducing the class of strong integrands. 

\begin{defi} \label{integrandsStrongIntegSquareMoments}
Let $\Lambda^{2}_{s}(\Psi,M;T)$ denote the collection of families $R=\{R(r,\omega,u): r \in [0,T], \omega \in \Omega, u \in U \}$ of operator-valued maps satisfying the following conditions:
\begin{enumerate}
\item $R(r,\omega,u) \in \mathcal{L}(\Phi'_{q_{r,u}},\Psi'_{\beta})$, for all $r \in [0, T]$, $\omega \in \Omega$, $u \in U$, 
\item $R$ is \emph{$q_{r,u}$-predictable}, i.e. for each $\phi \in \Phi$, $\psi \in \Psi$, the mapping $[0,T] \times \Omega \times U \rightarrow \R_{+}$ given by $(r,\omega,u) \mapsto q_{r,u}(R(r,\omega,u)' \psi, \phi)$ is $\mathcal{P}_{T} \otimes \mathcal{B}(U)$-measurable.
\item 
\begin{equation} \label{finiteSecondMomentIntegrandsStrongIntg}
\Exp \int_{0}^{T} \int_{U} q_{r,u}(R(r,u)'\psi)^{2} \mu(du) \lambda(dr) < \infty, \quad \forall \, \psi \in \Psi.
\end{equation} 
\end{enumerate}
\end{defi}

\begin{rema} \label{remaStrongIntegransSecondMomentsWellDefined}
Note that because $\Psi$ is reflexive,  the dual operator $R(r,\omega,u)'$ of $R(r,\omega,u)$ satisfies $R(r,\omega,u)' \in \mathcal{L}(\Psi, \Phi_{q_{r,u}})$. Then, Proposition \ref{compatibilityPredictableMapsAndSeminorms} guarantees that the map $(r,\omega,u) \mapsto q_{r,u}(R(r,\omega,u)'\psi)^{2}$ is $\mathcal{P}_{T} \otimes \mathcal{B}(U)$-measurable and hence the integral in \eqref{finiteSecondMomentIntegrandsStrongIntg} is well defined.  
\end{rema} 

It is easy to check that $\Lambda^{2}_{s}(\Psi,M;T)$ is a linear space. Now we introduce a class that extends the integrands considered in \cite{BojdeckiJakubowski:1990} to integrands depending also on the ``jump space variable'' $u$.

\begin{defi} \label{integrandsStrongIntegSquareMomentsInHilbertSpace}
Let $p$ be a continuous Hilbertian semi-norm on $\Psi$. Let $\Lambda^{2}_{s}(\Psi,M;p,T)$ denote the collection of families $R=\{R(r,\omega,u): r \in [0,T], \omega \in \Omega, u \in U \}$ of operator-valued maps satisfying the following conditions:
\begin{enumerate}
\item $R(r,\omega,u) \in \mathcal{L}_{2}(\Phi'_{q_{r,u}},\Psi'_{p})$, for all $r \in [0, T]$, $\omega \in \Omega$, $u \in U$, 
\item $R$ is \emph{$q_{r,u}$-predictable}, i.e. for each $\phi \in \Phi$, $\psi \in \Psi$, the mapping $[0,T] \times \Omega \times U \rightarrow \R_{+}$ given by $(r,\omega,u) \mapsto q_{r,u}(R(r,\omega,u)' \psi, \phi)$ is $\mathcal{P}_{T} \otimes \mathcal{B}(U)$-measurable.
\item 
\begin{equation} \label{finiteSecondMomentIntegrandsStrongIntgInHilbertSpace}
\norm{R}_{s,p,T}^{2} \defeq \Exp \int_{0}^{T} \int_{U} \norm{R(r,u)}^{2}_{\mathcal{L}_{2}(\Phi'_{q_{r,u}},\Psi'_{p})} \mu(du) \lambda(dr) < \infty.
\end{equation} 
\end{enumerate}
\end{defi}

\begin{rema} \label{remaStrongIntegrandsSecondMomentsHilbertSpaceWellDefined}
Note that as in Remark \ref{remaStrongIntegransSecondMomentsWellDefined}, Proposition \ref{compatibilityPredictableMapsAndSeminorms} guarantees that the map $(r,\omega,u) \mapsto \norm{R(r,\omega,u)}^{2}_{\mathcal{L}_{2}(\Phi'_{q_{r,u}},\Psi'_{p})}$ is $\mathcal{P}_{T} \otimes \mathcal{B}(U)$-measurable and hence the integral in \eqref{finiteSecondMomentIntegrandsStrongIntgInHilbertSpace} is well defined. 
\end{rema}

When there is no necessity to give emphasis to the dependence of the spaces $\Lambda^{2}_{s}(\Psi,M;T)$ and $\Lambda^{2}_{s}(\Psi,M;p,T)$ with respect to $\Psi$ and $M$, we denote these spaces by $\Lambda^{2}_{s}(T)$ and $\Lambda^{2}_{s}(p,T)$. We will use this shorter notation for the remainder of this article unless otherwise stated.

\begin{prop}[\cite{BojdeckiJakubowski:1990}, Proposition 2.4]\label{spaceStrongIntegrandsInHIlbertSpaceIsHilbertSpace}
 $\Lambda^{2}_{s}(p,T)$ is a Hilbert space when equipped with the inner product corresponding to the Hilbertian norm $\norm{\cdot}_{s,p,T}$. 
\end{prop} 

The relation between the spaces $\Lambda^{2}_{s}(T)$ and $\Lambda^{2}_{s}(p,T)$ is described in the following two results. 

\begin{prop} \label{spaceStrongIntegrandsHilbertSpaceAreOrdered}
If $p $ and $q$ are continuous Hilbertian semi-norms on $\Psi$ such that $p \leq q$, then $i'_{p,q} ( \Lambda^{2}_{s}(p,T)) \subseteq \Lambda^{2}_{s}(q,T)$, i.e. for each $R=\{R(r,\omega,u) \}  \in \Lambda^{2}_{s}(p,T)$, we have $i'_{p,q} R \defeq \{ i'_{p,q} R(r,\omega,u) \} \in \Lambda^{2}_{s}(q,T)$.  
\end{prop}
\begin{prf}
The fact that $i'_{p,q} \in \mathcal{L}(\Psi'_{p},\Psi'_{q})$ implies that for $R \in \Lambda^{2}_{s}(p,T)$, the family $i'_{p,q} R$ satisfies Definition \ref{integrandsStrongIntegSquareMomentsInHilbertSpace}(1)-(2) for $\Lambda^{2}_{s}(q,T)$. Moreover, from \eqref{finiteSecondMomentIntegrandsStrongIntgInHilbertSpace} it follows that $\norm{i'_{p,q} R}_{s,q,T} \leq \norm{i'_{p,q}}_{\mathcal{L}(\Psi'_{p},\Psi'_{q})} \norm{R}_{s,p,T}< \infty$.  
Hence, $i'_{p,q} R \in \Lambda^{2}_{s}(q,T)$.
\end{prf}

\begin{prop}\label{strongIntegrandsInHilbertSpaceAreSubspaces} For every continuous Hilbertian semi-norm $p$ on $\Psi$, we have $i'_{p} (\Lambda^{2}_{s}(p,T)) \subseteq \Lambda^{2}_{s}(T)$, i.e. for each $R=\{R(r,\omega,u) \}  \in \Lambda^{2}_{s}(p,T)$, we have $i'_{p} R \defeq \{ i'_{p} R(r,\omega,u) \} \in \Lambda^{2}_{s}(T)$. 
\end{prop} 
\begin{prf}
For $R \in \Lambda^{2}_{s}(p,T)$, the fact that $i'_{p} \in \mathcal{L}(\Psi'_{p},\Psi'_{\beta})$ implies that family $i'_{p} R$ satisfies Definition \ref{integrandsStrongIntegSquareMoments}(1)-(2). Moreover, from  \eqref{finiteSecondMomentIntegrandsStrongIntgInHilbertSpace} we have that 
\begin{equation*}
\Exp \int_{0}^{T} \int_{U} q_{r,u}(R(r,u)'i_{p} \psi)^{2} \mu(du) \lambda(dr) \\
\leq p(i_{p} \psi)^{2} \, \norm{R}_{s,p,T}^{2} < \infty, \quad \forall \, \psi \in \Psi. 
\end{equation*}
Hence, the fact that $(i'_{p} R(r,\omega,u))'=R(r,\omega,u)'i_{p}$ for every $(r,\omega,u)$, and the above inequality implies that $i'_{p} R$ satisfies \eqref{finiteSecondMomentIntegrandsStrongIntg}. Thus, $i'_{p} R \in \Lambda^{2}_{s}(T)$. 
\end{prf}

\begin{rema}
Proposition \ref{strongIntegrandsInHilbertSpaceAreSubspaces} shows that each $\Lambda^{2}_{s}(p,T)$ is a subspace of $\Lambda^{2}_{s}(T)$. Hence, our class of integrands $\Lambda^{2}_{s}(T)$ generalizes the class of integrands in \cite{BojdeckiJakubowski:1990}.  
\end{rema}

We now proceed to construct the strong stochastic integral. For integrands in the Hilbert space $\Lambda^{2}_{s}(p,T)$ we can try to generalize the arguments used in the construction of the weak stochastic integral by first defining the strong integral for a class of simple families and then to extend it by means of an isometry to integrands in $\Lambda^{2}_{s}(p,T)$. Indeed, a procedure of this type was carried out in \cite{BojdeckiJakubowski:1990} for integrals with respect to generalized Wiener process.  However, for the larger class of integrands $\Lambda^{2}_{s}(T)$ such a construction is not possible because the weak second moments condition \eqref{finiteSecondMomentIntegrandsStrongIntg} is not strong enough to provide a Hilbert space structure to the space $\Lambda^{2}_{s}(T)$. Therefore, a different approach has to be considered and we proceed to do this in what follows. 

Our construction of the strong stochastic integral can be summarized in the following way. First, we will identify each integrand $R$ in $\Lambda^{2}_{s}(T)$ with a unique element $\Delta(R)$ in $\mathcal{L}(\Psi,\Lambda^{2}_{w}(T))$ (see Theorem \ref{strongIntegrandsAsOperatorsToWeakIntegrans}). Then, we use the continuity and linearity of the weak integral map to show that  
$I^{w} \circ \Delta(R): \Psi \rightarrow \mathcal{M}^{2}_{T}(\R)$ defines a cylindrical martingale in $\Psi'$ satisfying the conditions in Theorem \ref{propertiesMartingales} and hence the existence of the strong integral will be provided by this theorem (see Theorem \ref{existenceStrongStochasticIntegSquareMoments}).
In accordance with the above plan, we proceed to show that $\Lambda^{2}_{s}(T)$ and $\mathcal{L}(\Psi,\Lambda^{2}_{w}(T))$ are isomorphic. 
 
\begin{theo} \label{strongIntegrandsAsOperatorsToWeakIntegrans}
The mapping $\Delta: \Lambda^{2}_{s}(T) \rightarrow \mathcal{L}(\Psi,\Lambda^{2}_{w}(T))$ given by 
\begin{equation} \label{isomorphismStrongIntegralsAndOperatorsToWeakIntegrands}
R \mapsto \left(\psi \mapsto R'\psi \defeq \{ R(r,\omega,u)'\psi: r \in [0,T], \omega \in \Omega, u \in U \} \right),
\end{equation}
is an isomorphism.  
\end{theo}   
\begin{prf} We divide the proof into two steps. 

\textbf{Step 1.} \emph{For every $R \in \Lambda^{2}_{s}(T)$, the map $\psi \mapsto R'\psi$ is an element of $\mathcal{L}(\Psi,\Lambda^{2}_{w}(T))$. Hence, the map $\Delta$ given by \eqref{isomorphismStrongIntegralsAndOperatorsToWeakIntegrands} is well-defined and linear. }

Let $R \in \Lambda^{2}_{s}(T)$. First, note that from Definition \ref{integrandsStrongIntegSquareMoments} and Remark \ref{remaStrongIntegransSecondMomentsWellDefined}, for every $\psi \in \Psi$, the family $R'\psi$ given by \eqref{isomorphismStrongIntegralsAndOperatorsToWeakIntegrands} satisfies the conditions of Definition \ref{integrandsWeakIntegSquareMoments}, and hence it is an element of $\Lambda^{2}_{w}(T)$. Therefore, the map $\psi \mapsto R'\psi$ from $\Psi$ into $\Lambda^{2}_{w}(T)$ is well-defined. Moreover, it is also linear as one can easily see from the linearity of each operator $R(r,\omega,u)' \in \mathcal{L}(\Psi, \Phi_{q_{r,u}})$. 

To prove that $\psi \mapsto R'\psi$ is also continuous, we will prove firstly that it is a sequentially closed operator. In such a case, from the fact that $\Psi$ is ultrabornological and $\Lambda^{2}_{w}(T)$ is a Hilbert space, the closed graph theorem (see \cite{NariciBeckenstein}, Theorem 14.7.3. p.475) implies that it is continuous. 

Let $\{ \psi_{n} \}_{n \in \N}$ be a sequence in $\Psi$ converging to some $\psi \in \Psi$ and let $X \in \Lambda^{2}_{w}(T)$ be such that $\{ R' \psi_{n} \}_{n \in \N}$ converges to $X$ in $\Lambda^{2}_{w}(T)$, i.e. we have
\begin{equation} \label{convergenceForProofClosedOperatorStrongIntegralsAsOperators}
\lim_{n \rightarrow \infty} \norm{R'\psi_{n}-X}_{w,T}^{2} = \lim_{n \rightarrow \infty} \Exp \int_{0}^{T} \int_{U} q_{r,u}(R(r,u)'\psi_{n}-X(r,u))^{2} \mu(du) \lambda(dr)=0 . 
\end{equation}
We need to prove that $X=R'\psi$. First, note that as for each $(r,\omega, u) \in [0,T] \times \Omega \times U$, we have $R(r,\omega,u)' \in \mathcal{L}(\Psi, \Phi_{q_{r,u}})$, then it follows that for each $(r,\omega,u)$, $\{ R(r,\omega,u)'\psi_{n} \}_{n \in \N}$ converges to $R(r,\omega,u)'\psi$ in $\Phi_{q_{r,u}}$ as $n \rightarrow \infty$. It follows from this, Fatou's lemma and \eqref{convergenceForProofClosedOperatorStrongIntegralsAsOperators}, that 
\begin{eqnarray*}
\norm{R'\psi - X}_{w,T}^{2} & = & \Exp \int_{0}^{T} \int_{U} q_{r,u}(R(r,u)'\psi - X(r,u))^{2} \mu(du) \lambda(dr) \\
& = & \Exp \int_{0}^{T} \int_{U} \lim_{n \rightarrow \infty} q_{r,u}(R(r,u)'\psi_{n}-X(r,u))^{2} \mu(du) \lambda(dr) \\
& \leq & \liminf_{n \rightarrow \infty} \Exp \int_{0}^{T} \int_{U} q_{r,u}(R(r,u)'\psi_{n}-X(r,u))^{2} \mu(du) \lambda(dr)=0 
\end{eqnarray*}
Therefore, we have $X=R'\psi$. Thus, $\psi \mapsto R'\psi$ is sequentially closed and by the closed graph theorem this implies that it is continuous. Hence, $\psi \mapsto R'\psi$ belongs to $\mathcal{L}(\Psi,\Lambda^{2}_{w}(T))$. This in particular implies that the mapping $\Delta$ is well-defined.

Finally, the fact that $\Delta$ is linear follows easily from \eqref{isomorphismStrongIntegralsAndOperatorsToWeakIntegrands} and the fact that for any $a \in \R$, $R, S \in \Lambda^{2}_{s}(T)$, for every $(r,\omega, u) \in [0,T]\times \Omega \times U$, it follows that $a R(r,\omega,u)'+S(r,\omega,u)'=\left(a R(r,\omega,u)+S(r,\omega,u) \right)'$.   

\textbf{Step 2.} \emph{The mapping $\Delta$ given by \eqref{isomorphismStrongIntegralsAndOperatorsToWeakIntegrands} is invertible.}   

We start by proving that $\Delta$ is injective. 
Let $R \in \Lambda^{2}_{s}(T)$ be such that $\Delta(R)=0$. Then, $R(r,\omega,u)' \psi =0$, for all $(r,\omega, u) \in [0,T]\times \Omega \times U$ and all $\psi \in \Psi$. Therefore, $R=0$. Thus, $\mbox{Ker}(\Delta)=\{ 0 \}$. But as $\Delta$ is linear, the above implies that it is also injective. 

Now, to prove that $\Delta$ is surjective, let $S \in \mathcal{L}(\Psi,\Lambda^{2}_{w}(T))$. It is easy to check that $q: \Psi \rightarrow \R_{+}$ given by 
$$ q(\psi) = \norm{ S \psi }_{w,T} = \left( \Exp \int_{0}^{T} \int_{U} q_{r,u}(S \psi)^{2} \mu(du) \lambda (dr) \right)^{1/2}, \quad \forall \, \psi \in \Psi, $$ is a continuous Hilbertian semi-norm on $\Psi$. Now, as $\Psi$ is a nuclear space, there exists a continuous Hilbertian semi-norm $p$ on $\Psi$ such that $q \leq p$ and $i_{q,p}$ is Hilbert-Schmidt. Hence, $ \norm{ S \psi }_{w,T} \leq p(\psi)$, for all $\psi \in \Psi$ and therefore $S$ is $p$-continuous. As $\Psi$ is dense in $\Psi_{p}$, it follows that $S$ has an extension $\tilde{S}$ such that $\tilde{S} \in \mathcal{L}(\Psi_{p}, \Lambda^{2}_{w}(T))$. Moreover, $\tilde{S}$ is Hilbert-Schmidt. This is because if $\{ \psi_{j}^{p} \}_{j \in \N} \subseteq \Psi$ is a complete orthonormal system in $\Psi_{p}$, then 
$$ \norm{ \tilde{S} }^{2}_{\mathcal{L}_{2}(\Psi_{p}, \Lambda^{2}_{w}(T))}= \sum_{j = 1}^{\infty} \norm{ \tilde{S} i_{p} \psi_{j}^{p} }^{2}_{w,T} = \sum_{j =1}^{\infty} q( i_{p} \psi_{j}^{p} )^{2}= \norm{ i_{q,p} }^{2}_{\mathcal{L}_{2}(\Psi_{p},\Psi_{q})} < \infty. $$   
Then, because $\tilde{S} \in \mathcal{L}_{2}(\Psi_{p}, \Lambda^{2}_{w}(T))$, there exists an orthonormal system $\{ \psi_{j}^{p} \}_{j \in J}$ in $\Psi_{p}$, an orthonormal system $\{ X_{j} \}_{j \in J}$ in $\Lambda^{2}_{w}(T)$ and a sequence of positive numbers $\{ \gamma_{j} \}_{j \in J}$ satisfying $\sum_{j \in J} \gamma_{j}^{2} < \infty$, with $J \subseteq \N$, such that $\tilde{S}$ admits the representation:
\begin{equation} \label{hilbertSchmidtRepresentExtenOperatS}
\tilde{S} \psi = \sum_{j \in J} \gamma_{j} \, p(\psi,\psi_{j}^{p}) X_{j}, \quad \forall \, \psi \in \Psi_{p}.  
\end{equation}    
Choose a complete orthonormal system $\{ \psi^{p}_{j} \}_{j \in \N}$ which is an extension of the orthonormal system $\{ \psi^{p}_{j} \}_{j \in J}$. Then, from \eqref{hilbertSchmidtRepresentExtenOperatS} we have
\begin{equation} \label{operatorSTildeInExtendedComplOrthoSystem}
\tilde{S} \psi_{j}^{p} = \gamma_{j} X_{j} \mbox{  if  } j \in J, \quad  \mbox{and} \quad \tilde{S} \psi^{p}_{j}=0 \mbox{  if  }j \in \N \setminus J. 
\end{equation}
Now, from Parseval's identity and the fact that $\tilde{S} \in \mathcal{L}_{2}(\Psi_{p}, \Lambda^{2}_{w}(T))$ it follows that 
\begin{equation} \label{finiteInOrthonormalBasisOperatorToWeakIntegrands}
 \Exp \int_{0}^{T} \int_{U} \sum_{j \in \N} q_{r,u}((\tilde{S} \psi_{j}^{p})(r,\omega,u))^{2} \mu(du) \lambda(dr) 
 = || \tilde{S} ||^{2}_{\mathcal{L}_{2}(\Psi_{p}, \Lambda^{2}_{w}(T))} < \infty.
\end{equation} 
Then, it follows from \eqref{operatorSTildeInExtendedComplOrthoSystem} and \eqref{finiteInOrthonormalBasisOperatorToWeakIntegrands} that there exists $\Gamma \subseteq [0,T] \times \Omega \times U$, such that $(\lambda \otimes \Prob \otimes \mu)(\Gamma)=1$ and 
\begin{equation} \label{almostSureHilbSchmFamilyOperMapDeltaIsSubjective}
\sum_{j \in J} \gamma_{j}^{2} q_{r,u}(X_{j}(r,\omega,u))^{2} = \sum_{j \in \N} q_{r,u}((\tilde{S} \psi_{j}^{p})(r,\omega,u))^{2}  < \infty, \quad \forall \,  (r,\omega,u) \in \Gamma. 
\end{equation}  

Let $F = \{ F(r,\omega,u): r \in [0,T], \omega \in \Omega, u \in U \}$, where for every $\psi \in \Psi$,   
\begin{equation}\label{dualFamilyOperatorsProofMapDeltaIsSubjective}
F(r,\omega,u) \psi = 
\begin{cases} 
(\tilde{S} \psi) (r,\omega,u), & \forall \, (r,\omega,u) \in \Gamma, \\
0, &  \forall \, (r,\omega,u) \in \Omega \setminus \Gamma. 
\end{cases}
\end{equation}
Our objective is to prove that the family $F$ satisfies the following properties:
\begin{enumerate}[label=(\alph*)]
\item $F(r,\omega,u) \in \mathcal{L}_{2}(\Psi_{p},\Phi_{q_{r,u}})$, for all $(r,\omega,u) \in [0,T] \times \Omega \times U$, 
\item The map $ (r,\omega,u) \mapsto q_{r,u}(F(r,\omega,u)\psi,\phi)$ is $\mathcal{P}_{T} \otimes \mathcal{B}(U)$-measurable, for each $\phi \in \Phi$, $\psi \in \Psi$, 
\item $\Exp \int_{0}^{T} \int_{U} \norm{F(r,u)}^{2}_{\mathcal{L}_{2}(\Psi_{p},\Phi_{q_{r,u}})} \mu(du) \lambda(dr) < \infty$. 
\end{enumerate}

To prove (a), first note that from \eqref{hilbertSchmidtRepresentExtenOperatS} and \eqref{dualFamilyOperatorsProofMapDeltaIsSubjective}, $F(r,\omega,u)$ is a linear operator from $\Psi_{p}$ into $\Phi_{q_{r,u}}$, for all $(r,\omega,u) \in [0,T] \times \Omega \times U$. 

Fix $(r,\omega,u) \in \Gamma$. Then, from \eqref{hilbertSchmidtRepresentExtenOperatS}, \eqref{dualFamilyOperatorsProofMapDeltaIsSubjective}, the Cauchy-Schwarz inequality and Parseval's identity, it follows that for all $\psi \in \Psi_{p}$ we have
\begin{equation*}
q_{r,u}( F(r,\omega, u)\psi)^{2} 
 \leq \left( \sum_{j \in J} p(\psi,\psi_{j}^{p})^{2} \right)   \,  \cdot \, \left( \sum_{j \in J} \gamma_{j}^{2} \, q_{r,u}\left( X_{j}(r,\omega, u) \right)^{2}  \right) \leq  C \, p(\psi)^{2},     
\end{equation*}
where $C= \sum_{j =1}^{\infty} \gamma_{j}^{2} q_{r,u}\left( X_{j}(r,\omega, u) \right)^{2} < \infty$ by \eqref{almostSureHilbSchmFamilyOperMapDeltaIsSubjective}. Thus, $F(r,\omega, u)$ is a continuous operator from $\Psi_{p}$ into $\Phi_{q_{r,u}}$. Moreover, because $\{\psi_{j}^{p}\}_{j \in \N}$ is a complete orthonormal system in $\Psi_{p}$, then \eqref{operatorSTildeInExtendedComplOrthoSystem}, \eqref{almostSureHilbSchmFamilyOperMapDeltaIsSubjective} and \eqref{dualFamilyOperatorsProofMapDeltaIsSubjective} show that $\norm{F(r,\omega,u)}_{\mathcal{L}_{2}(\Psi_{p},\Phi_{q_{r,u}})}< \infty$ and hence $F(r,\omega,u) \in \mathcal{L}_{2}(\Psi_{p},\Phi_{q_{r,u}})$. As for $(r,\omega,u) \in \Omega \setminus \Gamma$ we have $F(r,\omega,u)=0$, from the above it follows that $F(r,\omega,u) \in \mathcal{L}_{2}(\Psi_{p},\Phi_{q_{r,u}})$ for all $(r,\omega, u) \in [0,T] \times \Omega \times U$ and therefore we have proved (a). 

For (b), fix $\psi \in \Psi$ and $\phi \in \Phi$. From \eqref{hilbertSchmidtRepresentExtenOperatS} and  \eqref{dualFamilyOperatorsProofMapDeltaIsSubjective}, for all $(r,\omega,u) \in \Gamma$ we have 
\begin{equation} \label{innerProdPreImageFamilyMapDeltaSurjective}
 q_{r,u}(F(r,\omega,u) \psi, \phi)= \sum_{j \in J} \gamma_{j} \, p(\psi,\psi_{j}^{p}) q_{r,u}(X_{j}(r,\omega,u), \phi). 
\end{equation} 
As for each $j \in J$, $X_{j} \in \Lambda^{2}_{w}(T)$, then the map $(r,\omega,u) \mapsto q_{r,u}(X_{j}(r,\omega,u), \phi)$ is $\mathcal{P}_{T}$-measurable (see Definition \ref{integrandsWeakIntegSquareMoments}). Then, from   \eqref{innerProdPreImageFamilyMapDeltaSurjective} the map $(r,\omega,u) \mapsto  q_{r,u}(F(r,\omega,u) \psi, \phi)$ is $\mathcal{P}_{T}$-measurable. So we have proved (b). 

Finally, (c) follows because \eqref{finiteInOrthonormalBasisOperatorToWeakIntegrands} and \eqref{dualFamilyOperatorsProofMapDeltaIsSubjective} imply that 
\begin{multline*} 
\Exp \int_{0}^{T} \int_{U} \norm{ F(r,u)}^{2}_{\mathcal{L}_{2}(\Psi_{p},\Phi_{q_{r,u}})} \mu(du) \lambda(dr) 
=  \Exp \int_{0}^{T} \int_{U} \sum_{j \in J} q_{r,u}((\tilde{S} \psi_{j}^{p})(r,u))^{2} \mu(du) \lambda(dr) <  \infty.
\end{multline*} 

Define $R= \{ R(r,\omega,u): r \in [0,T], \omega \in \Omega, u \in U \}$ by:
\begin{equation} \label{preImageFamilyOperatorsProofMapDeltaIsSubjective}
R(r,\omega,u)=F(r,\omega,u)', \quad \forall \, (r,\omega,u) \in [0,T] \times \Omega \times U. 
\end{equation}
then from the properties (a)-(c) above it follows that $R \in \Lambda^{2}_{s}(p,T)$ (see Definition \ref{integrandsStrongIntegSquareMomentsInHilbertSpace}) and hence by Proposition \ref{strongIntegrandsInHilbertSpaceAreSubspaces} we have $i'_{p} R \in \Lambda^{2}_{s}(T)$. Moreover, as $\tilde{S}$ is an extension of $S$, from \eqref{dualFamilyOperatorsProofMapDeltaIsSubjective} for every $\psi \in \Psi$ and $(r,\omega,u) \in \Gamma$, we have that 
\begin{equation} \label{almostSureImageProofSurjectiveMapDelta}
(i'_{p} R(r,\omega,u))'\psi =F(r,\omega,u) i_{p} \psi = (\tilde{S} i_{p} \psi)(r,\omega,u)= (S \psi)(r,\omega,u),
\end{equation}
and then from \eqref{isomorphismStrongIntegralsAndOperatorsToWeakIntegrands} it follows that $S= \Delta (i'_{p} R)$. Therefore, the map $\Delta$ is surjective and hence it is an isomorphism. 
\end{prf}

\begin{coro} \label{strongIntegrandsAsUnionIntegrandsInHilbertSpaces} Let $R \in \Lambda^{2}_{s}(T)$. There exists a continuous Hilbertian semi-norm $p$ on $\Psi$ and $\tilde{R} \in \Lambda^{2}_{s}(p,T)$ such that $R(r,\omega,u)= i'_{p} \tilde{R}(r,\omega,u)$, for $\lambda \otimes \Prob \otimes \mu$-a.e. $(r,\omega,u) \in [0,T] \times \Omega \times U$. 

Moreover, if $H(\Psi)$ denotes the collection of all the continuous Hilbertian semi-norms on $\Psi$, then
$$\Lambda^{2}_{s}(T)= \bigcup_{p \in H(\Psi)} i'_{p} \Lambda^{2}_{s}(p,T).$$
\end{coro}
\begin{prf} First, from Step 1 of the proof of Theorem \ref{strongIntegrandsAsOperatorsToWeakIntegrans} we have that $\psi \mapsto R'\psi$ given in \eqref{isomorphismStrongIntegralsAndOperatorsToWeakIntegrands} is an element of $\mathcal{L}(\Psi,\Lambda^{2}_{w}(T))$. Then it follows from Step 2 of the proof of Theorem \ref{strongIntegrandsAsOperatorsToWeakIntegrans} that there exists a continuous Hilbertian semi-norm $p$ on $\Psi$ and there exists $\tilde{R}$ in $\Lambda^{2}_{s}(p,T)$ such that 
 for $\lambda \otimes \Prob \otimes \mu$-a.e. $(r,\omega,u) \in [0,T] \times \Omega \times U$, $R(r,\omega,u)'\psi= (i'_{p} \tilde{R}(r,\omega,u))'\psi$ (note that this is \eqref{almostSureImageProofSurjectiveMapDelta} with $S$ replaced by the map $\psi \mapsto R'\psi$). 

To prove the second statement, note that as a consequence of the first statement we have $\Lambda^{2}_{s}(T) \subseteq  \bigcup_{p \in H(\Psi)} i'_{p} \Lambda^{2}_{s}(p,T)$. Now, from Proposition \ref{strongIntegrandsInHilbertSpaceAreSubspaces} we have that $i'_{p} \Lambda^{2}_{s}(p_{\gamma},T) \subseteq \Lambda^{2}_{s}(T)$, for each $p \in H(\Psi)$. Then, $\bigcup_{p \in H(\Psi)} i'_{p} \Lambda^{2}_{s}(p,T) \subseteq \Lambda^{2}_{s}(T)$.     
\end{prf}

We are now ready to construct the strong stochastic integral for elements of $\Lambda^{2}_{s}(T)$. We do this in the following theorem.

\begin{theo} \label{existenceStrongStochasticIntegSquareMoments} Let $R \in \Lambda^{2}_{s}(T)$. Then there exist a continuous Hilbertian semi-norm $p$ on $\Psi$ and $\tilde{R} \in \Lambda^{2}_{s}(p,T)$ such that $R(r,\omega,u)= i'_{p} \tilde{R}(r,\omega,u)$, for $\lambda \otimes \Prob \otimes \mu$-a.e. $(r,\omega,u) \in [0,T] \times \Omega \times U$ and a $\Phi'_{p}$-valued zero-mean square-integrable c\`{a}dl\`{a}g martingale $I^{s}(R)=\{ I^{s}_{t}(R) \}_{t \in [0,T]}$,  satisfying      
\begin{equation} \label{itoIsometryStrongIntegHilbSchmiIntegrands}
\Exp \, p' (I^{s}_{t}(R))^{2} \, = \Exp \int_{0}^{t} \int_{U} \norm{\tilde{R}(r,u)}^{2}_{\mathcal{L}_{2}(\Phi'_{q_{r,u}},\Psi'_{p})} \mu(du) \lambda(dr)=\norm{\tilde{R}}_{s,p,t}^{2}, 
\end{equation}
for all $t \in [0,T]$. Moreover, $I^{s}(R)$ is also mean-square continuous and has a predictable version. 

Furthermore, $I^{s}(R) \in \mathcal{M}^{2}_{T}(\Phi'_{\beta})$ and it is the unique (up to indistinguishable versions) $\Psi'_{\beta}$-valued process such that for all $\psi \in \Psi$, $\Prob$-a.e. 
\begin{equation} \label{weakStrongCompatibilityStochIntegral}
I^{s}_{t}(R)[\psi]=I^{w}_{t}(R'\psi), \quad \forall  \, t \in [0,T], 
\end{equation} 
where the stochastic process in the right-hand side of \eqref{weakStrongCompatibilityStochIntegral} corresponds to the weak stochastic integral of $R' \psi \in \Lambda^{2}_{w}(T)$ defined in \eqref{isomorphismStrongIntegralsAndOperatorsToWeakIntegrands}. 
\end{theo}
\begin{prf}
Let $R \in \Lambda^{2}_{s}(T)$. First, it follows from Corollary \ref{strongIntegrandsAsUnionIntegrandsInHilbertSpaces} that there exists a continuous Hilbertian semi-norm $q$ on $\Psi$ and $\hat{R} \in \Lambda^{2}_{s}(q,T)$ such that $R(r,\omega,u)= i'_{q} \hat{R}(r,\omega,u)$, for $\lambda \otimes \Prob \otimes \mu$-a.e. $(r,\omega,u) \in [0,T] \times \Omega \times U$.

On the other hand, from the continuity of the weak integral map $I^{w}: \Lambda^{2}_{w}(T) \rightarrow \mathcal{M}^{2}_{T}(\R)$ (Theorem \ref{propertiesWeakIntegralSquareIntegIntegrands}) and Theorem \ref{strongIntegrandsAsOperatorsToWeakIntegrans}, it follows that the map $I^{w} \circ \Delta(R): \Psi \rightarrow \mathcal{M}^{2}_{T}(\R)$ is linear and continuous. Therefore, 
$I^{w} \circ \Delta(R) = \{ I_{t}^{w} \circ \Delta(R)\}_{t \in [0,T]}$ is a cylindrical zero-mean square integrable c\`{a}dl\`{a}g, martingale in $\Psi'$ such that for each $t \in [0,T]$ the linear map $I_{t}^{w} \circ \Delta(R): \Psi \mapsto L^{0} \ProbSpace$, $\psi \mapsto I_{t}^{w}(\Delta(R) \psi)$, is continuous. Then, by Theorem \ref{propertiesMartingales} there exist a  continuous Hilbertian semi-norm $p$ on $\Psi$, that we can choose  satisfying that $q \leq p$,  and a $\Phi'_{p}$-valued  zero-mean square integrable c\`{a}dl\`{a}g, martingale $I^{s}(R)=\{ I^{s}_{t}(R) \}_{t \in [0,T]}$ such that for each $\psi \in \Psi$ the real-valued process $I^{s}(R)[\psi]$ is a version of $I^{w} \circ \Delta(R)(\psi)=I^{w} (R'\psi)$. But as for every $\psi \in \Psi$ the processes $I^{s}(R)[\psi]$ and $I^{w} \circ \Delta(R)(\psi)=I^{w} (R'\psi)$ are both c\`{a}dl\`{a}g then they are indistinguishable. This shows \eqref{weakStrongCompatibilityStochIntegral}. Moreover, because $I^{s}(R)$ is also a $\Phi'_{\beta}$-valued regular c\`{a}dl\`{a}g process such that for each $\psi \in \Psi$ we have $I^{s} (R)[\psi] \in \mathcal{M}^{2}_{T}(\R)$  (this last follows from \eqref{weakStrongCompatibilityStochIntegral} and the fact that 
$I^{w} (R'\psi) \in \mathcal{M}^{2}_{T}(\R)$) then by definition we have that $I^{s}(R) \in \mathcal{M}^{2}_{T}(\Phi'_{\beta})$. Furthermore, \eqref{weakStrongCompatibilityStochIntegral} and Proposition \ref{propCondiIndistingProcess} shows that $I^{s}(R)$ is the unique (up to indistinguishable versions) $\Psi'_{\beta}$-valued process satisfying the conditions on the statement of the theorem. 

Now, if we define $\tilde{R}=i'_{q,p} \hat{R}$, then by Proposition \ref{spaceStrongIntegrandsHilbertSpaceAreOrdered} we have $\tilde{R} \in \Lambda^{2}_{s}(q,T)$ and moreover because $i'_{q}=i'_{p} \circ i'_{q,p}$ then we have that  $R(r,\omega,u)= i'_{p} \tilde{R}(r,\omega,u)$, for $\lambda \otimes \Prob \otimes \mu$-a.e. $(r,\omega,u) \in [0,T] \times \Omega \times U$. Hence, for each $t \in [0,T]$, from Parseval's identity, Fubini's theorem, \eqref{secondMomentWeakStochasticIntegral} and  \eqref{weakStrongCompatibilityStochIntegral} we have
\begin{eqnarray*}
\Exp \, p' (I^{s}_{t}(R))^{2} & = & \sum_{j = 1}^{\infty} \Exp \left[ \abs{I^{s}_{t}(R)[\psi_{j}^{p}] }^{2} \right] \\
& = &  \sum_{j = 1}^{\infty} \Exp \left[ \abs{I^{w}_{t}(R'\psi_{j}^{p}) }^{2} \right] \\
& = &  \sum_{j = 1}^{\infty} \Exp \int_{0}^{t} \int_{U} q_{r,u}(\tilde{R}(r,u)'i_{p} \psi_{j}^{p})^{2} \mu(du) \lambda(dr) \\
& = & \Exp \int_{0}^{t} \int_{U} \norm{\tilde{R}(r,u)}^{2}_{\mathcal{L}_{2}(\Phi'_{q_{r,u}},\Psi'_{p})} \mu(du) \lambda(dr).
\end{eqnarray*}
This proves \eqref{itoIsometryStrongIntegHilbSchmiIntegrands}. Now, to show that $I^{s}(R)$ is mean-square continuous, note that from \eqref{itoIsometryStrongIntegHilbSchmiIntegrands} it follows that for any $0 \leq s \leq t \leq T$ we have:
$$ \Exp \left( p'(I^{s}_{s}(R) - I^{s}_{t}(R) )^{2} \right) = \Exp \int^{t}_{s} \norm{\tilde{R}(r,u)}^{2}_{\mathcal{L}_{2}(\Phi'_{q_{r,u}},\Psi'_{p})} \mu(du) \lambda(dr) \leq \norm{\tilde{R}}_{s,p,T}^{2},$$
and hence from an application of the dominated convergence theorem we have 
$$ \Exp \left( p'(I^{s}_{s}(R) - I^{s}_{t}(R) )^{2} \right) \rightarrow 0 \quad \mbox{ as } s \rightarrow t, \mbox{ or } t \rightarrow s. $$
Thus, $I^{s}(R)$ is mean square continuous. Finally, as $I^{s}(R)$ is a $\Psi'_{p}$-valued, $\{ \mathcal{F}_{t}\}$-adapted and stochastically continuous process it has a predictable version (see Proposition 3.21 of Peszat and Zabczyk \cite{PeszatZabczyk}, p.27). \end{prf}

\begin{prop} \label{propStrongStochIntegContIfCylinMartIsCont}
If for each $A \in \mathcal{R}$ and $\phi \in \Phi$, the real-valued process $(M(t,A)(\phi): t \geq 0)$ is continuous, then for each $R \in \Lambda^{2}_{s}(T)$, the $\Psi'_{p}$-valued process $I^{s}(R)$ defined in Theorem \ref{existenceStrongStochasticIntegSquareMoments} is continuous.
\end{prop}
\begin{prf}
Let $R \in \Lambda^{2}_{s}(T)$. With the notation of the proof of Theorem \ref{existenceStrongStochasticIntegSquareMoments}, it follows from Proposition \ref{propWeakStochIntegContIfCylinMartIsCont} that for each $\psi \in \Psi$ the real-valued process $\{I_{t}^{w}(\Delta(R) \psi)\}_{t \in [0,T]}$ is continuous. Then, by Theorem \ref{propertiesMartingales} we can choose $p$ such that $I^{s}(R)$ is a $\Psi'_{p}$-valued zero-mean square integrable continuous martingale.  
\end{prf}

\begin{defi} \label{defiStrongStochasticIntegralSquareMoments} For $R \in \Lambda^{2}_{s}(T)$ let $I^{s}(R)$ be the $\Psi'_{\beta}$-valued process defined in Theorem \ref{existenceStrongStochasticIntegSquareMoments}. We call $I^{s}(R)$ the \emph{strong stochastic integral} of $R$. We will sometimes denote the stochastic integral $I^{s}(R)$ of $R$ by $\left\{ \int^{t}_{0} \int_{U} R (r,u) M (dr, du): t \in [0,T] \right\}$. The map $I^{s}:\Lambda^{2}_{s}(T) \rightarrow \mathcal{M}^{2}_{T}(\Psi'_{\beta})$ given by $R \mapsto I^{s}(R)$, will be called the \emph{strong integral mapping}.   
\end{defi}

The relation \eqref{weakStrongCompatibilityStochIntegral} between the weak and strong stochastic integral is called the \emph{weak-strong compatibility} (see \cite{Applebaum:2006}). Note that this relation is an intrinsic consequence of our construction. 

Now we proceed to study some properties of the strong integral mapping. 

\begin{prop} \label{linearityStrongIntegralLinearOperators}
The strong integral mapping $I^{s}:\Lambda^{2}_{s}(T) \rightarrow \mathcal{M}^{2}_{T}(\Psi'_{\beta})$ is linear.  
\end{prop}
\begin{prf} The result can be easily proved using \eqref{weakStrongCompatibilityStochIntegral}, the linearity of the weak integral and Proposition \ref{propCondiIndistingProcess}. We leave the details to the reader. 
\end{prf}

Our next objective is to introduce some topologies on the space  $\Lambda^{2}_{s}(T)$ for which the strong integral mapping is continuous. In view of Theorem \ref{strongIntegrandsAsOperatorsToWeakIntegrans}, because the spaces $\Lambda^{2}_{s}(T)$ and $\mathcal{L}(\Psi,\Lambda^{2}_{w}(T))$ are isomorphic we can then equip $\Lambda^{2}_{s}(T)$ with a locally convex topology induced by the operator topology on $\mathcal{L}(\Psi,\Lambda^{2}_{w}(T))$. 

In effect, we identify each element $R$ of $\Lambda^{2}_{s}(T)$ with the unique element $(\psi \mapsto R' \psi)$ in $\mathcal{L}(\Psi,\Lambda^{2}_{w}(T))$ given by \eqref{isomorphismStrongIntegralsAndOperatorsToWeakIntegrands}. Then, on $\Lambda^{2}_{s}(T)$ we define the topology of bounded (respectively simple) convergence as the locally convex topology generated by the following family of semi-norms:
\begin{equation} \label{semiNormsBoundedConvTopoStrongIntegrans}
 R \rightarrow \sup_{ \psi \in B} \norm{ R' \psi}_{w,T}= \sup_{\psi \in B}  \left( \Exp \int_{0}^{T} \int_{U} q_{r,u}(R(r,u)'\psi)^{2} \mu(du) \lambda(dr) \right)^{\frac{1}{2}},
\end{equation}
where $B$ runs over the bounded (respectively finite) subsets of $\Psi$. Hence, the topology of bounded (respectively simple) convergence on $\Lambda^{2}_{s}(T)$ is the topology of bounded (respectively simple) convergence on $\mathcal{L}(\Psi,\Lambda^{2}_{w}(T))$ defined on $\Lambda^{2}_{s}(T)$ via the isomorphism \eqref{isomorphismStrongIntegralsAndOperatorsToWeakIntegrands}. 

\begin{prop}
The space $\Lambda^{2}_{s}(T)$ is complete equipped with the topology of bounded convergence and quasi-complete equipped with the topology of simple convergence. 
\end{prop}
\begin{prf}
The assertion follow from the corresponding properties of the topologies of bounded and simple convergence of the space $\mathcal{L}(\Psi,\Lambda^{2}_{w}(T))$, and the fact that $\Psi$ is ultrabornological and $\Lambda^{2}_{w}(T)$ is a Hilbert space. See \cite{KotheII}, Section 39.6, for details on these topologies.
\end{prf}

From Proposition \ref{strongIntegrandsInHilbertSpaceAreSubspaces}, the spaces $\Lambda^{2}_{s}(p,T)$, where $p$ ranges over the continuous Hilbertian semi-norms $p$ on $\Psi$, are linear subspaces of $\Lambda^{2}_{s}(T)$. The following result shows that the Hilbert topology on each space $\Lambda^{2}_{s}(p,T)$ (see Proposition \ref{strongIntegrandsInHilbertSpaceAreSubspaces}) is finer than the subspace topology induced on them by the topologies of simple and bounded convergence on $\Lambda^{2}_{s}(T)$. 

\begin{prop} \label{topoHilbSubpIsFinerThanBounAndSimpConvStrongIntg}
Let $p$ be a continuous Hilbertian semi-norm on $\Psi$. Let $\Lambda^{2}_{s}(T)$ be equipped with either the topology of simple or the topology of bounded convergence. Then, the inclusion map $i'_{p}:\Lambda^{2}_{s}(p,T) \rightarrow \Lambda^{2}_{s}(T)$, $R \mapsto i'_{p}R$, is linear and continuous. 
\end{prop}
\begin{prf} The linearity of the inclusion map is evident. To prove its continuity, let $B$ be any bounded subset of $\Psi$. As $p$ is continuous, there exists $C >0$ such that $B \subseteq C B_{p}(1)$. Then, for any $R \in \Lambda^{2}_{s}(p,T)$ we have from \eqref{finiteSecondMomentIntegrandsStrongIntgInHilbertSpace} and \eqref{semiNormsBoundedConvTopoStrongIntegrans} that, 
\begin{eqnarray*}
\sup_{\psi \in B} \norm{R' i_{p} \psi}^{2}_{w,T} 
& \leq & C^{2} \sup_{\psi \in B_{p}(1)} \Exp \int_{0}^{T} \int_{U} q_{r,u}(R(r,u)' i_{p} \psi)^{2} \mu(du) \lambda(dr) \\
& \leq & C^{2} \left( \sup_{\psi \in B_{p}(1)} p(\psi)^{2} \right) \, \Exp \int_{0}^{T} \int_{U} \norm{ i'_{p} R(r,u)}^{2}_{\mathcal{L}_{2}(\Phi'_{q_{r,u}},\Psi'_{p})}\mu(du) \lambda(dr) \\
& = &  C^{2} \norm{i'_{p} R}_{s,p,T}^{2}.
\end{eqnarray*}
Then, the inclusion map $i'_{p}:\Lambda^{2}_{s}(p,T) \rightarrow \Lambda^{2}_{s}(T)$ is continuous. 
\end{prf}

The next result shows that the strong integral map is continuous from $\Lambda^{2}_{s}(T)$ into $\mathcal{M}^{2}_{T}(\Psi'_{\beta})$. We will need the topologies on $\mathcal{M}^{2}_{T}(\Psi'_{\beta})$ defined in Section \ref{subSectionMDNS}. 
 
\begin{prop} \label{continuityStrongIntegralMapLinearOperators}
Let $\Lambda^{2}_{s}(T)$ and $\mathcal{M}^{2}_{T}(\Psi'_{\beta})$ be equipped with either the topology of simple or the topology of bounded convergence. Then, the map $I^{s}:\Lambda^{2}_{s}(T) \rightarrow \mathcal{M}^{2}_{T}(\Psi'_{\beta})$ is continuous. 
\end{prop}
\begin{prf}
Let $B$ be any bounded subset of $\Psi$. For any $R \in \Lambda^{2}_{s}(T)$, it follows from \eqref{weakIntegralMapIsContinuousInequ}, \eqref{semiNormsBoundedConvTopoStrongIntegrans} and \eqref{weakStrongCompatibilityStochIntegral} that 
$$ \sup_{\psi \in B} \norm{I^{s}(R)[\psi]}^{2}_{\mathcal{M}^{2}_{T}(\R)} = \sup_{\psi \in B} \norm{I^{w}(R' \psi)}^{2}_{\mathcal{M}^{2}_{T}(\R)} \leq 4T \, \sup_{\psi \in B} \norm{R' \psi}^{2}_{w,T}.$$
And hence $I^{s}$ is continuous for  $\Lambda^{2}_{s}(T)$ and $\mathcal{M}^{2}_{T}(\Psi'_{\beta})$ equipped with either the topology of simple or of bounded convergence. 
\end{prf}

\subsection{Properties of the Strong Stochastic Integral}\label{subsubsectionSPSSI}

In this section we prove some further properties of the strong stochastic integral. Thanks to the weak-strong compatibility given in \eqref{weakStrongCompatibilityStochIntegral}, we will see that most of the properties of the weak integral can be ``transferred'' to the strong integral. 

\begin{prop} \label{propImageStrongIntegralUnderContinuousOperator} Let $\Upsilon$ be a quasi-complete, bornological, nuclear space and let $S \in \mathcal{L}(\Psi'_{\beta},\Upsilon'_{\beta})$. Then, for each $R \in \Lambda^{2}_{s}(\Psi,M;T)$, we have $S \circ R \defeq \{ S \circ R(r,\omega,u) : r \in [0,T], \omega \in \Omega, u \in U\} \in \Lambda^{2}_{s}(\Upsilon,M;T)$, and moreover $\Prob$-a.e., we have
\begin{equation} \label{imageStrongIntegralUnderContinuousOperator}
I^{s}_{t}(S \circ R) = S \left( I^{s}_{t}(R) \right), \quad \forall \, t \in [0,T]. 
\end{equation}  
\end{prop}
\begin{prf}
First, since for each $(r,\omega,u) \in [0,T] \times \Omega \times U$, we have $R(r,\omega,u) \in \mathcal{L}(\Phi'_{q_{r,u}},\Psi'_{\beta})$ and $S \in \mathcal{L}(\Psi'_{\beta},\Upsilon'_{\beta})$, it follows that $S \circ R(r,\omega,u) \in \mathcal{L}(\Phi'_{q_{r,u}},\Upsilon'_{\beta})$. Now, let $\phi \in \Phi$ and $\upsilon \in \Upsilon$. As $S' \upsilon \in \Psi$, Definition \ref{integrandsStrongIntegSquareMoments}(2) applied to $R$ implies that the mapping $[0,T] \times \Omega \times U \rightarrow \R_{+}$ given by 
$$(r,\omega,u) \mapsto q_{r,u}((S \circ R(r,\omega,u))' \upsilon, \phi)= q_{r,u}(R(r,\omega,u)' S' \upsilon, \phi),$$ 
is $\mathcal{P}_{T} \otimes \mathcal{B}(U)$-measurable. Finally, as $S' \upsilon \in \Psi$ for every $\upsilon \in \Upsilon$, Definition \ref{integrandsStrongIntegSquareMoments}(3) applied to $R$ implies that 
\begin{equation*} 
\Exp \int_{0}^{T} \int_{U} q_{r,u}((S \circ R(r,u))' \upsilon)^{2} \mu(du) \lambda(dr) 
= \Exp \int_{0}^{T} \int_{U} q_{r,u}(R(r,u)' S' \upsilon)^{2} \mu(du) \lambda(dr) < \infty,
\end{equation*} 
for every $\upsilon \in \Upsilon$. Therefore, $S \circ R \in \Lambda^{2}_{s}(\Upsilon,M;T)$. 

Now, note that \eqref{weakStrongCompatibilityStochIntegral} implies that for all $\upsilon \in \Upsilon$, for $\Prob$-a.e. $\omega \in \Omega$ we have 
\begin{equation*}
I^{s}_{t}(S \circ R)(\omega)[\upsilon] = I^{w}_{t}(R' \circ S' \upsilon)(\omega)= I^{s}_{t}(R)(\omega)[S' \upsilon]= S \left( I^{s}_{t}(R)(\omega) \right)[\upsilon], \quad \forall \, t \in [0,T].
\end{equation*}   
Therefore, we have that for all $\upsilon \in \Upsilon$, $I^{s}(S \circ R)[\upsilon] = S \left( I^{s}(R) \right)[\upsilon]$ are indistinguishable processes. Then, Proposition \ref{propCondiIndistingProcess} shows that the $\Psi'_{\beta}$-valued processes $I^{s}(S \circ R)$ and $S \left( I^{s}(R) \right)$ are indistinguishable. This shows \eqref{imageStrongIntegralUnderContinuousOperator}. 
\end{prf}
 
\begin{prop} \label{propStrongIntegralInSubintervalAndRandomSubset}
Let $0 \leq s_{0} < t_{0} \leq T$ and $F_{0} \in \mathcal{F}_{s_{0}}$. Then, for every $R \in \Lambda^{2}_{s}(T)$, $\Prob$-a.e. we have 
\begin{equation} \label{strongIntegralInSubintervalAndRandomSubset}
I^{s}_{t}(\mathbbm{1}_{]s_{0},t_{0}]\times F_{0}} R)= \mathbbm{1}_{F_{0}} \left( I^{s}_{t \wedge t_{0}}(R)-I^{s}_{t \wedge s_{0}}(R) \right), \quad \forall \, t \in [0,T]. 
\end{equation}
\end{prop}
\begin{prf}
Let $R \in \Lambda^{2}_{s}(T)$. Then, it is easy to see that $\mathbbm{1}_{]s_{0},t_{0}] \times F_{0}} R \in \Lambda^{2}_{s}(T)$ and hence its strong stochastic integral exists. Now, let $\psi \in \Psi$. It follows from Theorem \ref{strongIntegrandsAsOperatorsToWeakIntegrans} that $R' \psi \in \Lambda^{2}_{w}(T)$. Then, from Proposition \ref{propWeakIntegralInSubintervalAndRandomSubset} there exists $\Gamma_{\psi} \subseteq \Omega$, such that $\Prob (\Gamma_{\psi})=1$ and for each $\omega \in \Gamma_{\psi}$, 
\begin{equation} \label{weakIntegralInSubintervalAndRandomPointwise}
I^{w}_{t}(\mathbbm{1}_{]s_{0},t_{0}]\times F_{0}} R' \psi)(\omega)= \mathbbm{1}_{F_{0}} \left( I^{w}_{t \wedge t_{0}}(R'\psi)(\omega)-I^{w}_{t \wedge s_{0}}(R'\psi)(\omega) \right), \quad \forall \, t \in [0,T].
\end{equation}

On the other hand, it follows from \eqref{weakStrongCompatibilityStochIntegral} that there exists $\Omega_{\psi} \subseteq \Omega$, with $\Prob (\Omega_{\psi})=1$, such that for each $\omega \in \Omega_{\psi}$, we have
\begin{equation} \label{subIntervalWeakStrongComp1}
I^{s}_{t}(\mathbbm{1}_{]s_{0},t_{0}]\times F_{0}} R)(\omega)[\psi] =I^{w}_{t}(\mathbbm{1}_{]s_{0},t_{0}]\times F_{0}} R' \psi)(\omega), \quad \forall \, t \in [0,T],
\end{equation}
\begin{equation} \label{subIntervalWeakStrongComp2}
I^{s}_{t \wedge t_{0}}(R)(\omega)[\psi]- I^{s}_{t \wedge s_{0}}(R)(\omega)[\psi] = I^{w}_{t \wedge t_{0}}(R' \psi)(\omega) - I^{w}_{t \wedge s_{0}}(R' \psi)(\omega),  \quad \forall \, t \in [0,T]. 
\end{equation}
Let $\Theta_{\psi}= \Gamma_{\psi} \cap \Omega_{\psi}$. Then, $\Prob (\Theta_{\psi})=1$. Moreover, from \eqref{weakIntegralInSubintervalAndRandomPointwise}, \eqref{subIntervalWeakStrongComp1} and \eqref{subIntervalWeakStrongComp2}, for every $\omega \in \Theta_{\psi}$ it follows that
\begin{equation*}
I^{s}_{t}(\mathbbm{1}_{]s_{0},t_{0}]\times F_{0}} R)(\omega)[\psi] 
= I^{s}_{t \wedge t_{0}}(R)(\omega)[\psi]- I^{s}_{t \wedge s_{0}}(R)(\omega)[\psi],  \quad \forall \, t \in [0,T].
\end{equation*}
Thus, for every $\psi \in \Psi$, $I^{s}(\mathbbm{1}_{]s_{0},t_{0}]\times F_{0}} R)[\psi]$ and $I^{s}_{\cdot \wedge t_{0}}(R)[\psi]- I^{s}_{\cdot \wedge s_{0}}(R)[\psi]$ are indistinguishable processes. But as the $\Psi'_{\beta}$-valued processes $I^{s}(\mathbbm{1}_{]s_{0},t_{0}]\times F_{0}} R)$ and $I^{s}_{\cdot \wedge t_{0}}(R)- I^{s}_{\cdot \wedge s_{0}}(R)$ are regular and c\`{a}dl\`{a}g, it follows from Proposition \ref{propCondiIndistingProcess} that they are indistinguishable. This shows \eqref{strongIntegralInSubintervalAndRandomSubset}.   
\end{prf}

\begin{prop} \label{propStoppedIntegralStrongCase}
Let $R \in \Lambda^{2}_{s}(T)$ and $\sigma$ be an $\{\mathcal{F}_{t}\}$-stopping time such that $\Prob (\sigma \leq T)=1$. Then, $\Prob$-a.e. 
\begin{equation} \label{stoppedIntegralStrongCase}
I^{s}_{t}(\mathbbm{1}_{[0,\sigma]} R )= I^{s}_{t \wedge \sigma}(R), \quad \forall \, t \in [0,T].
\end{equation}
\end{prop}
\begin{prf}
The proof follows from Proposition \ref{propStoppedIntegralWeakCase}, Theorem \ref{strongIntegrandsAsOperatorsToWeakIntegrans} and similar arguments to those used in Proposition \ref{propStrongIntegralInSubintervalAndRandomSubset}. 
\end{prf}

\begin{rema}
An analogue of Proposition \ref{propDecompWeakIntegralSumIndepMartValMeasu} is also valid for the strong integral. We leave to the reader the task of stating and proving it using the techniques developed in this section. 
\end{rema}

\subsection{Extension of the Strong stochastic Integral}\label{subSectionECSI}

We now proceed to extend the strong stochastic integral to a larger class of integrands. 

\begin{defi} \label{integrandsExtStrongIntegAlmostSureSquareMoments}
Let $\Lambda_{s}(\Psi,M;T)$ denote the collection of families $R=\{R(r,\omega,u): r \in [0,T], \omega \in \Omega, u \in U \}$ of operator-valued maps satisfying the following conditions:
\begin{enumerate}
\item $R(r,\omega,u) \in \mathcal{L}(\Phi'_{q_{r,u}},\Psi'_{\beta})$, for all $r \in [0, T]$, $\omega \in \Omega$, $u \in U$, 
\item $R$ is \emph{$q_{r,u}$-predictable}, i.e. for each $\phi \in \Phi$, $\psi \in \Psi$, the mapping $[0,T] \times \Omega \times U \rightarrow \R_{+}$ given by $(r,\omega,u) \mapsto q_{r,u}(R(r,\omega,u)' \psi, \phi)$ is $\mathcal{P}_{T} \otimes \mathcal{B}(U)$-measurable.
\item For every $\psi \in \Psi$,   
\begin{equation} \label{almostSureSecondMomentExtIntegrandsStrongIntg}
\Prob \left( \omega \in \Omega: \int_{0}^{T} \int_{U} q_{r,u}(R(r,\omega,u)'\psi)^{2} \mu(du) \lambda(dr) < \infty \right)=1.
\end{equation} 
\end{enumerate}
\end{defi}

\begin{rema}
The class $\Lambda_{s}(\Psi,M;T)$ generalizes considerably the class of extended stochastic integrands in \cite{BojdeckiJakubowski:1990} (see Definition 2.6 there). Indeed, to the extent of our knowledge $\Lambda_{s}(\Psi,M;T)$ is one of the largest classes of integrands considered in the literature of stochastic integration in duals of nuclear spaces.      
\end{rema}

Again, when it is not necessary to give emphasis to the dependence of the space $\Lambda_{s}(\Psi,M;T)$ with respect to $\Psi$ and $M$, we denote this space by $\Lambda_{s}(T)$. One can easily check that $\Lambda_{s}(T)$ is a linear space. Moreover, $\Lambda^{2}_{s}(T) \subseteq \Lambda_{s}(T)$. 

We proceed to construct the strong stochastic integral for the integrands belonging to $\Lambda_{s}(T)$. We start with the following result that is the analogue of Theorem \ref{strongIntegrandsAsOperatorsToWeakIntegrans} for the elements of $\Lambda_{s}(T)$. 

\begin{theo} \label{localMomentsStrongIntegrandsAsOperatorsToLocalWeakIntegrands}
The mapping $\Delta': \Lambda_{s}(T) \rightarrow \mathcal{L}(\Psi,\Lambda^{2,loc}_{w}(T))$ given by 
\begin{equation} \label{embeddingIntegralsAndOperatorsToWeakIntegrands}
R \mapsto \left(\psi \mapsto R'\psi \defeq \{ R(r,\omega,u)'\psi: r \in [0,T], \omega \in \Omega, u \in U \} \right),
\end{equation}
is an injective linear operator.  
\end{theo}   
\begin{prf} The proof follows from similar arguments to those used in the proof of Theorem \ref{strongIntegrandsAsOperatorsToWeakIntegrans} and hence we will mention only the main points. 

First, note that for every $R \in \Lambda^{2}_{s}(T)$ the properties listed in Definition \ref{integrandsExtStrongIntegAlmostSureSquareMoments} imply that the map $\psi \mapsto R'\psi$ from $\Psi$ into $\Lambda^{2,loc}_{w}(T)$ is well-defined. Moreover, we can easily see that it is also linear; indeed this follows from the linearity of each operator $R(r,\omega,u)' \in \mathcal{L}(\Psi, \Phi_{q_{r,u}})$. 

We need to prove that $\psi \mapsto R'\psi$ is also continuous. First, we can show that $\psi \mapsto R' \psi$ is sequentially closed, this by following similar arguments to those used in Step 1 of the proof of Theorem \ref{strongIntegrandsAsOperatorsToWeakIntegrans} but with the norm $\norm{\cdot}_{w,T}$ there being replaced by the metric $d_{\Lambda}$ defined in the proof of Proposition \ref{extendedClassWeakIntegrandsInMetrizable}. Then, the closed graph theorem shows that $\psi \mapsto R'\psi$ is continuous. Therefore the mapping $\Delta'$ is well-defined. The proof that $\Delta'$ is linear and injective is exactly as in the proof of Theorem \ref{strongIntegrandsAsOperatorsToWeakIntegrans}. 
\end{prf}

\begin{rema} \label{remaMapDeltaStrongIntegralIsNotSurjective}
We do not know if the map $\Delta'$ defined in Theorem \ref{localMomentsStrongIntegrandsAsOperatorsToLocalWeakIntegrands} is surjective. This is because as the space $\Lambda^{2,loc}_{w}(T)$ is not in general locally convex (see Remark \ref{spaceExtendedWeakIntegrandsNonLocallyConvex}), it is not clear how 
the arguments used in Step 2 of the proof of Theorem \ref{strongIntegrandsAsOperatorsToWeakIntegrans} can be modified for elements of $\mathcal{L}(\Psi,\Lambda^{2,loc}_{w}(T))$. 
\end{rema}
 
The existence of the extension of the strong stochastic integral to the elements of $\Lambda_{s}(T)$ is provided in the following result. 

\begin{theo} \label{existenceStrongStochIntegForAlmostSureSecondMomentIntegrands}
Let $R \in \Lambda_{s}(T)$. There exist a unique (up to indistinguishable versions) process $\hat{I}^{s}(R)=\{ \hat{I}^{s}_{t}(R) \}_{t \in [0,T]} \in \mathcal{M}^{2,loc}_{T}(\Phi'_{\beta})$, such that for all $\psi \in \Psi$, $\Prob$-a.e. 
\begin{equation} \label{weakStrongCompatibilityAlmostSureIntegrands}
\hat{I}^{s}_{t}(R)[\psi]=\hat{I}^{w}_{t}(R' \psi), \quad \forall \, t \in [0,T]. 
\end{equation} 
where for each $\psi \in \Psi$, the stochastic process in the right-hand side of \eqref{weakStrongCompatibilityAlmostSureIntegrands} corresponds to the weak stochastic integral of $R' \psi \in \Lambda^{2,loc}_{w}(T)$. 
\end{theo}
\begin{prf}  
The proof follows from similar arguments to those used in the proof of Theorem  \ref{existenceStrongStochasticIntegSquareMoments} by applying Theorem \ref{localMomentsStrongIntegrandsAsOperatorsToLocalWeakIntegrands} and Proposition \ref{propCondiIndistingProcess}. 
\end{prf}

\begin{defi} \label{defiStochasticIntegralAlmostSureSecondMmentsStrongCase} For every $R \in \Lambda_{s}(T)$, we will call the process $\hat{I}^{s}(R)$ given in Theorem  \eqref{existenceStrongStochIntegForAlmostSureSecondMomentIntegrands}  the \emph{strong stochastic integral} of $R$. We will sometimes denote the process $\hat{I}^{s}(R)$ by $\left\{ \int^{t}_{0} \int_{U} R (r,u) M (dr, du): t \in [0,T] \right\}$.  The map $\hat{I}^{s}: \Lambda_{s}(T) \rightarrow \mathcal{M}^{2,loc}_{T}(\Psi'_{\beta})$ given by $R \mapsto \hat{I}^{s}(R)$, will be called the \emph{extended strong integral mapping}.
\end{defi}

By using the \emph{weak-strong compatibility}  \eqref{weakStrongCompatibilityAlmostSureIntegrands} and the same arguments in the proof of Proposition \ref{linearityStrongIntegralLinearOperators} we can show the following result. 

\begin{prop}\label{propExtendedStrongIntegMappingIsLinear} The extended strong integral mapping $\hat{I}^{s}: \Lambda_{s}(T) \rightarrow \mathcal{M}^{2,loc}_{T}(\Psi'_{\beta})$ is linear. 
\end{prop}

From \eqref{weakStrongCompatibilityAlmostSureIntegrands} and the properties of the weak stochastic integral for integrands in $\Lambda^{2,loc}_{w}(T)$ (see Proposition \ref{propPropertWeakStochIntegExtenToIntegLocalSecoMomen}) we can show that the properties of the stochastic integral for integrands in $\Lambda^{2}_{s}(T)$ (see Section \ref{subsubsectionSPSSI}) are also satisfied for the strong stochastic integral for integrands in $\Lambda_{s}(T)$. We summarize this in the following result:

\begin{prop} \label{propStochasticIntegralExtendsToIntegrandsLocalSecondMoments}
Let $R \in \Lambda_{s}(T)$. Then, all the assertions in Propositions \ref{propImageStrongIntegralUnderContinuousOperator}, \ref{propStrongIntegralInSubintervalAndRandomSubset} and  \ref{propStoppedIntegralStrongCase} are true for the strong stochastic integral $\hat{I}^{s}(R)$ of $R$. 
\end{prop}

\section{Stochastic Evolution Equations in Duals of Nuclear Spaces}
\label{sectionSEEDNS}

\subsection{Semigroups of Linear Operators in Locally Convex Spaces}\label{sectionSLOLCS}
 
Only for this section, let $\Psi$ denote a quasi-complete locally convex space. In this section we review basic properties  of $(C_{0},1)$-semigroups on $\Psi$.  This class of semigroups was introduced by Babalola  \cite{Babalola:1974}. 

We recall that a family $\{ S(t) \}_{t \geq 0} \subseteq \mathcal{L}(\Psi,\Psi)$  is called a \emph{$C_{0}$-semigroup} on $\Psi$ if: \begin{inparaenum}[(i)] \item $S(0)=I$, $S(t)S(s)=S(t+s)$ for all $t, s \geq 0$, and \item $\lim_{t \rightarrow s}S(t) \psi = S(s) \psi$, for all $s \geq 0$ and any $\psi \in \Psi$. 
\end{inparaenum} The \emph{infinitesimal generator} $A$ of a $C_{0}$-semigroup $\{ S(t) \}_{t \geq 0}$ on $\Psi$ is defined by 
$$ A \psi = \lim_{h \downarrow 0} \frac{S(h) \psi -\psi}{h} \quad \mbox{(limit in $\Psi$)},$$
whenever the limit exists, the domain of $A$ being the set $\mbox{Dom}(A) \subseteq \Psi$ for which the above limit exists. 

A $C_{0}$-semigroup $\{ S(t) \}_{t \geq 0}$ on $\Psi$ is said to be a \emph{$(C_{0},1)$-semigroup} if for each continuous semi-norm $p$ on $\Psi$ there exist some $\vartheta_{p} \geq 0$ and a continuous semi-norm $q$ on $\Psi$ such that $p(S(t)\psi) \leq e^{\vartheta_{p} t} q(\psi)$, for all $t \geq 0$, $\psi \in \Psi$. Furthermore, if the above is satisfied with $\vartheta_{p} = 0$  then we say that $\{ S(t) \}_{t \geq 0}$ is an \emph{equicontinuous semigroup}. Hence,  every equicontinuous semigroup is a $(C_{0},1)$-semigroup but the converse is not true in general (see \cite{Babalola:1974} p.177). 

Some of the most important properties of $(C_{0},1)$-semigroup for our study of solutions to stochastic evolution equations are given in the following result. 

\begin{theo}[\cite{Babalola:1974}, Theorems 2.3 and 2.6] \label{theoExtensC01SemigroupsToBanachSpaces}
Let $\{ S(t) \}_{t \geq 0}$ be a $C_{0}$-semigroup on $\Psi$. Then, $\{ S(t) \}_{t \geq 0}$ is a $(C_{0},1)$-semigroup on $\Psi$ if and only if  there exists a family $\Pi$ of semi-norms generating the topology on $\Psi$ such that for each $p \in \Pi$ there exist $M_{p} \geq 1$ and  $\theta_{p} \geq 0$ such that 
$$ p(S(t)\psi) \leq M_{p} e^{\theta_{p} t} p(\psi), \quad \mbox{for all } t \geq 0, \, \psi \in \Psi,$$ 
(with $M_{p}=1$, $\theta_{p} = 0$ if $\{ S(t) \}_{t \geq 0}$ is equicontinuous). In that case, for each $p \in \Pi$, there exists a $C_{0}$-semigroup $\{ S_{p}(t) \}_{t \geq 0}$ ($C_{0}$-semigroup of contractions if $\{ S(t) \}_{t \geq 0}$ is equicontinuous) on the Banach space $\Psi_{p}$ such that
\begin{equation} \label{defExtenSemiGroupToBanach}
 S_{p}(t) i_{p} \psi= i_{p} S(t) \psi, \quad \forall \, \psi \in \Psi, \, t \geq 0.
 \end{equation}
\end{theo}

The following result will be important for the existence and uniqueness of solutions to stochastic evolution equations (see Section \ref{sectionEUWMS}). 

\begin{prop} \label{propExistFamiHilbSeminormExtenSemiGroup} 
Let $\{ S(t) \}_{t \geq 0}$ be a $(C_{0},1)$-semigroup on $\Psi$ and let $\Pi$ be a family of continuous semi-norms on $\Psi$ satisfying the conditions in Theorem \ref{theoExtensC01SemigroupsToBanachSpaces}. Assume also that for each $p \in \Pi$ the Banach space $\Psi_{p}$ is separable. Then, for each $p \in \Pi$ there exists a continuous Hilbertian semi-norm $q$ on $\Psi$, $q \leq p$ and a $C_{0}$-semigroup $\{ S_{q}(t)\}_{t \geq 0}$ on the Hilbert space $\Psi_{q}$, such that   
\begin{equation} \label{propSemiGropExtToHilbSpacePsi1}
 S_{p}(t) i_{q,p} \psi= i_{q,p} S_{q}(t) \psi, \quad \forall \, \psi \in \Psi_{p}, \, t \geq 0, 
\end{equation}
and 
\begin{equation} \label{propSemiGropExtToHilbSpacePsi2}
S_{q}(t) i_{q} \psi= i_{q} S(t) \psi, \quad \forall \, \psi \in \Psi, \, t \geq 0.
\end{equation}
\end{prop} 
\begin{prf}
Let $p \in \Pi$ and let $\{ S_{p}(t)\}_{t \geq 0}$ be the $C_{0}$-semigroup on $\Psi_{p}$ satisfying \eqref{defExtenSemiGroupToBanach}. Because $\Psi_{p}$ is a separable Banach space, then by Theorem 1.3 in \cite{vanNeerven:1999} there exist a separable Hilbert space $(H, \norm{\cdot}_{H})$, a continuous dense embedding $j_{H,p}: \Psi_{p} \rightarrow H$, and a $C_{0}$-semigroup $\{ T_{H}(t)\}_{t \geq 0}$ on $H$ such that 
\begin{equation} \label{defExtC0SemigrHilbSpaH1}
T_{H}(t)  j_{H,p} \psi = j_{H,p}  S_{p} \psi, \quad \forall \, t \geq 0, \, \psi \in \Psi. 
\end{equation}  
Therefore, $j_{H}: \Psi \rightarrow H$ given by $j_{H} = j_{H,p} i_{p}$ is a continuous dense embedding. Moreover, by \eqref{defExtenSemiGroupToBanach} and \eqref{defExtC0SemigrHilbSpaH1} we have that 
\begin{equation} \label{defExtC0SemigrHilbSpaH2}
T_{H}(t)  j_{H} \psi=j_{H,p}  S_{p}(t)  i_{p} \psi  = j_{H}  S(t) \psi, \quad \forall \, t \geq 0, \, \psi \in \Psi. 
\end{equation}
Let $B_{H}$ denote the unit ball in $H$. The continuity of $j_{H}$ implies that $j_{H}^{-1}(B_{H})$ is a neighborhood of zero of $\Psi$. Therefore, $q: \Psi \rightarrow \R$ given by $q(\psi)= \norm{j_{H} \psi}_{H}$ $\forall \, \psi \in \Psi$ is a continuous Hilbertian semi-norm on $\Psi$. Hence, the map $j_{H}$ defines an isometric isomorphism between the pre-Hilbert spaces $(\Psi / \mbox{ker}(q), q)$ and $(j_{H} \Psi, \norm{ \cdot }_{H})$. 

Let $t \geq 0$. Note that from \eqref{defExtC0SemigrHilbSpaH2} we have that $T_{H}(t)(j_{H} \Psi) \subseteq j_{H} \Psi$. Therefore, $T_{H}(t)$ restricts to a continuous and linear operator on $(j_{H} \Psi, \norm{ \cdot }_{H})$ and hence in $(\Psi / \mbox{ker}(q), q)$ because these spaces are isometrically isomorphic. We denote by  $S_{q}(t)$ the continuous and linear extension to $\Psi_{q}$ of the restriction of $T_{H}(t)$ to $(\Psi / \mbox{ker}(q), q)$. Then, $\{ S_{q}(t)\}_{t \geq 0}$ is a $C_{0}$-semigroup on $\Psi$ and from \eqref{defExtC0SemigrHilbSpaH1} and \eqref{defExtC0SemigrHilbSpaH2} it satisfies \eqref{propSemiGropExtToHilbSpacePsi1} and \eqref{propSemiGropExtToHilbSpacePsi2}. 
\end{prf}

\begin{rema} \label{remaHilbSemNormExtSemiGroupToHilbert}
Even in the case where $\Psi$ is nuclear, we do not know if it is possible to choose the family $\Pi$ of semi-norms on $\Psi$ given in Theorem \ref{theoExtensC01SemigroupsToBanachSpaces} to be such that each $p \in \Pi$ is Hilbertian. This is assumed for example in \cite{Ding:1999} as part of the definition of $(C_{0},1)$-semigroup.  A partial result in this direction is given in Proposition \ref{propExistFamiHilbSeminormExtenSemiGroup} where it has been shown that there exists a non-empty family of continuous Hilbertian semi-norms on $\Psi$ for which \eqref{propSemiGropExtToHilbSpacePsi2} is satisfied. However, this family does not necessarily generate the topology on $\Psi$. 
\end{rema}

Let $\{ S(t) \}_{t \geq 0}$ be a $C_{0}$-semigroup on $\Psi$ with generator $A$. If the space $\Psi$ is reflexive, then the family $\{ S(t)' \}_{t \geq 0}$ of  dual operators is a $C_{0}$-semigroup on $\Psi'_{\beta}$ with generator $A'$, that we call the \emph{dual semigroup} and the \emph{dual generator} respectively. Moreover, if $\{ S(t) \}_{t \geq 0}$ is equicontinuous then $\{ S(t)' \}_{t \geq 0}$ is also equicontinuous (see \cite{Komura:1968}, Theorem 1 and its Corollary).    
However, even when $\Psi$ is reflexive it is not true in general that the dual semigroup of a $(C_{0},1)$-semigroup is a $(C_{0},1)$-semigroup on $\Psi'_{\beta}$ (see  \cite{Babalola:1974}, Section 6).

\subsection{Stochastic Evolution Equations: The General Setting} \label{sectionGSSEE}

In this section we will introduce the general model of stochastic evolution equations in the dual of a nuclear space driven by a nuclear cylindrical martingale-valued measure. 

Let $\Phi$ be a locally convex space and $\Psi$ be a quasi-complete, bornological, nuclear space, both defined over $\R$. Let $U$ be a topological space. We are concerned with the following class of stochastic evolution equations
\begin{equation}\label{generalFormSEE}
d X_{t}= (A'X_{t}+ B (t,X_{t})) dt+\int_{U} F(t,u,X_{t}) M (dt,du), \quad \mbox{for }t \geq 0,
\end{equation}
where we will assume the following: 

\begin{assu} \label{assumptionsCoefficients} \hfill 

\textbf{(A1)} $A$ is the infinitesimal generator of a  $(C_{0},1)$-semi-group $\{S(t)\}_{t\geq 0}$ on $\Psi$. 

\textbf{(A2)} $M$ is a nuclear cylindrical martingale-valued measure on $\R_{+} \times \mathcal{R}$, where $\mathcal{R}$ is a ring $\mathcal{R}\subseteq \mathcal{B}(U)$ that generates the Borel $\sigma$-algebra $\mathcal{B}(U)$ of the topological space $U$, and the covariance of $M$ is determined by the measure $\lambda=\mbox{Leb}$ on $\R_{+}$, a $\sigma$-finite Borel measure $\mu$ on $U$, and the semi-norms $\{q_{r,u}: r \in \R_{+}, u \in U\}$; all satisfying the conditions in Definition \ref{nuclearMartingaleValMeasDualSpace} and Assumption \ref{assumpDenseSubsetSemiNormsMartValuedMeasures}.

\textbf{(A3)} $B:\R_{+} \times \Psi' \rightarrow \Psi'$ is such that the map $(r,g) \mapsto B(r,g)[\psi]$ is $ \mathcal{B}(\R_{+}) \otimes \mathcal{B} (\Psi'_{\beta})$-measurable, for every $\psi \in \Psi$.

\textbf{(A4)} $F= \{F(r,u,g): r \in \R_{+}, u \in U, g \in \Psi'\}$ is such that
\begin{enumerate}[label=(\alph*)]
\item $ F(r,u,g) \in \mathcal{L} (\Phi'_{q_{r,u}}, \Psi'_{\beta})$, $\forall r \geq 0$, $u \in U$, $g \in \Psi'$. 
\item The mapping $(r,u,g)\mapsto q_{r,u}(F(r,u,g)'\phi,\psi)$ is $ \mathcal{B}(\R_{+}) \otimes \mathcal{B}(U) \otimes \mathcal{B}(\Psi'_{\beta})$-measurable, for every $\phi \in \Phi$, $\psi \in \Psi$.
\end{enumerate}
\end{assu}

Note that $\Psi$ being reflexive, assumption (A1) implies that $A'$ is the infinitesimal generator of the dual semi-group $\{S(t)'\}_{t\geq 0}$ and this last is a $C_{0}$-semigroup on $\Psi'_{\beta}$.

We are interested in to studying weak and mild solutions to \eqref{generalFormSEE}. The precise formulation of these types of solutions is given below. 

\begin{defi} \label{defiWeakSolution}
A $\Psi'_{\beta}$-valued regular and predictable process $X=\{X_{t}\}_{t \geq 0}$ is called a \emph{weak solution} to \eqref{generalFormSEE} if
\begin{enumerate}[label=(\alph*)]
\item For every $t >0$, $X$, $B$ and $F$ satisfy  the following conditions:
\begin{gather*}
\Prob \left( \omega \in \Omega: \int^{t}_{0} \abs{X_{r}(\omega)[\psi]}dr < \infty\right) = 1, \quad \forall \, \psi \in \Psi. \\
\Prob \left( \omega \in \Omega: \int^{t}_{0} \abs{B (r,X_{r}(\omega))[\psi]}dr < \infty\right) = 1, \quad \forall  \, \psi \in \Psi. \\
\Prob \left( \omega \in \Omega: \int^{t}_{0} \int_{U} q_{r,u} (F(r,u,X_{r}(\omega))'\psi)^{2} \mu (du) dr < \infty \right) = 1, \quad \forall \,  \psi \in \Psi. 
\end{gather*}
\item For every $\psi \in \mbox{Dom}(A)$ and every $t \geq 0$, $\Prob$-a.e.
\begin{eqnarray} 
X_{t}[\psi] & = & X_{0}[\psi]+ \int^{t}_{0}(X_{r}[A \psi]+ B (r,X_{r})[\psi])dr  \label{equationWeakSolution}\\
& + & \int^{t}_{0} \int_{U} F(r,u,X_{r})'\psi M(dr,du), \nonumber
\end{eqnarray}
where the first integral in the right-hand side of \eqref{equationWeakSolution} is a Lebesgue integral that is defined for each $\psi \in \Psi$ for $\Prob$-a.e. $\omega \in \Omega$. The second integral in the right-hand side of \eqref{equationWeakSolution} is the weak stochastic integral of $F'\psi= \{ F(r,u,X_{r}(\omega))'\psi:  r \in [0,t], \omega \in \Omega, u \in U \} \in \Lambda^{2,loc}_{w}(t)$, and is well-defined for all $\psi \in \Psi$.
\end{enumerate}
\end{defi}

The proof of the following result can be carried out from standard arguments. We leave the details to the reader.

\begin{prop} \label{remaIntegralsWeakSoluWellDefined}
The assumptions (A1)-(A4) together with the conditions (a) of Definition \ref{defiWeakSolution}
are sufficient to guarantee the existence of all the integrals in \eqref{equationWeakSolution}. 
\end{prop}

\begin{defi} \label{defiMildSolution}
A $\Psi'_{\beta}$-valued regular and predictable process $X=\{X_{t}\}_{t \geq 0}$ is called a \emph{mild solution} to \eqref{generalFormSEE} if
\begin{enumerate}[label=(\alph*)]
\item For every $t \geq 0$, for all $\psi \in \Psi$, 
\begin{gather*}
\Prob \left( \omega \in \Omega: \int^{t}_{0} \abs{S(t-r)'B(r,X_{r}(\omega))[\psi]}dr < \infty \right)= 1. \\
\Prob \left( \omega \in \Omega: \int^{t}_{0} \int_{U} q_{r,u}(F(r,u,X_{r}(\omega))' S(t-r)\psi )^{2} \mu(du) dr < \infty \right) = 1.
\end{gather*}
\item For every $ t \geq 0$, $\Prob$-a.e.
\begin{equation} \label{equationMildSolution}
X_{t} = S(t)'X_{0}+ \int^{t}_{0} S(t-r)' B(r,X_{r})dr  +  \int^{t}_{0} \int_{U} S(t-r)' F(r,u,X_{r}) M(dr,du), 
\end{equation}
where the first integral at the right-hand side of \eqref{equationMildSolution} is a $\Psi'_{\beta}$-valued regular, $\{ \mathcal{F}_{t} \}$-adapted process $\left\{ \int^{t}_{0} S(t-r)'B(r,X_{r})dr: t \geq 0 \right\}$ such that for all $t \geq 0$ and $\psi \in \Psi$, for $\Prob$-a.e. $\omega \in \Omega$,
\begin{equation} \label{defNonRandomConvolutionIntegral}
\left( \int^{t}_{0} S(t-r)'B(r,X_{r}(\omega))dr \right)[\psi]= \int^{t}_{0} S(t-r)' B(r,X_{r}(\omega))[\psi] dr,
\end{equation}
where for each $t \geq 0$, $\psi \in \Psi$, the integral on the right-hand side of \eqref{defNonRandomConvolutionIntegral} is the Lebesgue integral of the function $\ind{[0,t]}{\cdot}S(t-\cdot)'B(\cdot,X_{\cdot}(\omega))[\psi]$ defined on $[0,t]$ for $\Prob$-a.e. $\omega \in \Omega$.
The second integral at the right-hand side of \eqref{equationMildSolution} is the strong stochastic integral of 
$\{\ind{[0,t]}{r} S(t-r)'F(r,u,X_{r}(\omega)): r \in [0,t], \omega \in \Omega, u \in U \}$.
\end{enumerate}
\end{defi}

For the proof of the next proposition we will need to recall some properties of absolutely continuous functions. For $t>0$, let $A C_{t}$ denotes the linear space of all absolutely continuous functions on $[0,t]$ which are zero at $0$.
It is well-known (see Theorem 5.3.6 in \cite{BogachevMT}, p.339) that $G \in A C_{t}$ if and only if there exists an integrable function $g$ defined on $[0,t]$ such that:
\begin{equation} \label{defDensityAbsContFunct}
G(s)= \int^{s}_{0} g(r) dr, \quad \forall s \in [0,t].
\end{equation}
The space $A C_{t}$ is a Banach space equipped with the norm $\norm{\cdot}_{A C_{t}}$ given by $\norm{G}_{A C_{t}}= \int^{t}_{0} \abs{g(r)} dr$,  for $G \in A C_{t}$ with $g$ satisfying \eqref{defDensityAbsContFunct}.

\begin{prop} \label{remaIntegralsMildSoluWellDefined}
The assumptions (A1)-(A4) together with the conditions (a) of Definition \ref{defiMildSolution}
are sufficient to guarantee the existence of all the integrals in \eqref{equationMildSolution}. 
\end{prop}
\begin{prf}
We start with the existence of the process $\left\{ \int^{t}_{0} S(t-r)'B(r,X_{r})dr: t \geq 0 \right\}$. Fix $t \geq 0$. 
From the predictability of $X$ and (A3) it follows that  $(r,\omega) \mapsto B(r,X_{r}(\omega))[\psi]$ is $\mathcal{P}_{\infty}$-measurable, for every $\psi \in \Psi$. Then, for any $s \in [0,t]$, the continuity of  $r \mapsto \ind{[0,s]}{r} S(s-r)\psi$, implies that 
$$ (r,\omega) \mapsto  \ind{[0,s]}{r} B(r,X_{r}(\omega))[ S(s-r)\psi ]= \ind{[0,s]}{r} S(s-r)'B(r,X_{r}(\omega))[\psi],$$
is $\mathcal{P}_{s}$-measurable, for all $\psi \in \Psi$.

Now, for every $\psi \in \Psi$, let $\Omega_{t,\psi}=\{ \omega \in \Omega: \int^{t}_{0} \abs{S(t-r)'B(r,X_{r}(\omega))[\psi]}dr < \infty \}$. Note that from Definition \ref{defiMildSolution}(a) it follows that $\Prob \left( \Omega_{t,\psi} \right)=1$. 
Let $J_{t}: \Psi \mapsto L^{0}(\Omega, \mathcal{F}, \Prob; A C_{t})$ given for every $\psi \in \Psi$ by
\begin{equation} \label{defMapJtProofRegulTheoNonRandomInteg}
J_{t}(\psi)(\omega)(s)= 
\begin{cases}
\int^{s}_{0} S(s-r)'B(r,X_{r}(\omega))[\psi]dr , & \mbox{for } \omega \in \Omega_{t,\psi}, \, s \in ]0,t], \\
0, & \mbox{elsewhere.}
\end{cases}
\end{equation}
It is clear from arguments on the previous paragraphs that $J_{t}$ is well-defined and linear. Moreover, by using similar ideas to those in the Step 1 of the proof of Theorem \ref{strongIntegrandsAsOperatorsToWeakIntegrans}, it can be show that the map $J_{t}$ is sequentially closed. Then, as $\Psi$ is ultrabornological and $AC_{t}$ is a Banach space, the closed graph theorem  shows that $J_{t}$ is continuous. Hence, $J_{t}$ defines a cylindrical random variable such that $J_{t}: \Psi \mapsto L^{0}\ProbSpace$ is continuous. Then, the regularization theorem (see \cite{Ito}, Theorem 2.3.2) shows that there exists a $\Phi'_{\beta}$-valued regular random variable $\int^{t}_{0} S(t-r)'B(r,X_{r})dr$ that is a version of $J_{t}$, i.e. such that \eqref{defNonRandomConvolutionIntegral} is satisfied. Moreover, because for each $\psi \in \Psi$, $J_{t}(\psi)$ and hence $\int^{s}_{0} S(s-r)'B(r,X_{r}(\omega))[\psi]dr$ is $\mathcal{F}_{t}$-measurable, then the fact that $\int^{t}_{0} S(t-r)'B(r,X_{r})dr$ is a regular random variable implies that it is also $\mathcal{F}_{t}$-measurable. 

Now, for the stochastic integral $ \int^{t}_{0} \int_{U} S(t-r)' F(r,u,X_{r}) M(dr,du)$ to be well-defined, we have to check that for each $t \geq 0$, the integrand is an element of $\Lambda_{s}(t)$ (Definition \ref{integrandsExtStrongIntegAlmostSureSquareMoments}). 

Let $R=\{ R(r,\omega,u)\}$ be given by   
$$ R(r,\omega,u) = S(t-r)' F(r,u,X_{r}(\omega)), \quad \forall \, r \in [0,t], \, \omega \in \Omega, \, u \in U. $$
It is clear that $ R(r,\omega,u) \in \mathcal{L}(\Phi'_{q_{r,u}},\Psi'_{\beta})$, for each $r \in [0,t]$, $\omega \in \Omega$, $u \in U$. Now, because $X$ is predictable together with (A4) it follows that the map $(r,\omega,u) \mapsto q_{r,u}(F(r,u,X_{r})'\varphi, \phi)$ is $\mathcal{P}_{\infty} \otimes \mathcal{B}(U)$-measurable, for every $\phi \in \Phi$, $\varphi \in \Psi$. Then, by the continuity of the map $r \mapsto S(t-r)\psi$ for $r \in [0,t]$ and fixed $\psi \in \Psi$, it follows that the map 
$$ (r,\omega,u) \mapsto q_{r,u}(R(r,\omega,u)'\psi, \phi)= q_{r,u}(F(r,u,X_{r}(\omega))'S(t-r)\psi, \phi),$$
defined on $[0,t] \times \Omega \times U$ is $\mathcal{P}_{t} \otimes \mathcal{B}(U)$-measurable for each $\psi \in \Psi$. Finally, Definition \ref{defiMildSolution}(a) implies that $R$ satisfies \eqref{almostSureSecondMomentExtIntegrandsStrongIntg}. Therefore, $R \in \Lambda_{s}(t)$ and hence Theorem \ref{existenceStrongStochIntegForAlmostSureSecondMomentIntegrands} shows the existence of the stochastic integral $\int^{t}_{0} \int_{U} S(t-r)' F(r,u,X_{r}) M(dr,du)$.  Moreover, from \eqref{weakStrongCompatibilityAlmostSureIntegrands} the following holds for all $\psi \in \Psi$, $t \in [0,T]$, $\Prob$-a.e.  
\begin{equation}\label{weakStrongCompaStochasticConvolution}
\int^{t}_{0} \int_{U} S(t-r)' F(r,u,X_{r}) M(dr,du) [\psi] = \int^{t}_{0} \int_{U} F(r,u,X_{r})'S(t-r) \psi M(dr,du). 
\end{equation}
\end{prf}

\subsection{Equivalence Between Mild and Weak Solutions}\label{subSectionEBWMS}

In this section we provide sufficient conditions for the equivalence between mild and weak solutions. The main result of this section is the following:

\begin{theo} \label{theoEquiWeakMild}
Let $X=\{ X_{t} \}_{t \geq 0}$ be a $\Psi'_{\beta}$-valued regular and predictable process and assume that for every $T>0$, $X$, $B$ and $F$ satisfy:
\begin{equation} \label{assumpXMomentTheoWeakMildSol}
\Exp \int^{T}_{0} \abs{X_{r}[\psi]} dr < \infty, \quad \forall \, \psi \in \Psi.
\end{equation} 
\begin{equation} \label{assumpBMomentTheoWeakMildSol}
 \Exp \int^{T}_{0} \abs{B(r,X_{r}) [\psi]} dr < \infty, \quad \forall \psi \in \Psi.
\end{equation}
\begin{equation} \label{assumpFMomentTheoWeakMildSol}
\Exp \int^{T}_{0} \int_{U} q_{r,u} (F (r,u, X_{r})'\psi)^{2} \mu (du) dr < \infty, \quad \forall \psi \in \Psi. 
\end{equation}
Then, $X$ is a weak solution to \eqref{generalFormSEE} if and only if it is a mild solution to \eqref{generalFormSEE}.
\end{theo}

For our proof we benefit from ideas taken from Peszat and Zabczyk \cite{PeszatZabczyk} and Gorajski \cite{Gorajski:2014} on the proof of equivalence between weak and mild solutions on separable Hilbert spaces and in UMD Banach spaces respectively.   

To prove Theorem \ref{theoEquiWeakMild} we will need to make some technical preparations that for the convenience of the reader we wil present in the following three lemmas.  

\begin{lemm} \label{lemmEquaFubiniEquivWeakMild} Let $X=\{ X_{t} \}_{t \geq 0}$ be a $\Psi'_{\beta}$-valued regular and predictable process and assume that $F$ satisfies \eqref{assumpFMomentTheoWeakMildSol}. For every $\psi \in \mbox{Dom}(A)$ and $t >0$, the following identities holds $\Prob$-a.e.
\begin{flalign} \label{equaFubiniEquivWeakMild1}
& \int^{t}_{0} \left( \int^{s}_{0} \int_{U} S(s-r)'F (r,u, X_{r})M(dr,du) [A \psi]\right) ds \nonumber \\
& = \int^{t}_{0} \left( \int^{s}_{0} \int_{U} F (r,u,X_{r})' S(t-s) A \psi M (dr,du)\right) ds \nonumber  \\
& = \int^{t}_{0} \int_{U} F (r,u, X_{r})'S(t-r) \psi M (dr,du)- \int^{t}_{0} \int_{U} F(r,u, X_{r})'\psi M(dr,du)   
\end{flalign}
\end{lemm}
\begin{prf}
Fix $\psi \in \mbox{Dom}(A)$ and $t >0$. Consider the following families of Hilbert-space valued maps:
\begin{gather*}
Y_{1}(r,\omega,u,s)= \ind{[0,s]}{r} F (r,u, X_{r}(\omega))'S(t-s) A \psi, \\
Y_{2}(r,\omega,u,s)= \ind{[0,s]}{r} F (r,u, X_{r}(\omega))'S(s-r) A \psi.
\end{gather*}
for $r \in [0,t]$, $\omega \in \Omega$, $u \in U$ and $s \in [0,t]$.

Our first task is to verify that both $Y_{1}$ and $Y_{2}$ belong to $\Xi^{1,2}_{w} (t, [0,t])$ (see Definition \ref{defiClassIntegrandsStochFubiniTheo}) for $E=[0,t]$, $ \mathcal{E}= \mathcal{B}([0,t])$, $\varrho= \mbox{Leb}$. In that case, the Stochastic Fubini's Theorem (Theorem \ref{stochasticFubiniTheorem}) show that all the integrals in \eqref{equaFubiniEquivWeakMild1} exist. 

We start by proving that $Y_{1}$ satisfies the conditions of Definition \ref{defiClassIntegrandsStochFubiniTheo}. First, as $S(t-s) A  \psi \in \Psi$, $\forall \, s \in [0,t]$, by (A4)(a) it follows that $Y_{1}(r,\omega,u,s) \in \Phi_{q_{r,u}}$ for $(r,\omega,u,s) \in [0,t] \times \Omega \times U \times [0,t]$.

Now, let $\phi \in \Phi$. From the strong continuity of the semigroup $\{S(t)\}_{t \geq 0}$ it follows that the map $[0,t] \backin s \mapsto S(t-s) A\psi \in \Psi$ is continuous and therefore Borel measurable. This fact together with (A4)(b) and the predictability of $X$ implies that the mapping 
$$(r,\omega,u,s) \mapsto \ind{[0,s]}{r} q_{r,u}(F(r,u,X_{r}(\omega))'S(t-s)A \psi, \phi),$$ 
is $\mathcal{P}_{t} \otimes \mathcal{B}(U)\otimes \mathcal{B}([0,t])$-measurable. Finally, note that (A4), the predictability of $X$ and \eqref{assumpFMomentTheoWeakMildSol} implies that $\{ F(r,u,X_{r}(\omega)):  r \in [0,t], \omega \in \Omega, u \in U \} \in \Lambda^{2}_{s}(t)$. Then, 
from Theorem \ref{existenceStrongStochasticIntegSquareMoments} there exists a continuous Hilbertian semi-norm $p$ on $\Psi$ and $\tilde{F}_{X} \in \Lambda^{2}_{s}(p,t)$ such that $F(r,u,X_{r}(\omega))= i'_{p} \tilde{F}_{X}(r,\omega,u)$, for $\mbox{Leb} \otimes \Prob \otimes \mu$-a.e. $(r,\omega,u) \in [0,t] \times \Omega \times U$ satisfying \eqref{itoIsometryStrongIntegHilbSchmiIntegrands}.   

Now, as $\{S(t)\}_{t\geq 0}$ is a $(C_{0},1)$-semigroup on $\Psi$ and $p$ is a continuous semi-norm on $\Psi$, from Theorem \ref{theoExtensC01SemigroupsToBanachSpaces} there exists a continuous semi-norm $q$ on $\Psi$, $p \leq q$ and there exists a $C_{0}$-semigroup $\{ S_{q}(t) \}_{t \geq 0}$ on the Banach space $\Psi_{q}$ such that 
\begin{equation} \label{extendedSemigroupQWeakMild}
S_{q}(t) i_{q} \varphi = i_{q} S(t) \varphi, \quad \forall \, \varphi \in \Psi, \, t \geq 0.
\end{equation} 
Moreover, there exist $M_{q} \geq 1$, $\theta_{q} \geq 0$ such that 
\begin{equation} \label{exponContSemigroupQWeakMildSol}
q(S_{q}(t) i_{q}  \varphi) \leq M_{q} e^{\theta_{q} t} q (i_{q} \varphi), \quad \quad \forall \, \varphi \in \Psi, \, t \geq 0.
\end{equation}
Now, from the definition of $\tilde{F}_{X}$ and  \eqref{extendedSemigroupQWeakMild}, it follows that for $\mbox{Leb} \otimes \Prob \otimes \mu \otimes \mbox{Leb}$-a.e. $(r,\omega,u,s)$,  
\begin{equation*}
Y_{1}(r,\omega,u,s)  =  \ind{[0,s]}{r} F(r,u,X_{r}(\omega))' S(t-s) A \psi 
 =  \ind{[0,s]}{r} \tilde{F}_{X}(r,\omega , u)' i_{p,q} S_{q}(t-s) i_{q} A \psi.
\end{equation*}
Therefore it follows that for $\mbox{Leb} \otimes \Prob \otimes \mu \otimes \mbox{Leb}$-a.e. $(r,\omega,u,s)$,
\begin{equation*}
q_{r,u}(Y_{1}(r,\omega,u,s))^{2}  \leq \ind{[0,s]}{r} \norm{\tilde{F}_{X}(r,\omega , u)'}^{2}_{\mathcal{L}_{2}(\Psi_{p},\Phi_{q_{r,u}})} \norm{i_{p,q}}^{2}_{\mathcal{L}(\Psi_{q},\Psi_{p})} M_{q}^{2} e^{2 \theta_{q} (t-s)} q (i_{q} A \psi)^{2}.
\end{equation*}  
From this last inequality and because $\tilde{F}_{X} \in \Lambda^{2}_{s}(p,t)$, it follows that
\begin{eqnarray*}
\int^{t}_{0} \norm{ Y_{1}(\cdot,\cdot,\cdot,s)}_{w,t} ds 
& = & \int^{t}_{0} \left( \Exp \int^{t}_{0} \int_{U} q_{r,u} (Y_{1}(r,\omega ,u,s))^{2} \mu (du) dr \right)^{1/2} ds \\
& \leq & M_{q} e^{ \theta_{q} T} q (i_{q} A \psi) \norm{i_{p,q}}_{\mathcal{L}(\Psi_{q},\Psi_{p})} \norm{\tilde{F}_{X}}_{s,p,t} < \infty.
\end{eqnarray*}
Therefore, $Y_{1}$ satisfies all the conditions of Definition \ref{defiClassIntegrandsStochFubiniTheo} and hence $Y_{1} \in \Xi^{1,2}_{w}(t, [0,t])$. By similar reasoning we find that $Y_{2}$ also satisfies all the conditions of Definition \ref{defiClassIntegrandsStochFubiniTheo} and hence we have $Y_{2} \in \Xi^{1,2}_{w}(t, [0,t])$.

We now prove \eqref{equaFubiniEquivWeakMild1}. First, note that for all $r \in [0,t]$, $u \in U$, from standard properties of $C_{0}$ semigroups (see \cite{Komura:1968}) the following identity holds $\Prob$-a.e.
\begin{eqnarray} \label{identRiemannIntCoeffFSemigroup}
\int^{t}_{0} \ind{[0,s]}{r} F(r,u,X_{r})'S(s-r) A \psi ds 
& = & \int^{t}_{0} \ind{[0,s]}{r} F(r,u,X_{r})'S(t-s) A \psi ds  \nonumber \\
& = & F(r,u,X_{r})'(S(t-r) \psi - \psi).  
\end{eqnarray}
Then, from \eqref{weakStrongCompaStochasticConvolution} applied to $Y_{2}$ (where $\psi$ is there replaced by $A \psi$), \eqref{identRiemannIntCoeffFSemigroup}, stochastic Fubini's theorem applied to $Y_{1}$, and the linearity of the weak stochastic integral, we have $\Prob$-a.e.   
\begin{flalign*}
& \int^{t}_{0} \left( \int^{s}_{0} \int_{U} S(s-r)'F (r,u, X_{r})M(dr,du) [A \psi]\right) ds \\
& = \int^{t}_{0} \left( \int^{s}_{0}\int_{U} F (r,u, X_{r})'S(s-r) A \, \psi M (dr,du)\right) ds \\
& =\int^{t}_{0} \left( \int^{s}_{0}\int_{U} F (r,u, X_{r})'S(t-s) A \, \psi M (dr,du)\right) ds \\
& = \int^{t}_{0} \int_{U}\left(\int^{t}_{0} \ind{[0,s]}{r} F (r,u, X_{r})'S (t-s) A \, \psi ds\right) M (dr,du)  \\
& =\int^{t}_{0} \int_{U}F (r,u, X_{r})'S (t-r) \psi M (dr,du) - \int^{t}_{0} \int_{U}F (r,u, X_{r})'\psi M (dr,du). 
\end{flalign*} 
\end{prf}

\begin{lemm} \label{lemmHilbertMomentsXAndBWeakMild} Let $X$ be a $\Psi'_{\beta}$-valued regular and predictable process and assume that $X$ and $B$ satisfy \eqref{assumpXMomentTheoWeakMildSol} and \eqref{assumpBMomentTheoWeakMildSol}. Then, for each $T>0$ there exists a continuous Hilbertian semi-norm $\varrho$ on $\Psi$ such that 
\begin{equation} \label{hilbertMomentXWeakMild}
\Exp \int_{0}^{T} \varrho'(X_{r}) dr < \infty,
\end{equation}
\begin{equation} \label{hilbertMomentBWeakMild}
\Exp \int_{0}^{T} \varrho'(B(r,X_{r})) dr < \infty. 
\end{equation}
\end{lemm}
\begin{prf} Let $T>0$. We start by showing the existence of the semi-norm for $X$.  

Let $\sigma(\cdot)=\mbox{Leb}(\cdot)/T$, where $\mbox{Leb}$ denotes the Lebesgue measure on $[0,T]$. Then $( [0,T] \times \Omega, \mathcal{P}_{T}, \sigma \otimes \Prob)$ is a complete probability space. The predictability of $X$ implies that $Y^{T}: \Psi \rightarrow L^{1} ( [0,T] \times \Omega, \mathcal{P}_{T}, \sigma \otimes \Prob)$, given by 
\begin{equation} \label{defiVersionXTMomentWeakMild} 
Y^{T}(\psi)(r,\omega)=X_{r}(\omega)[\psi], \quad \forall \psi \in \Psi, \, (r,\omega) \in [0,T] \times \Omega,
\end{equation} 
defines a cylindrical random variable. Moreover, an application of Fatou's lemma shows that $Y^{T}$ is sequentially closed. Then, because $\Psi$ is ultrabornological and $ L^{1} ( [0,T] \times \Omega, \mathcal{P}_{T}, \sigma \otimes \Prob)$ is a Banach space,  by the closed graph theorem it follows that $Y^{T}$ is continuous. Therefore, the semi-norm  $\nu$ on $\Psi$ defined by $\nu(\psi)=   \int_{[0,T] \times \Omega} \abs{Y^{T}(r,\omega)[\psi]} (\sigma \otimes \Prob)(d(r,\omega))$ for all $\psi \in \Psi$, is continuous. Hence, if we apply Theorem \ref{theoExistenceCadlagContVersionHilbertSpaceUniformBoundedMoments} to the cylindrical random variable $Y^{T}$ and from \eqref{defiVersionXTMomentWeakMild}, there exists a continuous Hilbertian semi-norm $p$ on $\Psi$ such that 
$$ \Exp \int_{0}^{T} p'(X_{r}) dr = \int_{[0,T] \times \Omega} p'(Y^{T}(r,\omega)) (\sigma \otimes \Prob)(d(r,\omega)) < \infty.$$
Now, following the same arguments as above we show that there exists some continuous Hilbertian semi-norm $q$ on $\Psi$ such that $\Exp \int_{0}^{T} q'(B(r,X_{r})) dr  < \infty$.  

Finally, choose $\varrho$ such that $p \leq \varrho$ and $q \leq \varrho$. Then, we have that
$$ \Exp \int_{0}^{T} \varrho'(X_{r}) dr 
\leq \norm{ i'_{p,\varrho}}_{\mathcal{L}(\Psi'_{p}, \Psi'_{\rho})} \Exp \int_{0}^{T} p'(X_{r}) dr  < \infty, $$
and 
$$ \Exp \int_{0}^{T} \varrho'(B(r,X_{r})) dr
\leq  \norm{ i'_{q,\varrho}}_{\mathcal{L}(\Psi'_{q}, \Psi'_{\rho})} \Exp \int_{0}^{T} q'(B(r,X_{r})) dr < \infty.$$
\end{prf}

\begin{lemm} \label{lemmaFubiniEquaNonRandomIntegralsWeakMildSol}
Let $X$ be a $\Psi'_{\beta}$-valued regular and predictable process and assume that $X$ and $B$ satisfy \eqref{assumpXMomentTheoWeakMildSol} and \eqref{assumpBMomentTheoWeakMildSol}. For each $\psi \in \mbox{Dom} (A)$ and $t>0$, the following identities holds $\Prob$-a.e.
\begin{equation} \label{equalFubiNonRandomIntegralsXSA}
 \int^{t}_{0} \left( \int^{s}_{0} X_{r}[A S(t-s) A \, \psi] dr\right) ds 
 = \int^{t}_{0} X_{r}[S(t-r) A  \, \psi] dr - \int^{t}_{0} X_{r}[A \, \psi] dr,
\end{equation}
\begin{flalign}
& \int^{t}_{0}\left( \int^{s}_{0}S(s-r)'B(r,X_{r}) dr [A \, \psi] \right) ds  \nonumber \\
& =  \int^{t}_{0}\left( \int^{s}_{0} B(r,X_{r})[S(t-s) A  \, \psi] dr\right) ds \nonumber \\
& =  \int^{t}_{0} B(r,X_{r})[S(t-r) \psi] dr - \int^{t}_{0} B(r,X_{r})[\psi] dr. \label{equalFubiNonRandomIntegralsXSB1}
\end{flalign}
\end{lemm}
\begin{prf} Fix $\psi \in \mbox{Dom} (A)$ and $t \geq 0$. We start by showing \eqref{equalFubiNonRandomIntegralsXSA}. First, we need to prove that the integrals exist. Note that the predictability of $X$ and the strong continuity of the semigroup $\{ S(t) \}_{t \geq 0}$ implies that all the integrands in \eqref{equalFubiNonRandomIntegralsXSA} are $\mathcal{P}_{t}$-measurable.

Now, let $\varrho$ be a continuous Hilbertian semi-norm on $\Psi$ satisfying the conditions in Lemma \ref{lemmHilbertMomentsXAndBWeakMild} (with $T=t$). As in the proof of Lemma \ref{lemmEquaFubiniEquivWeakMild}, because $\{S(t)\}_{t\geq 0}$ is a $(C_{0},1)$-semigroup on $\Psi$ and $\varrho$ is a continuous semi-norm on $\Psi$, there exists a continuous semi-norm $q$ on $\Psi$, $\varrho \leq q$, a $C_{0}$-semigroup $\{ S_{q}(t) \}_{t \geq 0}$ on the Banach space $\Psi_{q}$ satisfying \eqref{extendedSemigroupQWeakMild}, and there exist $M_{q} \geq 1$, $\theta_{q} \geq 0$ such that $\{ S_{q}(t) \}_{t \geq 0}$ satisfies \eqref{exponContSemigroupQWeakMildSol}. Then, from \eqref{extendedSemigroupQWeakMild}, \eqref{exponContSemigroupQWeakMildSol} and \eqref{hilbertMomentXWeakMild}, it follows that
\begin{eqnarray}
 \Exp \int^{t}_{0} \abs{X_{r}[S(t-r) A \psi]}dr 
& \leq &  \Exp \int^{t}_{0} \varrho'(X_{r})\varrho(i_{\varrho} S(t-r) A \psi)dr \nonumber \\
& \leq & \Exp \int^{t}_{0} \varrho'(X_{r})\norm{i_{\varrho,q}}_{\mathcal{L}(\Psi_{q},\Psi_{\varrho})} q( S_{q}(t-r) i_{q} A \psi )dr \nonumber \\
& \leq & M_{q} e^{\theta_{q} t} q(i_{q} A \psi) \norm{i_{\varrho,q}}_{\mathcal{L}(\Psi_{q},\Psi_{\varrho})} \Exp \int_{0}^{t} \varrho'(X_{r}) dr. \label{integCond1FubiNonRandomIntegralsXSA}
\end{eqnarray}

Let $\varphi=A \psi$. In a similar way to \eqref{integCond1FubiNonRandomIntegralsXSA}, we get that
\begin{equation} \label{integCond2FubiNonRandomIntegralsXSA}
 \Exp \int^{t}_{0}\left( \int^{s}_{0} \abs{X_{r}[A S(t-s) A \psi]} dr\right) ds 
 \leq  t M_{q} e^{\theta_{q}t} q(i_{q} A \varphi) \norm{i_{\varrho,q}}_{\mathcal{L}(\Psi_{q},\Psi_{\varrho})} \Exp \int_{0}^{t} \varrho'(X_{r}) dr. 
\end{equation}
Then, from \eqref{hilbertMomentXWeakMild}, \eqref{integCond1FubiNonRandomIntegralsXSA} and \eqref{integCond2FubiNonRandomIntegralsXSA} it follows that all the integrals in \eqref{equalFubiNonRandomIntegralsXSA} exist for $\Prob$-a.e. $\omega \in \Omega$. Moreover, from Fubini's theorem, and standard properties of the dual semi-group $\{ S(t)'\}_{t \geq 0}$ and its generator $A'$, we have for 
$\Prob$-a.e. $\omega \in \Omega$,
\begin{eqnarray} 
\int^{t}_{0}\left( \int^{s}_{0}  X_{r}(\omega)[A S(t-s) A \psi] dr\right) ds  
& = & \int^{t}_{0}\left( \int^{t-r}_{0}X_{r}(\omega)[A S(s) A \psi] ds\right) dr \nonumber \\
& = & \int^{t}_{0}\left( \int^{t-r}_{0} S(s)'A' X_{r}(\omega)[A \psi]ds \right) dr \nonumber \\
& = & \int^{t}_{0} \left( \int^{t-r}_{0} S(s)'A' X_{r}(\omega)ds \right) [A \psi] dr \nonumber \\
& = & \int^{t}_{0} \left(S(t-r)' X_{r}(\omega)-X_{r}(\omega)\right) [A \psi] dr \label{equalityFubiNonRandomIntegralsXSA}.
\end{eqnarray}

We can prove \eqref{equalFubiNonRandomIntegralsXSB1} by using similar arguments to those used to show \eqref{equalFubiNonRandomIntegralsXSA}. First, the predictability of $X$, the measurability  properties of $\mbox{B}$ in Assumption (A3) and the strong continuity of the semi-group $\{S(t)\}_{t \geq 0}$ implies that all the integrands in \eqref{equalFubiNonRandomIntegralsXSB1} are $\mathcal{P}_{t}$-measurable (see the proof of Proposition \ref{remaIntegralsMildSoluWellDefined}). Second, by following  similar arguments to those used in \eqref{integCond1FubiNonRandomIntegralsXSA}, from  \eqref{hilbertMomentBWeakMild} and Proposition \ref{remaIntegralsMildSoluWellDefined}, we can show that all the integrals in \eqref{equalFubiNonRandomIntegralsXSB1} exists. Finally, we can show that \eqref{equalFubiNonRandomIntegralsXSB1} holds $\Prob$-a.e. by employing Fubini's theorem and standard properties of the dual semigroup $\{ S(t)'\}_{t \geq 0}$ and its generator $A'$ as we did in 
\eqref{equalityFubiNonRandomIntegralsXSA}. 
\end{prf}

\begin{proof}[Proof of Theorem \ref{theoEquiWeakMild}]
Assume $X$ is a weak solution to \eqref{generalFormSEE}. Fix $t \geq 0$. We start by showing that for all $ \psi \, \in \mbox{Dom}(A)$, the following holds $\Prob$-a.e.
\begin{flalign} \label{equali1WeakIsMildSol}
& \int^{t}_{0}\left( \int^{s}_{0}\int_{U} F(r,u,X_{r})'S(t-s) A \psi M (dr,du)\right) ds \\
& =X_{0}[\psi]-S(t)'X_{0}[\psi]+\int^{t}_{0} X_{r}[A \psi]ds  -\int^{t}_{0} B(r,X_{r})[S(t-r) \psi]dr+ \int^{t}_{0} B(r,X_{r})[\psi]dr. \nonumber
\end{flalign}
First, note that for fixed $s \in [0,t]$ and $\psi \in \mbox{Dom}(A)$, $S(t-s) A \psi \in \mbox{Dom}(A)$, hence from the definition of weak solution to \eqref{generalFormSEE} (where $\psi$ is there replaced by $S(t-s) A \psi$), we have $\Prob$-a.e.
\begin{flalign} \label{equali2WeakIsMildSol}
& \int^{s}_{0}\int_{U} F(r,u,X_{r})'S(t-s) A \psi M (dr,du) \\
&= (X_{s}-X_{0})[S(t-s)A\psi]- \int^{s}_{0}(X_{r}[A S(t-s) A \psi]+ B(r,X_{r})[S(t-s) A \psi])dr. \nonumber
\end{flalign}
Now, integrating both sides of \eqref{equali2WeakIsMildSol} on $[0,t]$ with respect to the Lebesgue measure, and then using \eqref{equalFubiNonRandomIntegralsXSA} and \eqref{equalFubiNonRandomIntegralsXSB1}, we have $\Prob$-a.e.
\begin{flalign}
& \int^{t}_{0}\left( \int^{s}_{0}\int_{U} F(r,u,X_{r})'S(t-s) A \psi M (dr,du)\right) ds \nonumber \\
& =\int^{t}_{0} X_{s}[S(t-s) A \psi]ds- \int^{t}_{0}X_{0}[S(t-s)A\psi]ds \nonumber \\
& \hspace{10pt} -\int^{t}_{0}\left( \int^{s}_{0}X_{r}[A S(t-s) A \psi]dr\right) ds- \int^{t}_{0}\left( \int^{s}_{0}B(r,X_{r})[S(t-s) A \psi]dr\right) ds \nonumber \\
& = -\int^{t}_{0} X_{0}[S(t-s) A \psi]ds+\int^{t}_{0} X_{r}[A\psi]dr \nonumber \\
& \hspace{10pt} -\int^{t}_{0}B(r,X_{r})[S(t-s)\psi]dr +\int^{t}_{0}B(r,X_{r})[\psi]dr. \label{equali3WeakIsMildSol}
\end{flalign}
Now, similar calculations to those used in \eqref{equalityFubiNonRandomIntegralsXSA} (for $r=0$ and for $\psi$ instead of $A \psi$), shows that $\Prob$-a.e.
\begin{equation} \label{equali4WeakIsMildSol}
\int^{t}_{0} X_{0}[S(t-s) A \psi]ds=\int^{t}_{0} X_{0}[S(s) A \psi]ds=(S(t)'X_{0}-X_{0})[\psi].
\end{equation} 
And hence from \eqref{equali3WeakIsMildSol} and \eqref{equali4WeakIsMildSol} we obtain \eqref{equali1WeakIsMildSol}. 

Substituting \eqref{equaFubiniEquivWeakMild1} into the definition of weak solution \eqref{equationWeakSolution}, and then using \eqref{equali1WeakIsMildSol}, we get that $\Prob$-a.e.
\begin{flalign}
& X_{t}[\psi] \label{equali5WeakIsMildSol} \\ 
& =  X_{0}[\psi]+ \int^{t}_{0} (X_{r}[A \psi]+ B(r,X_{r})[\psi])dr  + \int^{t}_{0}\int_{U}F(r,u,X_{r})'\psi M (dr,du) \nonumber \\
& =  X_{0}[\psi]+ \int^{t}_{0} (X_{r}[A \psi]+ B(r,X_{r})[\psi])dr  +\int^{t}_{0}\int_{U}F(r,u,X_{r})'S(t-r) \psi M (dr,du)  \nonumber \\  
& \hspace{10pt} - \int^{t}_{0}\left( \int^{s}_{0}\int_{U} F(r,u,X_{r})'S(t-s) A \psi M (dr,du)\right) ds \nonumber \\
& =  X_{0}[\psi]+ \int^{t}_{0} (X_{r}[A \psi]+ B(r,X_{r})[\psi])dr  +\int^{t}_{0}\int_{U}F(r,u,X_{r})'S(t-r) \psi M (dr,du) \nonumber \\
& \hspace{10pt} - X_{0}[\psi] + S(t)'X_{0}[\psi] - \int^{t}_{0} X_{r}[A \psi]dr  + \int^{t}_{0} B(r,X_{r})[S(t-r) \psi]dr -\int^{t}_{0} B(r,X_{r})[\psi] dr  \nonumber \\
& =  S(t)'X_{0}[\psi]+ \int^{t}_{0} B(r,X_{r})[S(t-r) \psi]dr+ \int^{t}_{0} \int_{U} F(r,u,X_{r})'S(t-r) \psi M (dr,du). \nonumber
\end{flalign}
Now, substituting \eqref{defNonRandomConvolutionIntegral} and \eqref{weakStrongCompaStochasticConvolution} in \eqref{equali5WeakIsMildSol}, we get that $\Prob$-a.e.
\begin{multline} \label{equali6WeakIsMildSol}
X_{t}[\psi]
=\Biggl( S(t)'X_{0}+\int^{t}_{0} S(t-r)'B(r,X_{r})dr
+\int^{t}_{0} \int_{U} S(t-r)'F(r,u,X_{r}) M (dr,du)\Biggr) [\psi].
\end{multline}
As \eqref{equali6WeakIsMildSol} is valid for all $\psi \in \mbox{Dom}(A)$ and $\mbox{Dom}(A)$ is dense in $\Psi$ (see \cite{Komura:1968}), then we have $\Prob$-a.e.
$$X_{t}=S(t)'X_{0}+\int^{t}_{0} S(t-r)'B(r,X_{r})dr+\int^{t}_{0} \int_{U} S(t-r)'F(r,u,X_{r}) M (dr,du)$$
and therefore $X$ is a mild solution to \eqref{generalFormSEE}.

Conversely, assume $X$ is a mild solution to \eqref{generalFormSEE}. Fix $\psi \in \mbox{Dom}(A)$ and $t \geq 0$. For $s \in [0,T]$, from the definition of mild solution \eqref{equationMildSolution}, where $\psi$ is there replaced by $A \psi$ and $t$ is replaced by $s$, we have $\Prob$-a.e.   
\begin{eqnarray}
X_{s}[A\psi]  & = & S(s)'X_{0}[A\psi]+ \int^{s}_{0} S(s-r)'B(r,X_{r})dr[A\psi] \label{equali1MildIsWeakSol} \\ 
& { } & + \int^{s}_{0} \int_{U} S(t-r)'F(r,u,X_{r}) M (dr,du)[A \psi]. \nonumber
\end{eqnarray}
Then, integrating both sides of \eqref{equali1MildIsWeakSol} on $[0,t]$ with respect to the Lebesgue measure, then using   
\eqref{equaFubiniEquivWeakMild1},  \eqref{equalFubiNonRandomIntegralsXSB1} and \eqref{equali4WeakIsMildSol}, regrouping terms and finally by using \eqref{equali5WeakIsMildSol} (that from the arguments above is equivalent to the definition of mild solution), we have $\Prob$-a.e.
\begin{flalign*}
& \int^{t}_{0} X_{s}[A \psi]ds  \\
& =  \int^{t}_{0} S(s)'X_{0}[A \psi]ds+ \int^{t}_{0}\left( \int^{s}_{0}S(s-r)'B(r,X_{r})dr[A \psi]\right) ds \\
& \hspace{10pt} + \int^{t}_{0}\left( \int^{s}_{0}\int_{U}S(s-r)'F(r,u,X_{r}) M (dr,du)[A \psi]\right) ds \\
& = S(t)'X_{0}[\psi]-X_{0}[\psi]+\int^{t}_{0}B(r,X_{r})[S(t-r)\psi]dr-\int^{t}_{0} B(r,X_{r})[\psi]dr \\
& \hspace{10pt} +  \int^{t}_{0}\int_{U} F(r,u,X_{r})'S(t-r) \psi M (dr,du)-\int^{t}_{0}\int_{U} F(r,u,X_{r})'\psi M (dr,du) \\
&  =  S(t)'X_{0}[\psi]+\int^{t}_{0}B(r,X_{r})[S(t-r)\psi]dr +\int^{t}_{0}\int_{U} F(r,u,X_{r})'S(t-r) \psi M (dr,du)  \\
& \hspace{10pt} -X_{0}[\psi]-\int^{t}_{0}B(r,X_{r})[\psi]dr-\int^{t}_{0}\int_{U} F(r,u,X_{r})'\psi M (dr,du) \\
& =X_{t}[\psi]-X_{0}[\psi]-\int^{t}_{0}B(r,X_{r})[\psi]dr-\int^{t}_{0}\int_{U} F(r,u,X_{r})'\psi M (dr,du)
\end{flalign*}
Therefore, we have $\Prob$-a.e.
$$X_{t}[\psi]=X_{0}[\psi]+\int^{t}_{0}(X_{r}[A \psi]+B(r,X_{r})[\psi])dr+\int^{t}_{0}\int_{U} F(r,u,X_{r})'\psi M (dr,du),$$
and hence $X$ is a weak solution to \eqref{generalFormSEE}.
\end{proof}

\subsection{Regularity of the Stochastic Convolution} \label{sectionRPSC}

In this section our main interest is to study the regularity of the stochastic convolution process $\left\{ \int^{t}_{0} \int_{U} S(t-r)' R(r,u) M(dr,du): t \in [0,T] \right\}$ for $R \in \Lambda^{2}_{s}(T)$. This will play an important role in the study of existence and uniqueness of mild solutions in Section \ref{sectionEUWMS}. Before we present our main result, we will introduce some notation:

\begin{nota} \label{notationStochaConvolution}
Sometimes, we will denote by $S'\ast R = \{(S'\ast R)_{t}\}_{t\geq 0}$ the stochastic convolution process 
$\left\{ \int^{t}_{0} \int_{U} S (t-r)' R(r,u) M (dr,du): t \in [0,T]\right\}$, for a given fixed $M$.  
\end{nota}

\begin{theo} \label{theoExistHilbPredicCadlagVersionStochConvol}
Let $R \in \Lambda^{2}_{s}(T)$. There exists a continuous Hilbertian semi-norm $\varrho$ on $\Psi$ such that the process $S'\ast R$ has a $\Psi'_{\varrho}$-valued, mean-square continuous, predictable version $\widetilde{S'\ast R} = \{(\widetilde{S'\ast R})_{t}\}_{t\geq 0}$ satisfying 
\begin{equation} \label{boundedSecondMomeVersionStochConv}
\sup_{t \in [0,T]} \Exp \left[  \varrho' \left( (\widetilde{S'\ast R})_{t} \right)^{2} \right]< \infty.
\end{equation}
\end{theo}
\begin{prf} 
First, it is important to remark that the fact that $R \in \Lambda^{2}_{s}(T)$ and by using similar arguments to those in the proof of Proposition \ref{remaIntegralsMildSoluWellDefined} it follows that the stochastic convolution $S'\ast R$ is well-defined. 

Now we prove the existence of a Hilbert space-valued predictable version of the stochastic convolution process. First, as $R \in \Lambda^{2}_{s}(T)$, from Corollary \ref{strongIntegrandsAsUnionIntegrandsInHilbertSpaces} there exists a continuous Hilbertian semi-norm $p$ on $\Psi$ and $\tilde{R} \in \Lambda^{2}_{s}(p,T)$ such that $R(r,\omega,u)= i'_{p} \tilde{R}(r,\omega,u)$, for $\mbox{Leb} \otimes \Prob \otimes \mu$-a.e. $(r,\omega,u) \in [0,T] \times \Omega \times U$. 

Now, as in the proof of Lemma \ref{lemmEquaFubiniEquivWeakMild}, because $\{S(t)\}_{t\geq 0}$ is a $(C_{0},1)$-semigroup on $\Psi$ and $p$ is a continuous semi-norm on $\Psi$, there exists a continuous semi-norm $q$ on $\Psi$, $p \leq q$, and there exists a $C_{0}$-semigroup $\{ S_{q}(t) \}_{t \geq 0}$ on the Banach space $\Psi_{q}$ such that \eqref{extendedSemigroupQWeakMild} holds. Moreover, there exist $M_{q} \geq 1$, $\theta_{q} \geq 0$ such that \eqref{exponContSemigroupQWeakMildSol} holds. 

Let $\eta$ be a continuous Hilbertian semi-norm on $\Psi$ such that $q \leq \eta$. Then, for fixed $t \in [0,T]$ it follows from the above properties that for $\mbox{Leb} \otimes \Prob \otimes \mu$-a.e. $(r,\omega,u)$, 
\begin{equation} \label{decompEquivaVersionHilbSpaceStochConvol}
\ind{[0,t]}{r} S(t-r)' R(r,\omega,u)  = \ind{[0,t]}{r} i'_{\eta} \, i'_{q,\eta} \, S_{q}(t-r)' \, i'_{p,q} \,  \tilde{R}(r,\omega,u). 
\end{equation}
Our objective is to prove that $\{ \ind{[0,t]}{r} i'_{q,\eta} \, S_{q}(t-r)' \, i'_{p,q} \,  \tilde{R}(r,\omega,u) \} \in \Lambda^{2}_{s}(\eta,t)$ for each $t \in [0,T]$.  

First, for every $(r,\omega, u) \in [0,T] \times \Omega \times U$, because $\tilde{R}(r,\omega,u) \in \mathcal{L}_{2}(\Phi'_{q_{r,u}},\Psi'_{p})$, $i'_{p,q} \in \mathcal{L}(\Psi'_{p},\Psi'_{q})$, $S_{q}(t-r)' \in \mathcal{L}(\Psi'_{q},\Psi'_{q})$ and $i'_{q,\eta} \in \mathcal{L}(\Psi'_{q},\Psi'_{\eta})$, then it follows that $i'_{q,\eta} \, S_{q}(t-r)' \, i'_{p,q} \,  \tilde{R}(r,\omega,u) \in \mathcal{L}_{2}(\Psi'_{\eta},\Phi'_{q_{r,u}})$.

Now, fix $\psi \in \Psi$ and $\phi \in \Phi$. Because the map $(r,\omega,u) \mapsto \ind{[0,t]}{r} q_{r,u}(\tilde{R}(r,\omega,u)'S(t-r)\psi, \phi)$
 is $\mathcal{P}_{t} \otimes \mathcal{B}(U)$-measurable (see Proposition \ref{remaIntegralsMildSoluWellDefined}) and from \eqref{decompEquivaVersionHilbSpaceStochConvol} it follows that the map 
$$(r,\omega,u) \mapsto \ind{[0,t]}{r} q_{r,u}(\tilde{R}(r,\omega,u)'\, i_{p,q} \, S_{q}(t-r) \, i_{q,\eta} i_{\eta} \psi, \phi),$$ 
is also $\mathcal{P}_{t} \otimes \mathcal{B}(U)$-measurable. Finally, from \eqref{exponContSemigroupQWeakMildSol}, we have
\begin{multline} \label{finiteHSMomentEquiVersionHilbSpacStochConv}
\Exp \int_{0}^{t} \int_{U} \norm{ i'_{q,\eta} \, S_{q}(t-r)' \, i'_{p,q} \, \tilde{R}(r,u) }^{2}_{\mathcal{L}_{2}(\Psi_{\eta},\Phi_{q_{r,u}})} \mu(du) \lambda (dr)  \\
\leq M_{q}^{2} e^{2\theta_{q} t} \norm{i_{p,q}}^{2}_{\mathcal{L}(\Psi_{q},\Psi_{p})} \norm{i_{q,\varrho}}^{2}_{\mathcal{L}(\Psi_{\eta},\Psi_{q})} \norm{\tilde{R}}^{2}_{s,p,t} < \infty.  
\end{multline}
Then, $\{ \ind{[0,t]}{r} i'_{q,\eta} \, S_{q}(t-r)' \, i'_{p,q} \,  \tilde{R}(r,\omega,u)\}$ satisfies the conditions of Definition \ref{integrandsStrongIntegSquareMomentsInHilbertSpace} and hence belongs to $\Lambda^{2}_{s}(\eta,t)$. Therefore, from Proposition \ref{spaceStrongIntegrandsHilbertSpaceAreOrdered}, Theorem \ref{existenceStrongStochasticIntegSquareMoments} and  \eqref{decompEquivaVersionHilbSpaceStochConvol}, there exists a continuous Hilbertian seminorm $\varrho$ on $\Psi$, $\eta \leq \varrho$, such that for each $t \in [0,T]$ we have 
$$ \ind{[0,t]}{r} S(t-r)' R(r,\omega,u)  = \ind{[0,t]}{r} i'_{\varrho} \, i'_{q,\varrho} \, S_{q}(t-r)' \, i'_{p,q} \,  \tilde{R}(r,\omega,u), \quad \mbox{Leb} \otimes \Prob \otimes \mu \mbox{-a.e.}, $$ 
$\{ \ind{[0,t]}{r} i'_{q,\varrho} \, S_{q}(t-r)' \, i'_{p,q} \,  \tilde{R}(r,\omega,u)\} \in \Lambda^{2}_{s}(\varrho,T)$ and 
$\int_{0}^{t} \int_{U} i'_{q,\varrho} \, S_{q}(t-r)' \, i'_{p,q} \,  \tilde{R}(r,u) M(dr,du)$ is a $\Psi'_{\rho}$-valued $\mathcal{F}_{t}$-measurable version of $\int_{0}^{t} \int_{U} S(t-r)' R(r,u) M(dr,du)$.

Our next objective is to prove that the $\Psi'_{\varrho}$-valued process 
$$\left\{ \int_{0}^{t} \int_{U} i'_{q,\varrho} \, S_{q}(t-r)' \, i'_{p,q} \, \tilde{R}(r,u) M(dr,du): t \in [0,T] \right\},$$
is mean square continuous. We will prove the left continuity as the right continuity follows from similar arguments. Let $0 < t \leq T$. Then, from the linearity of the strong stochastic integral and Proposition \ref{propStrongIntegralInSubintervalAndRandomSubset}, for any $0 \leq s <t$ we have 
\begin{flalign}
& \Exp \Biggl[ \varrho \Biggl( \int_{0}^{t} \int_{U} i'_{q,\varrho} \, S_{q}(t-r)' \, i'_{p,q} \, \tilde{R}(r,u) M(dr,du)  - \int_{0}^{s} \int_{U} i'_{q,\varrho} \, S_{q}(s-r)' \, i'_{p,q} \,  \tilde{R}(r,u) M(dr,du) \Biggl)^{2} \, \Biggl] \nonumber \\
& \leq 2 \, \Exp  \Biggl[ \varrho \left( \int_{0}^{t} \int_{U} \ind{[s,t]}{r} i'_{q,\varrho} \, S_{q}(t-r)' \, i'_{p,q} \,  \tilde{R}(r,u) M(dr,du) \right)^{2} \, \Biggl] \nonumber \\
& \hspace{15pt} + 2 \, \Exp \Biggl[ \varrho \left( \int_{0}^{s} \int_{U} i'_{q,\varrho} \left(  S_{q}(t-r)' -  S_{q}(s-r)' \right) i'_{p,q} \,  \tilde{R}(r,u) M(dr,du) \right)^{2} \, \Biggl] \label{meanSquareContVersionStochConvol} 
\end{flalign}

Now, we start with the first term in the right-hand side of the inequality in \eqref{meanSquareContVersionStochConvol}.  
From \eqref{itoIsometryStrongIntegHilbSchmiIntegrands} and arguing in a similar way to the derivation of \eqref{finiteHSMomentEquiVersionHilbSpacStochConv} we have for any $0 \leq s <t$ that
\begin{flalign}
& \Exp  \Biggl[ \varrho \left( \int_{0}^{t} \int_{U} \ind{[s,t]}{r} i'_{q,\varrho} \, S_{q}(t-r)' \, i'_{p,q} \, \tilde{R}(r,u) M(dr,du) \right)^{2} \, \Biggl] \label{inequaPart1MeanSquareContVersStochConv} \\ 
& = \Exp \int_{s}^{t} \int_{U} \norm{ i'_{q,\varrho} \, S_{q}(t-r)' \, i'_{p,q} \, \tilde{R}(r,u) }^{2}_{\mathcal{L}_{2}(\Phi'_{q_{r,u}}, \Psi'_{\varrho})} \mu(du) dr  \nonumber \\
& \leq M_{q}^{2} e^{2 \theta_{q} (t-s)} \norm{i_{p,q}}^{2}_{\mathcal{L}(\Psi_{q},\Psi_{p})} \norm{i_{q,\varrho}}^{2}_{\mathcal{L}(\Psi_{\varrho},\Psi_{q})} \norm{\tilde{R}}^{2}_{s,p,T},\nonumber
\end{flalign}
Then, from \eqref{inequaPart1MeanSquareContVersStochConv} we have 
\begin{equation} \label{part1MeanSquareContVersStochConv}
\lim_{s \rightarrow t-} \Exp  \Biggl[ \varrho \left( \int_{0}^{t} \int_{U} \ind{[s,t]}{r} i'_{q,\varrho} \, S_{q}(t-r)' \, i'_{p,q} \,  \tilde{R}(r,u) M(dr,du) \right)^{2} \, \Biggl] = 0.
\end{equation}
For the second term in the right-hand side of the inequality in \eqref{meanSquareContVersionStochConvol}, proceeding as in \eqref{inequaPart1MeanSquareContVersStochConv}, we can prove that for any $0 \leq s <t$,
\begin{flalign}
& \Exp \Biggl[ \varrho \left( \int_{0}^{s} \int_{U} i'_{q,\varrho} \left(  S_{q}(t-r)' -  S_{q}(s-r)' \right) i'_{p,q} \,   \tilde{R}(r,u) M(dr,du) \right)^{2} \, \Biggl] \label{inequaPart2MeanSquareContVersStochConv} \\ 
& = \Exp \int_{0}^{s} \int_{U} \norm{ \tilde{R}(r,u)'\, i_{p,q} \left(  S_{q}(t-r) -  S_{q}(s-r) \right)  i_{q,\varrho} }^{2}_{\mathcal{L}_{2}(\Psi_{\varrho},\Phi_{q_{r,u}})} \mu(du) dr  \nonumber \\
& \leq M_{q}^{2} e^{2 \theta_{q} T} \norm{i_{p,q}}^{2}_{\mathcal{L}(\Psi_{q},\Psi_{p})} \norm{i_{q,\varrho}}^{2}_{\mathcal{L}(\Psi_{\varrho},\Psi_{q})} \norm{\tilde{R}}^{2}_{s,p,T} < \infty. \nonumber
\end{flalign}
Now, let $\{ \psi_{j}^{\varrho} \}_{j \in \N} \subseteq \Psi$ be a complete orthonormal system in $\Psi_{\varrho}$. For each  $j \in \N$, the strong continuity of the semigroup $\{ S_{q}(t) \}_{ t \geq 0}$, the continuity of the maps $i_{p,q}$ and of $R(r,\omega,u)'$ (for fixed $(r,\omega,u)$), and the dominated convergence theorem imply that 
\begin{equation} \label{meanSquareContInOrthBasisVersionStochConvo}
\lim_{s \rightarrow t-} \Exp \int_{0}^{T} \int_{U}\ind{[0,s]}{r} q_{r,u} \left( \tilde{R}(r,u)'\, i_{p,q} \left(  S_{q}(t-r) -  S_{q}(s-r) \right)  i_{q,\varrho} \psi_{j}^{\varrho} \right)^{2} \mu(du) dr =0.  
\end{equation}
By Fubini's theorem and Parseval's identity we have
\begin{flalign}
& \Exp \int_{0}^{s} \int_{U} \norm{ \tilde{R}(r,u)'\, i_{p,q} \left(  S_{q}(t-r) -  S_{q}(s-r) \right)  i_{q,\varrho} }^{2}_{\mathcal{L}_{2}(\Psi_{\varrho},\Phi_{q_{r,u}})} \mu(du) dr  \nonumber \\
& = \sum_{j =1}^{\infty} \Exp \int_{0}^{T} \int_{U}\ind{[0,s]}{r} q_{r,u} \left( \tilde{R}(r,u)'\, i_{p,q} \left(  S_{q}(t-r) -  S_{q}(s-r) \right)  i_{q,\varrho} \psi_{j}^{\varrho} \right)^{2} \mu(du) dr. \nonumber
\end{flalign}
Hence, from \eqref{inequaPart2MeanSquareContVersStochConv}, \eqref{meanSquareContInOrthBasisVersionStochConvo} and the dominated convergence theorem it follows that 
\begin{equation}  \label{part2MeanSquareContVersStochConv} 
\lim_{s \rightarrow t-} \Exp \Biggl[ \varrho \left( \int_{0}^{s} \int_{U} i'_{q,\varrho} \left(  S_{q}(t-r)' -  S_{q}(s-r)' \right) i'_{p,q} \,  \tilde{R}(r,u) M(dr,du) \right)^{2} \, \Biggl] = 0. 
\end{equation} 
Finally, from \eqref{meanSquareContVersionStochConvol}, \eqref{part1MeanSquareContVersStochConv} and \eqref{part2MeanSquareContVersStochConv}, it follows that $\widetilde{S'\ast R} = \{(\widetilde{S'\ast R})_{t}\}_{t\geq 0}$ given by 
$$ (\widetilde{S'\ast R})_{t}= \int_{0}^{t} \int_{U} i'_{q,\varrho} \, S_{q}(t-r)' \, i'_{p,q} \,  \tilde{R}(r,u) M(dr,du), \quad \forall \, t \in [0,T],$$
is mean square continuous. Furthermore, as it is also $\{ \mathcal{F}_{t} \}$-adapted and $\Psi'_{\varrho}$ is a separable Hilbert space, then it has a predictable version (see \cite{PeszatZabczyk}, Proposition 3.21, p.27). Moreover, from \eqref{inequaPart1MeanSquareContVersStochConv} (taking $s=0$) we have 
\eqref{boundedSecondMomeVersionStochConv}. 
\end{prf}

\subsection{Existence and Uniqueness of Weak and Mild Solutions} \label{sectionEUWMS}

In this section we prove the existence and uniqueness of weak and mild solutions to \eqref{generalFormSEE} under some Lipschitz and growth conditions on the coefficients $B$ and $F$.  
We will need the following additional assumptions on the dual space $\Psi'_{\beta}$ and the dual semigroup $\{ S(t)' \}_{t \geq 0}$ in this section. 

\begin{assu} \label{assumpForDualSpaceAndDualSemigroup} \hfill 
\begin{enumerate}
\item Every continuous semi-norm on $\Psi'_{\beta}$ is separable. 
\item The dual semigroup $\{ S(t)' \}_{t \geq 0}$ is a $(C_{0},1)$ semigroup on $\Psi'_{\beta}$. 
\end{enumerate}
\end{assu}

\begin{rema}
Assumption \ref{assumpForDualSpaceAndDualSemigroup}(1) is satisfied if $\Psi'_{\beta}$ is either separable or nuclear (e.g if $\Psi'_{\beta}$ is the space of distributions $\mathscr{D}'$ or the space of tempered distributions $\mathscr{S}'$ on $\R^{d}$; see \cite{Treves}, Chapter 51). Assumption \ref{assumpForDualSpaceAndDualSemigroup}(2) is satisfied if $\{ S(t) \}_{t \geq 0}$ is equicontinuous because in that case $\{ S(t)' \}_{t \geq 0}$ is also equicontinuous (see Section \ref{sectionSLOLCS}).  
\end{rema}

Recall from Section \ref{subsectionNuclSpace} that for each $K \subseteq \Psi$ bounded, $\eta_{K}: \Psi' \rightarrow \R_{+}$ given by 
$$\eta_{K}(f) \defeq  p_{K^{0}}(f)=   \sup_{ \psi \in K} \abs{ f[\psi]}, \quad \forall \, f \in \Psi', $$ 
is a continuous semi-norm on $\Psi'_{\beta}$, where $p_{K^{0}}$ is the Minkowski functional of $K^{0}$. Moreover, the family $\{ \eta_{K}: K \subseteq \Psi, \, K \mbox{ is bounded} \}$ generates the topology on $\Psi'_{\beta}$. 

For each $K \subseteq \Psi$ bounded we denote by $\Psi'_{K}$ the Banach space that corresponds to the completion of $(\Psi'/\mbox{ker}(\eta_{K}), \tilde{\eta_{K}})$ where $\tilde{\eta_{K}}(f+\mbox{ker}(\eta_{K}))=\eta_{K}(f)$. The canonical inclusion from $\Psi'_{\beta}$ into $\Psi'_{K}$ will be denoted by $j_{K}$. If $K$, $D$ are any bounded subsets of $\Psi$ such that $K \subseteq D$, then we have $\eta_{K} \leq \eta_{D}$ and we denote by $j_{K,D}$ the canonical inclusion from $\Psi'_{D}$ into $\Psi'_{K}$. If for $K \subseteq \Psi$ bounded we have that $\Psi'_{K}$ is a Hilbert space, then we say that $K$ is a \emph{Hilbertian set}.      

The following key property of the dual semigroup $\{ S(t)' \}_{t \geq 0}$ will be of great importance for our proof of existence and uniqueness of solutions to \eqref{generalFormSEE}. 

\begin{lemm} \label{existenceBoundedHilbertFamilyAndSemigroup}
There exists a non-empty family $\mathcal{K}_{H}(\Psi)$ of bounded subsets of $\Psi$, such that for all $K \in \mathcal{K}_{H}(\Psi)$, $\Psi'_{K}$ is a separable Hilbert space and there exists a $C_{0}$-semigroup $\{ S_{K}(t) \}_{t \geq 0}$ on $\Psi'_{K}$ such that 
\begin{equation} \label{compatiDualSemigroupWithHilbertExtension}
S_{K}(t) j_{K} f = j_{K} S(t)' f, \quad \forall \, t \geq 0, \, f \in \Psi'.
\end{equation}
\end{lemm}
\begin{prf} The result follows from Proposition \ref{propExistFamiHilbSeminormExtenSemiGroup}, Assumption \ref{assumpForDualSpaceAndDualSemigroup} and the fact that because $\Psi$ is reflexive, then every continuous semi-norm $\nu$ on $\Psi'_{\beta}$ satisfies $\nu=\eta_{K}$ where $K=B_{\nu}(1)^{0} \subseteq \Psi$ is bounded (see \cite{Schaefer}, Theorems IV.5.2 and IV.5.6).
\end{prf}

Now, for the proof of existence and uniqueness of solutions to \eqref{generalFormSEE} we will follow a fixed point theorem argument and to do this we will need a class of $\Psi'_{\beta}$-process where the solution will lie. Adapting the ideas used in  \cite{DaPratoZabczyk} p.188 to our context, we define this class as follows:

\begin{defi}
Let $T>0$. We denote by $\mathcal{H}^{2}(T, \Psi'_{\beta})$ the vector space of all the (equivalence classes of) $\Psi'_{\beta}$-valued, regular, predictable processes $X=\{ X_{t} \}_{t \in [0,T]}$ such that $\forall \, K \in \mathcal{K}_{H}(\Psi)$, 
\begin{equation*}
\sup_{t \in [0,T]} \Exp \, \eta_{K}(j_{K} X_{t})^{2} < \infty. 
\end{equation*}
\end{defi}

Now, we need to equip the linear space $\mathcal{H}^{2}(T, \Psi'_{\beta})$ with a locally convex topology. To do this, will need the following family of Banach spaces. 

\begin{defi}
Let $T>0$ and $K \in \mathcal{K}_{H}(\Psi)$. Denote by $\mathcal{H}^{2}(T,\Psi'_{K})$ the vector space of all (equivalence classes of) $\Psi'_{K}$-valued predictable process $X= \{X_{t}\}_{t \in [0,T]}$ such that 
$$\norm{X}_{K,T} \defeq \sup_{t \in [0,T]} \left( \Exp \,   \eta_{K}(X_{t})^{2}  \right)^{1/2} < \infty.$$ 
\end{defi}
The space $\mathcal{H}^{2}(T,\Psi'_{K})$ is a Banach space when equipped with the topology defined by the norm $\norm{\cdot}_{K,T}$ . 
\begin{rema} \label{equivNormBanSpacExisUniq} If for $\upsilon \geq 0$ we define $\norm{ \cdot }_{\upsilon,K,T}$ by (see \cite{PeszatZabczyk}, p. 164) 
\begin{equation*} 
\norm{ X }_{\upsilon,K,T} = \sup_{t \in [0,T]} \left( e^{-\upsilon t} \Exp \left( \eta_{K}(X_{t})^{2} \right) \right)^{1/2}, \quad  \forall \, X \in \mathcal{H}^{2}(T, \Psi'_{\beta}),
\end{equation*}
then $\norm{\cdot}_{K,T}= \norm{\cdot}_{0,K,T}$ and it is clear that the norms $\abs{\norm{ \cdot }}_{\upsilon,K,T}$, $\upsilon \geq 0$ are equivalent. 
\end{rema}

The following result is an immediate consequence of the definition of the spaces $\mathcal{H}^{2}(T, \Psi'_{\beta})$ and $\mathcal{H}^{2}(T,\Psi'_{K})$, and the continuity and the linearity of the map $j_{K}$. 

\begin{lemm} \label{mapsJKDefineProjectiveSystem}
For each $K \in \mathcal{K}_{H}(\Psi)$, the map $j_{K}: \mathcal{H}^{2}(T, \Psi'_{\beta}) \rightarrow \mathcal{H}^{2}(T,\Psi'_{K})$ given by $X \mapsto j_{K} X$ is continuous, linear and injective. 
\end{lemm}

Now, from Lemma \ref{mapsJKDefineProjectiveSystem} it follows that  we can equip the space $\mathcal{H}^{2}(T, \Psi'_{\beta})$ with the projective topology with respect to the family $\{ (\mathcal{H}^{2}(T,\Psi'_{K}),j_{K}): K \in \mathcal{K}_{H}(\Psi)\}$. Then, equipped with this topology the space $\mathcal{H}^{2}(T, \Psi'_{\beta})$ is a complete, Hausdorff, locally convex space (this is a consequence of the fact that each $\mathcal{H}^{2}(T,\Psi'_{K})$ satisfies these properties; see \cite{Schaefer}, Theorems II.5.1 and II.5.3, p.51-2). Moreover, the topology on $\mathcal{H}^{2}(T, \Psi'_{\beta})$  is generated by the family of semi-norms $\{ \abs{\norm{ \cdot }}_{K,T}: K \in \mathcal{K}_{H}(\Psi) \}$, given for each $K \in \mathcal{K}_{H}(\Psi)$ by  
\begin{equation} \label{defiSemiNormsTopologyContraction}
\abs{\norm{ X }}_{K,T} \defeq \norm{j_{K} X}_{K,T} = \sup_{t \in [0,T]} \left( \Exp \left( \eta_{K}(j_{K} X_{t})^{2} \right) \right)^{1/2}, \quad  \forall \, X \in \mathcal{H}^{2}(T, \Psi'_{\beta}). 
\end{equation}

The Lipschitz and growth conditions that we assume for our coefficients $B$ and $F$ are the following:

\begin{assu}\label{lipschitzGrowthConditions} There exist two functions $a, b: \Psi \times \R_{+} \rightarrow \R_{+}$ satisfying:
\begin{enumerate}
\item For each $T>0$ and $K \subseteq \Psi$ bounded, 
\begin{equation*}
 \int_{0}^{T} \sup_{\psi \in K} ( a(\psi,r)^{2} +  b(\psi,r)^{2}) dr < \infty. 
\end{equation*}
\item (Growth conditions) For all $r \in \R_{+}$, $g \in \Psi'$,   
\begin{align*}
\abs{B(r,g)[\psi]} & \leq  a(\psi,r)(1+\abs{g[\psi]}), \\
\int_{U} q_{r,u}(F(r,u,g)'\psi)^{2} \mu(du) & \leq b(\psi,r)^{2}(1+\abs{g[\psi]})^{2}. 
\end{align*} 
\item (Lipschitz conditions) For all $r \in \R_{+}$, $g_{1}, g_{2} \in \Psi'$,  
\begin{align*}
\abs{B(r,g_{1})[\psi]-B(r,g_{2})[\psi]} & \leq  a(\psi,r) \abs{g_{1}[\psi]-g_{2}[\psi]}, \\
\int_{U} q_{r,u}(F(r,u,g_{1})'\psi-F(r,u,g_{2})'\psi)^{2} \mu(du)  & \leq b(\psi,r)\abs{g_{1}[\psi]-g_{2}[\psi]}^{2}. 
\end{align*}
\end{enumerate}
\end{assu}

We are ready for the main result of this section. 

\begin{theo} [Existence and uniqueness] \label{theoExistenceAndUniquenessMildSolutions}  Let $Z_{0}$ be a $\Psi'_{\beta}$-valued, regular, $\mathcal{F}_{0}$-measurable, square integrable random variable. Then, there exists a unique (up to modification) mild solution $X= \{ X_{t} \}_{t \geq 0}$ to \eqref{generalFormSEE} with initial condition $X_{0}=Z_{0}$. Moreover, for every $T>0$ there exists a continuous Hilbertian semi-norm $\rho=\rho(T)$ on $\Psi$ such that
$X=\{ X_{t} \}_{t \in [0,T]}$ has a $\Psi'_{\rho}$-valued predictable version $\tilde{X}=\{ \tilde{X}_{t} \}_{t \in [0,T]}$ satisfying $\sup_{t \in [0,T]} \Exp \left( \rho'(\tilde{X}_{t})^{2} \right)< \infty$. Furthermore, $X$ is also a weak solution to \eqref{generalFormSEE}. 
\end{theo}

Let $T>0$. Consider the operator $\mathbb{A}: \mathcal{H}^{2}(T,\Psi'_{\beta}) \rightarrow \mathcal{H}^{2}(T,\Psi'_{\beta})$  defined by 
$$ \mathbb{A}(X) = \mathbb{A}_{0}(X) +\mathbb{A}_{1}(X)+ \mathbb{A}_{2}(X), \quad \forall \, X \in \mathcal{H}^{2}(T,\Psi'_{\beta}), $$
where for each $t \in [0,T]$,
\begin{align}
\mathbb{A}_{0}(X)_{t} & \defeq S(t)' Z_{0}, \nonumber \\
\mathbb{A}_{1}(X)_{t} & \defeq \int^{t}_{0} S(t-r)'B(r,X_{r})dr, \nonumber \\
\mathbb{A}_{2}(X)_{t} & \defeq \int^{t}_{0} \int_{U} S(t-r)'F(r,u,X_{r}) M(dr,du). \nonumber
\end{align}
Our objective is to show that the map $\mathbb{A}$ is a contraction on $\mathcal{H}^{2}(T, \Psi'_{\beta})$. Then, we have to show that for every $K \in \mathcal{K}_{H}(\Psi)$ there exists $0 < C_{K,T} <1$ such that 
$$  \abs{\norm{ \mathbb{A}X- \mathbb{A}Y  }}_{K,T} \leq C_{K,T} \abs{\norm{ X-Y }}_{K,T}, \quad \forall \, X,Y \in \mathcal{H}^{2}(T,\Psi'_{\beta}). $$
However, by Remark \ref{equivNormBanSpacExisUniq} it is equivalent to show that for each $K \in \mathcal{K}_{H}(\Psi)$ there exists $\upsilon \geq 0$ and a constant $0 <C_{\upsilon,K,T} <1$ such that
\begin{equation} \label{inequalityContractExistMildSol}
 \abs{\norm{ \mathbb{A}X- \mathbb{A}Y  }}_{\upsilon,K,T} \leq C_{\upsilon,K,T} \abs{\norm{ X-Y }}_{\upsilon,K,T}, \quad \forall \, X,Y \in \mathcal{H}^{2}(T,\Psi'_{\beta}). 
\end{equation}
where the semi-norm $\abs{\norm{ \cdot }}_{\upsilon,K,T}$ is given  by  
 \begin{equation} \label{defiEquivalentSemiNormsTopologyContraction}
\abs{\norm{ X }}_{\upsilon,K,T} \defeq \norm{ j_{K} X }_{\upsilon,K,T} = \sup_{t \in [0,T]} \left( e^{-\upsilon t} \Exp \left( \eta_{K}(j_{K} X_{t})^{2} \right) \right)^{1/2}, \quad  \forall \, X \in \mathcal{H}^{2}(T, \Psi'_{\beta}).
 \end{equation}

In the next result we show that $\mathbb{A}$ is well-defined and that it is a contraction on $\mathcal{H}^{2}(T, \Psi'_{\beta})$.

\begin{lemm} \label{contractionWellDefined}
The operator $\mathbb{A}$ is a contraction on $\mathcal{H}^{2}(T, \Psi'_{\beta})$. Moreover, for every $X=\{ X_{t} \}_{t \in [0,T]} \in  \mathcal{H}^{2}(T,\Psi'_{\beta})$ there exists a continuous Hilbertian semi-norm $\rho$ on $\Psi$ such that
$\mathbb{A}(X)$ has a $\Psi'_{\rho}$-valued predictable version $\widetilde{\mathbb{A}(X)}=\{ \widetilde{\mathbb{A}(X)}_{t} \}_{t \in [0,T]}$ satisfying 
\begin{equation} \label{boundedHilSpacMomenMildSol}
\sup_{t \in [0,T]} \Exp \left( \rho'\left(\widetilde{\mathbb{A}(X)}_{t} \right)^{2} \right)< \infty.
\end{equation}
\end{lemm}
\begin{prf} We proceed in four steps. 

\textbf{Step 1:} \emph{Estimating $\mathbb{A}_{0}$.} First, as $Z_{0}$ is square integrable and regular, then $\psi \mapsto \Exp \abs{Z_{0}[\psi]}^{2}$ defines a continuous Hilbertian semi-norm on $\Psi$. Then,
Theorem \ref{theoExistenceCadlagContVersionHilbertSpaceUniformBoundedMoments}  shows that there exists a continuous Hilbertian semi-norm $q$ on $\Psi$, such that $Z_{0}$ possesses a $\Psi'_{q}$-valued, $\mathcal{F}_{0}$-measurable version $\tilde{Z}_{0}$ satisfying $\Exp q'(\tilde{Z}_{0})^{2} < \infty$. 

Now, as in the proof of Lemma \ref{lemmEquaFubiniEquivWeakMild}, because $\{S(t)\}_{t\geq 0}$ is a $(C_{0},1)$-semigroup on $\Psi$ and $q$ is a continuous semi-norm on $\Psi$, there exists a continuous semi-norm $\varrho_{0}$ on $\Psi$, $q \leq \varrho_{0}$, and there exists a $C_{0}$-semigroup $\{ S_{\varrho_{0}}(t) \}_{t \geq 0}$ on the Banach space $\Psi_{\varrho_{0}}$ such that \eqref{extendedSemigroupQWeakMild} holds. Moreover, there exist $M_{\varrho_{0}} \geq 1$, $\theta_{\varrho_{0}} \geq 0$ such that \eqref{exponContSemigroupQWeakMildSol} holds. 

Hence, from the above we have that for each $t \in [0,T]$, $\Prob$-a.e. 
\begin{equation*}
S(t)'Z_{0}= i'_{\varrho_{0}} S_{\varrho_{0}}(t)' i'_{q,\varrho_{0}}\tilde{Z}_{0}.  
\end{equation*}
Then, $\{ S_{\varrho_{0}}(t)' i'_{q,\varrho_{0}}\tilde{Z}_{0} \}_{t \geq 0}$ is a $\Psi'_{\varrho_{0}}$-valued version of $\{ S(t)'Z_{0} \}_{t \geq 0}$. Moreover, $\{ S_{\varrho_{0}}(t)' i'_{q,\varrho_{0}}\tilde{Z}_{0} \}_{t \geq 0}$ is a $\Psi'_{\varrho_{0}}$-valued, $\{ \mathcal{F}_{t} \}$-adapted, continuous process and $\Psi'_{\varrho_{0}}$ is a separable Banach space, then it has a predictable version (see  \cite{PeszatZabczyk}, Proposition 3.21, p.27). Furthermore, we have that  
\begin{equation} \label{finiteMomeHilbLinearTermMildSol}
\Exp \sup_{t \in [0,T]} \varrho'_{0} \left( S_{\varrho_{0}}(t)' i'_{q,\varrho_{0}}\tilde{Z}_{0} \right)^{2} \leq M_{\varrho_{0}}^{2} \, e^{2 \theta_{\varrho_{0}} T }   \norm{i_{q,\varrho_{0}}}^{2}_{\mathcal{L}(\Psi_{\varrho_{0}},\Psi_{q})}  \Exp p'(\tilde{Z}_{0})^{2} < \infty.
\end{equation}
From the corresponding properties of $\{ S_{\varrho_{0}}(t)' i'_{q,\varrho_{0}}\tilde{Z}_{0} \}_{t \geq 0}$ we conclude that 
$\{ S(t)'Z_{0} \}_{t \geq 0}$ is a $\Psi'_{\beta}$-valued continuous, regular process which has a predictable version. 
Moreover, if $K \in \mathcal{K}_{H}(\Psi)$ the map $j_{K} \circ i'_{\varrho_{0}}: \Psi'_{\varrho_{0}} \rightarrow \Psi'_{K}$ is linear and continuous. Therefore, from the arguments in the above paragraphs we have 
\begin{equation} \label{finiteMomentLinearTermMildSol}
\Exp \left( \sup_{t \in [0,T]} \eta_{K}(j_{K} S(t)'Z_{0})^{2} \right)    \leq  \norm{j_{K} i'_{\varrho_{0}}}^{2}_{\mathcal{L}(\Psi'_{\varrho_{0}},\Psi'_{K})} \Exp \sup_{t \in [0,T]} \varrho'_{0} \left( S_{\varrho_{0}}(t)' i'_{q,\varrho_{0}}\tilde{Z}_{0}  \right)^{2}  < \infty. 
\end{equation}
Hence, $\{ S(t)'Z_{0} \}_{t \geq 0}$ is an element of $\mathcal{H}^{2}(T,\Psi'_{\beta})$ and consequently $\mathbb{A}_{0}$ is well-defined. 

\textbf{Step 2:} \emph{Estimating $\mathbb{A}_{1}$.} We proceed in two steps. 
  
(a) We will show that $\mathbb{A}_{1}$ is well-defined. Let $X \in \mathcal{H}^{2}(T,\Psi'_{\beta})$.  First, from Assumption \ref{lipschitzGrowthConditions} we have that for every $\psi \in \Psi$, 
\begin{equation}
\Exp \int_{0}^{T} \abs{B(r,X_{r})[\psi]}^{2} dr  \leq  2 \left( 1+ \sup_{t \in [0,T]} \Exp  \abs{X_{t}[\psi]}^{2} \right) \int_{0}^{T} a(\psi ,r)^{2} dr < \infty. 
\end{equation}
Hence, from similar arguments to those used in Lemma \ref{lemmHilbertMomentsXAndBWeakMild} it follows that there exists a continuous Hilbertian semi-norm $p$ on $\Psi$ and a $\Psi'_{p}$-valued $\{ \mathcal{F}_{t} \}$-adapted version $\{ \tilde{B}(t,X_{t}) \}_{t \in [0,T]}$ of $\{ B(t,X_{t}) \}_{t \in [0,T]}$ such that $\Exp \int_{0}^{T} p'(\tilde{B}(r,X_{r}))^{2} dr < \infty$. 

Now, as in the proof of Lemma \ref{lemmEquaFubiniEquivWeakMild}, because $\{S(t)\}_{t\geq 0}$ is a $(C_{0},1)$-semigroup on $\Psi$ and $p$ is a continuous semi-norm on $\Psi$, there exists a continuous semi-norm $\varrho_{1}$ on $\Psi$, $p \leq \varrho_{1}$, and there exists a $C_{0}$-semigroup $\{ S_{\varrho_{1}}(t) \}_{t \geq 0}$ on the Banach space $\Psi_{\varrho_{1}}$ such that \eqref{extendedSemigroupQWeakMild} holds. Moreover, there exist $M_{\varrho_{1}} \geq 1$, $\theta_{\varrho_{1}} \geq 0$ such that $\{ S_{\varrho_{1}}(t) \}_{t \geq 0}$ satisfies \eqref{exponContSemigroupQWeakMildSol}. Then, for every $t \geq 0$ we have 
\begin{flalign} 
& \Exp \int^{t}_{0} \varrho_{1}' \left(S_{\varrho_{1}}(t-r)' i_{p,\varrho_{1}}' \tilde{B}(r,X_{r}) \right)^{2} dr  \nonumber \\
& \leq  M^{2}_{\varrho_{1}} e^{2\theta_{\varrho_{1}} t} \norm{i'_{p,\varrho_{1}}}^{2}_{\mathcal{L}(\Psi'_{p},\Psi'_{\varrho_{1}})} \Exp \int_{0}^{T} p'(\tilde{B}(r,X_{r}))^{2} dr  < \infty.  \label{secondMomeDeterConvoExisUniq}
\end{flalign} 
Thus, for every $t \geq 0$ the Bochner integral $\int^{t}_{0} S_{\varrho_{1}}(t-r)' i_{p,\varrho_{1}}' \tilde{B}(r,X_{r})  dr$ is defined $\Prob$-a.e. and hence $\left\{ \int^{t}_{0} S_{\varrho_{1}}(t-r)' i_{p,\varrho_{1}}' \tilde{B}(r,X_{r})  dr: t \in [0,T] \right\}$ is a $\Psi'_{\varrho_{1}}$-valued $\{ \mathcal{F}_{t} \}$-adapted square integrable process. Moreover, because for $0 \leq s \leq t \leq T$, we have $\Prob$-a.e.
\begin{flalign*}
& \varrho_{1}' \left( \int^{s}_{0} S_{\varrho_{1}}(s-r)' i_{p,\varrho_{1}}' \tilde{B}(r,X_{r}) dr -\int^{t}_{0} S_{\varrho_{1}}(t-r)' i_{p,\varrho_{1}}' \tilde{B}(r,X_{r}) dr  \right)  \\
& \leq \int^{T}_{0} \ind{[0,s]}{r} \varrho_{1}' \left((S_{\varrho_{1}}(s-r)'-S_{\varrho_{1}}(t-r)') i_{p,\varrho_{1}}' \tilde{B}(r,X_{r}) \right) dr \\
& \hspace{10pt} + \int^{T}_{0} \ind{[s,t]}{r} \varrho_{1}' \left(S_{\varrho_{1}}(t-r)' i_{p,\varrho_{1}}' \tilde{B}(r,X_{r}) \right) dr. 
\end{flalign*}
Then, by \eqref{secondMomeDeterConvoExisUniq} and following similar arguments to those used in the proof of 
Theorem \ref{theoExistHilbPredicCadlagVersionStochConvol}, we can show that $\left\{ \int^{t}_{0} S_{\varrho_{1}}(t-r)' i_{p,\varrho_{1}}' \tilde{B}(r,X_{r})  dr: t \in [0,T] \right\}$ is continuous $\Prob$-a.e. As it is also $\{ \mathcal{F}_{t} \}$-adapted then it has a predictable version (see \cite{PeszatZabczyk}, Proposition 3.21, p.27). 

Now, for every $t \in [0,T]$ we define 
\begin{equation} \label{defiDetermConvoProfExistUniq}
\int^{t}_{0} S(t-r)' B(r,X_{r})  dr \defeq  i'_{\varrho_{1}} \int^{t}_{0} S_{\varrho_{1}}(t-r)' i_{p,\varrho_{1}}' \tilde{B}(r,X_{r})  dr. 
\end{equation}
Then, we have that  $\left\{ \int^{t}_{0} S(t-r)' B(r,X_{r})  dr: t \in [0,T] \right\}$ is a $\Psi'_{\beta}$-valued, regular, square integrable predictable process. Moreover, if $K \in \mathcal{K}_{H}(\Psi)$ it follows from \eqref{secondMomeDeterConvoExisUniq} and \eqref{defiDetermConvoProfExistUniq} that
\begin{flalign} 
& \sup_{t \in [0,T]} \Exp \left( \eta_{K}\left( j_{K} \int^{t}_{0} S(t-r)' B(r,X_{r})  dr \right)^{2}  \right)   \nonumber \\
& \leq  \norm{j_{K} i'_{\varrho_{1}}}^{2}_{\mathcal{L}(\Psi'_{\varrho_{1}},\Psi'_{K})} \sup_{t \in [0,T]} \Exp   \int^{t}_{0} \varrho_{1}' \left(S_{\varrho_{1}}(t-r)' i_{p,\varrho_{1}}' \tilde{B}(r,X_{r}) \right)^{2} dr < \infty.  \label{finiteNormDeteConvExisUniq}
\end{flalign}
Hence, $\left\{ \int^{t}_{0} S(t-r)' B(r,X_{r})  dr: t \in [0,T] \right\} \in \mathcal{H}^{2}(T,\Psi'_{\beta})$ and therefore $\mathbb{A}_{1}$ is well-defined. Furthermore, note that for every $t \in [0,T]$ and $\psi \in \Psi$ we have $\Prob$-a.e. 
\begin{eqnarray*}
\left( \int^{t}_{0} S_{\varrho_{1}}(t-r)' i_{p,\varrho_{1}}' \tilde{B}(r,X_{r})  dr \right) [i_{\varrho_{1}} \psi] 
& = &  \int^{t}_{0} S_{\varrho_{1}}(t-r)' i_{p,\varrho_{1}}'\tilde{B}(r,X_{r}) [i_{\varrho_{1}} \psi]  dr \\
& = &  \int^{t}_{0} i'_{\varrho_{1}}  S_{\varrho_{1}}(t-r)' i_{p,\varrho_{1}}' \tilde{B}(r,X_{r}) [\psi]  dr \\
& = &  \int^{t}_{0} S(t-r)' B(r,X_{r}) [\psi]  dr.
\end{eqnarray*}
Thus, our definition of $\int^{t}_{0} S(t-r)' B(r,X_{r})  dr$ in \eqref{defiDetermConvoProfExistUniq} coincides with that given in  \eqref{defNonRandomConvolutionIntegral}. 

(b) Our next objective is to show that $\mathbb{A}_{1}$ is a contraction. Let $K \in \mathcal{K}_{H}(\Psi)$. It follows from Proposition  \ref{existenceBoundedHilbertFamilyAndSemigroup} that there exists a $C_{0}$-semigroup $\{ S_{K}(t)' \}_{t \geq 0}$ on the Hilbert space $\Psi'_{K}$ satisfying \eqref{compatiDualSemigroupWithHilbertExtension}. Moreover, there exist $M_{K} \geq 1$ and $\theta_{K} \geq 0$ such that $\{ S_{K}(t)' \}_{t \geq 0}$ satisfies 
\begin{equation} \label{expBoundSemigroupPsiK}
\eta_{K}(S_{K}(t)f) \leq M_{K} e^{\theta_{K}t} \eta_{K} (f), \quad \forall f \in \Psi'_{K}.
\end{equation}
Now, observe that for every $t \in [0,T]$, $\Prob$-a.e. the Bochner integral $\int_{0}^{t} S_{K}(t-r)'j_{K} B(r,X_{r})dr$ exists in $\Psi'_{K}$. Indeed, because $j_{K} i'_{\varrho_{1}} \in \mathcal{L}(\Psi'_{\varrho_{1}},\Psi'_{K})$, then from \eqref{extendedSemigroupQWeakMild}, \eqref{compatiDualSemigroupWithHilbertExtension} and \eqref{defiDetermConvoProfExistUniq} we have $\Prob$-a.e. 
\begin{eqnarray}
 \int_{0}^{t} S_{K}(t-r)'j_{K} B(r,X_{r})dr 
& = & \int_{0}^{t} j_{K} S(t-r)' i_{p,\varrho_{1}}' \tilde{B}(r,X_{r}) dr \nonumber \\
& = & \int_{0}^{t} j_{K} i'_{\varrho_{1}} S_{\varrho_{1}}(t-r)' i_{p,\varrho_{1}}' \tilde{B}(r,X_{r}) dr \nonumber \\ 
& = & j_{K} i'_{\varrho_{1}} \int_{0}^{t} S_{\varrho_{1}}(t-r)' i_{p,\varrho_{1}}' \tilde{B}(r,X_{r}) dr \nonumber \\
& = & j_{K}  \int^{t}_{0} S(t-r)' B(r,X_{r})  dr.  \label{defDeterConvOnPsiKExistUniqX}
\end{eqnarray}
Then, for every $\upsilon \geq 0$ and $X, Y \in \mathcal{H}^{2}(T,\Psi'_{\beta})$, from Assumption \ref{lipschitzGrowthConditions}, \eqref{expBoundSemigroupPsiK} and   \eqref{defDeterConvOnPsiKExistUniqX},  we have (recall the definition of $\abs{\norm{ \cdot }}_{\upsilon,K,T}$ in \eqref{defiEquivalentSemiNormsTopologyContraction})
\begin{flalign*}
& \abs{\norm{ \mathbb{A}_{1}X- \mathbb{A}_{1}Y }}^{2}_{\upsilon,K,T} \\
& = \sup_{t \in [0,T]} e^{-\upsilon t} \Exp \left( \eta_{K} \left( j_{K} \int_{0}^{t} S(t-r)'B(r,X_{r}) dr -j_{K} \int_{0}^{t} S(t-r)'B(r,Y_{r}) dr   \right)^{2} \right) \\
& = \sup_{t \in [0,T]} e^{-\upsilon t} \Exp \left( \eta_{K} \left( \int_{0}^{t} S_{K}(t-r)'j_{K} (B(r,X_{r})- B(r,Y_{r})) dr \right) \right)^{2} \\
& \leq M^{2}_{K} e^{2\theta_{K} T} \sup_{t \in [0,T]} e^{-\upsilon t} \Exp \left( \int_{0}^{t} \eta_{K} \left( j_{K} (B(r,X_{r})- B(r,Y_{r})) \right) dr \right)^{2} \\
& \leq M^{2}_{K} e^{2\theta_{K} T} \sup_{t \in [0,T]} e^{-\upsilon t} \Exp \left( \int_{0}^{t} \sup_{\psi \in K} a(\psi, r) \,  \eta_{K} \left( j_{K} (X_{r}-Y_{r}) \right) dr \right)^{2} 
\end{flalign*}
Now, by the Cauchy-Schwarz inequality,
$$  \int_{0}^{t} \sup_{\psi \in K} a(\psi, r) \,  \eta_{K} \left( j_{K} (X_{r}-Y_{r}) \right) dr  \leq \left( \int_{0}^{t} \sup_{\psi \in K} a(\psi, r)^{2} dr\right)^{\frac{1}{2}} \left( \int_{0}^{t} \eta_{K} \left( j_{K} (X_{r}-Y_{r}) \right)^{2} dr \right)^{\frac{1}{2}}. $$ 
Then, 
\begin{flalign*}
& \abs{\norm{ \mathbb{A}_{1}X- \mathbb{A}_{1}Y }}^{2}_{\upsilon,K,T} \\
& \leq M^{2}_{K} e^{2\theta_{K} T} \left( \int_{0}^{T} \sup_{\psi \in K} a(\psi, r)^{2} dr \right) \sup_{t \in [0,T]} \int_{0}^{t} e^{-\upsilon (t-r)} e^{-\upsilon r} \Exp \left(  \eta_{K} \left( j_{K} (X_{r}-Y_{r}) \right)^{2} \right) dr \\
& \leq M^{2}_{K} e^{2\theta_{K} T} \left( \int_{0}^{T} \sup_{\psi \in K} a(\psi, r)^{2} dr \right) \left( \int_{0}^{T} e^{-\upsilon r} dr \right) \abs{\norm{ X- Y }}^{2}_{\upsilon,K,T}. 
\end{flalign*}
Therefore, we have 
\begin{equation} \label{inequalityContract1ExistMildSol}
\abs{\norm{ \mathbb{A}_{1}X- \mathbb{A}_{1}Y }}^{2}_{\upsilon,K,T} \leq C^{(1)}_{\upsilon,K,T} \abs{\norm{ X- Y }}^{2}_{\upsilon,K,T},  
\end{equation}
where 
\begin{equation} \label{constantContra1ExistMildSol}
C^{(1)}_{\upsilon,K,T} = M^{2}_{K} e^{2\theta_{K} T}  \left( \int_{0}^{T} \sup_{\psi \in K} a(\psi, r)^{2} dr \right) \left( \int_{0}^{T} e^{-\upsilon r} dr \right).
\end{equation}

  
\textbf{Step 3:} \emph{Estimating $\mathbb{A}_{2}$.} Again we proceed in two steps. 
  
(a) We will show that $\mathbb{A}_{2}$ is well-defined. Let $X \in \mathcal{H}^{2}(T,\Psi'_{\beta})$. First, we will show that $F_{X}= \{ F(r,u,X_{r}(\omega)):  r \in [0,T], \omega \in \Omega, u \in U \} \in \Lambda^{2}_{s}(T)$. In effect, for every $\psi \in \Psi$, we have from Assumption \ref{lipschitzGrowthConditions} that  
\begin{eqnarray*}
\Exp \int_{0}^{T} \int_{U} q_{r,u}(F(r,u,X_{r})'\psi)^{2}  \mu(du) dr 
& \leq & \Exp \int_{0}^{T}  b(\psi, r)^{2} \left(1+\abs{X_{r}[\psi]} \right)^{2}  dr \\
& \leq & 2 \left(1+ \sup_{t \in [0,T]} \Exp \abs{X_{t}[\psi]}^{2} \right) \int_{0}^{T}  b(\psi, r)^{2}  dr < \infty.
\end{eqnarray*} 
Hence, $F_{X} \in \Lambda^{2}_{s}(T)$.  Then, by Theorem \ref{theoExistHilbPredicCadlagVersionStochConvol} there exists a continuous Hilbertian semi-norm $\varrho_{2}$ on $\Psi$ such that the stochastic convolution process $S'\ast F_{X} = \{(S'\ast F_{X})_{t}\}_{t\geq 0}$ has a $\Psi'_{\varrho_{2}}$-valued, mean-square continous, predictable version $\widetilde{S'\ast F_{X}}= \{(\widetilde{S'\ast F_{X}})_{t}\}_{t\geq 0}$  satisfying 
\begin{equation} \label{finiteMomentConvoFX}
\sup_{t \in [0,T]} \Exp  \varrho_{2}' \left( (\widetilde{S'\ast F_{X}})_{t} \right)^{2} < \infty.
\end{equation}
Therefore, because $(S'\ast F_{X})_{t}=i'_{\varrho} \widetilde{(S'\ast F_{X})}_{t}$, for all $t \in [0,T]$ $\Prob$-a.e., then $S'\ast F_{X}$ is a $\Psi'_{\beta}$-valued, regular, square integrable process that has a predictable version. 
Moreover, given any $K \in \mathcal{K}_{H}(\Psi)$ we have from \eqref{finiteMomentConvoFX} that 
\begin{equation*}
\sup_{t \in [0,T]} \Exp \left[ \eta_{K}(j_{K} (S'\ast F_{X})_{t})^{2} \right] 
 \leq \norm{j_{K} i'_{\varrho_{2}}}^{2}_{\mathcal{L}(\Psi'_{\varrho_{2}},\Psi'_{K})}  
\sup_{t \in [0,T]} \Exp \varrho_{2}'\left( (\widetilde{S'\ast F_{X}})_{t} \right)^{2} < \infty.
\end{equation*}
Thus, $S'\ast F_{X} \in \mathcal{H}^{2}(T,\Psi'_{\beta})$ and hence $\mathbb{A}_{2}$ is well-defined.  

(b) Now we will show that $\mathbb{A}_{2}$ is a contraction. 

Let $K \in \mathcal{K}_{H}(\Psi)$,  $\upsilon \geq 0$ and $X, Y \in \mathcal{H}^{2}(T,\Psi'_{\beta})$. 
Denote by $\Psi_{K}$ the dual of the Hilbert space $\Psi'_{K}$. Let $j'_{K}$ be the dual operator of $j_{K}$. Then, $j'_{K}$ corresponds to the canonical inclusion from $\Psi_{K}$ into $\Psi$. Let $\{ \psi_{j} \}_{j \in \N}$ be a complete orthonormal system in $\Psi_{K}$. Then, from Parseval's identity, Fubini's theorem, \eqref{secondMomentWeakStochasticIntegral}, \eqref{weakStrongCompatibilityStochIntegral}, \eqref{compatiDualSemigroupWithHilbertExtension} and \eqref{expBoundSemigroupPsiK}, for every $t \in [0,T]$ we have 
\begin{flalign}
& \Exp \left(  \eta_{K} \left( j_{K} \int_{0}^{t} \int_{U} S(t-r)' \left(  F(r,u,X_{r}) -F(r,u,Y_{r}) \right) M(dr,du) \right)^{2} \right) \nonumber \\
& =  \sum^{\infty}_{j=1} \Exp \int^{t}_{0} \int_{U} q_{r,u}( \left(  F(r,u,X_{r})' -F(r,u,Y_{r})' \right)S(t-r)j'_{K} \psi_{j})^{2} \mu (du) dr \nonumber \\
& = \Exp \int^{t}_{0} \int_{U} \norm{\left(  F(r,u,X_{r})' -F(r,u,Y_{r})' \right)S(t-r)j'_{K}}^{2}_{\mathcal{L}_{2}(\Psi_{K},\Phi_{q_{r,u}})} \mu (du) dr \nonumber \\
& = \Exp \int^{t}_{0} \int_{U} \norm{S_{K}(t-r)' j_{K} \left(  F(r,u,X_{r}) -F(r,u,Y_{r}) \right)}^{2}_{\mathcal{L}_{2}(\Phi'_{q_{r,u}},\Psi'_{K})}\mu (du) dr \nonumber \\ 
& \leq M^{2}_{K} e^{2\theta_{K}t} \, \Exp \int^{t}_{0} \int_{U} \norm{j_{K} \, \left(  F(r,u,X_{r}) -F(r,u,Y_{r}) \right)}^{2}_{\mathcal{L}_{2}(\Phi_{q_{r,u}},\Psi'_{K})}\mu(du) dr
\label{ineqSecMomOnPsiKStochConv}
\end{flalign}
On the other hand, for every $t \in [0,T]$ we have from Parseval's identity, Fubini's theorem and Assumption \ref{lipschitzGrowthConditions} that 
\begin{flalign}
&\Exp \int^{t}_{0} \int_{U} \norm{j_{K} \, \left(  F(r,u,X_{r}) -F(r,u,Y_{r}) \right)}^{2}_{\mathcal{L}_{2}(\Phi_{q_{r,u}},\Psi'_{K})}\mu(du) dr \nonumber \\
& =  \sum^{\infty}_{j=1} \Exp \int^{t}_{0} \int_{U} q_{r,u}( \left(  F(r,u,X_{r})' -F(r,u,Y_{r})' \right) j'_{K} \psi_{j})^{2} \mu (du) dr \nonumber \\
& \leq \Exp \int^{t}_{0} \sup_{\psi \in K} b(\psi, r) \, \eta_{K} \left( j_{K} (X_{r}-Y_{r}) \right)^{2}  dr \label{inequLipschitzStochInteg}
\end{flalign}
Then, from \eqref{ineqSecMomOnPsiKStochConv}, \eqref{inequLipschitzStochInteg} and the Cauchy-Schwarz inequality we have
\begin{flalign*}
& \abs{\norm{ \mathbb{A}_{2} X- \mathbb{A}_{2} Y }}^{2}_{\upsilon,K,T} \\
& = \sup_{t \in [0,T]} e^{-\upsilon t} \Exp \left( \eta_{K} \left( j_{K} \int_{0}^{t} \int_{U} S(t-r)' \left(  F(r,u,X_{r}) -F(r,u,Y_{r}) \right) M(dr,du) \right)^{2} \right) \\
& \leq M^{2}_{K} e^{2\theta_{K}T} \sup_{t \in [0,T]} e^{-\upsilon t} \, \Exp \int^{t}_{0} \sup_{\psi \in K} b(\psi, r) \, \eta_{K} \left( j_{K} (X_{r}-Y_{r}) \right)^{2}  dr\\
& \leq M^{2}_{K} e^{2\theta_{K}T} \abs{\norm{ X- Y }}^{2}_{\upsilon,K,T} \sup_{t \in [0,T]}  \int^{t}_{0}  \sup_{\psi \in K} b(\psi, r) e^{-\upsilon (t-r)}  dr\\
& \leq M^{2}_{K} e^{2\theta_{K}T} \left( \int^{T}_{0} \sup_{\psi \in K} b(\psi, r)^{2} dr \right)^{1/2} \left( \int_{0}^{T} e^{-2\upsilon r} dr \right)^{1/2} \abs{\norm{ X- Y }}^{2}_{\upsilon,K,T}. 
\end{flalign*}
Therefore, 
\begin{equation} \label{inequalityContract2ExistMildSol}
\abs{\norm{ \mathbb{A}_{2}X- \mathbb{A}_{2}Y }}^{2}_{\upsilon,K,T} \leq C^{(2)}_{\upsilon,K,T} \abs{\norm{ X- Y }}^{2}_{\upsilon,K,T},  
\end{equation}
where 
\begin{equation} \label{constantContra2ExistMildSol}
C^{(2)}_{\upsilon,K,T} = M^{2}_{K} e^{2\theta_{K}T} \left( \int^{T}_{0} \sup_{\psi \in K} b(\psi, r)^{2} dr \right)^{1/2} \left( \int_{0}^{T} e^{-2\upsilon r} dr \right)^{1/2}.
\end{equation}

\textbf{Step 4} \emph{Collecting the estimates for $\mathbb{A}$}. 

It follows from Steps 1 to 3 that $\mathbb{A}$ is well-defined. Moreover, note that for every $K \in \mathcal{K}_{H}(\Psi)$, $\upsilon \geq 0$ and $X,Y \in \mathcal{H}^{2}(T,\Psi'_{\beta})$,  we have from the definition of the map $\mathbb{A}$ that
\begin{equation} \label{decompInequalityContractExistMildSol}
\abs{\norm{ \mathbb{A}X- \mathbb{A}Y  }}_{\upsilon,K,T}^{2} \leq 2 \abs{\norm{ \mathbb{A}_{1}X- \mathbb{A}_{1}Y  }}^{2}_{\upsilon,K,T}+ 2 \abs{\norm{ \mathbb{A}_{2}X- \mathbb{A}_{2}Y  }}^{2}_{\upsilon,K,T}. 
\end{equation}
Then, it follows from \eqref{inequalityContract1ExistMildSol}, \eqref{constantContra1ExistMildSol},  \eqref{inequalityContract2ExistMildSol},  \eqref{constantContra2ExistMildSol} and \eqref{decompInequalityContractExistMildSol} that \eqref{inequalityContractExistMildSol} is satisfied for $C_{\upsilon,K,T}=(2C^{(1)}_{\upsilon,K,T}+2C^{(2)}_{\upsilon,K,T})^{1/2}$ 
and then we can take $\upsilon$ sufficiently large such that $C_{\upsilon,K,T}<1$ and consequently $\mathbb{A}$ is a contraction on $\mathcal{H}^{2}(T,\Psi'_{\beta})$. 

Now let $X \in \mathcal{H}^{2}(T,\Psi'_{\beta})$. From Steps 1 to 3, there exist continuous semi-norms $\varrho_{0}$, $\varrho_{1}$ and $\varrho_{2}$ on $\Psi$ such that  $\widetilde{\mathbb{A}_{0}X} \defeq \left\{  S_{\varrho_{0}}(t)' i'_{q,\varrho_{0}}\tilde{Z}_{0}: t \in [0,T]  \right\}$ is a $\Psi'_{\varrho_{0}}$-valued, continuous, predictable version of $\mathbb{A}_{0}X$ that satisfies \eqref{finiteMomeHilbLinearTermMildSol}, $\widetilde{\mathbb{A}_{1}X} \defeq \left\{ \int^{t}_{0} S_{\varrho_{1}}(t-r)' i_{p,\varrho_{1}}' \tilde{B}(r,X_{r})  dr: t \in [0,T] \right\}$ is a 
$\Psi'_{\varrho_{1}}$-valued, continuous, predictable version of $\mathbb{A}_{1}X$ that satisfies \eqref{secondMomeDeterConvoExisUniq}, and $\widetilde{\mathbb{A}_{2}X} \defeq \widetilde{S'\ast F_{X}}= \{(\widetilde{S'\ast F_{X}})_{t}\}_{t\geq 0}$ is a $\Psi'_{\varrho_{2}}$-valued, mean-square continous, predictable version of $\mathbb{A}_{2}X$ that satisfies \eqref{finiteMomentConvoFX}. 

Let $\rho$ be a continuous Hilbertian semi-norm on $\Psi$ such that $\varrho_{i} \leq \rho$, for $i=0,1,2$. Then, the inclusions $i_{\varrho_{i},\rho}:\Psi_{\rho} \rightarrow \Psi_{\varrho_{i}}$, $i=0,1,2$ are linear and continuous. Hence, if we take 
$$\widetilde{\mathbb{A}X}= i'_{\varrho_{0},\rho} \, \widetilde{\mathbb{A}_{0}X}+ i'_{\varrho_{1},\rho} \,  \widetilde{\mathbb{A}_{2}X}+ i'_{\varrho_{2},\rho} \,  \widetilde{\mathbb{A}_{2}X},$$
then $\widetilde{\mathbb{A}X}$ is a $\Psi'_{\rho}$-valued predictable version of $\mathbb{A}X$ and from \eqref{finiteMomeHilbLinearTermMildSol}, \eqref{secondMomeDeterConvoExisUniq}, \eqref{finiteMomentConvoFX} it follows that $\widetilde{\mathbb{A}X}$ satisfies \eqref{boundedHilSpacMomenMildSol}. 
\end{prf}
  
We are ready to show that there exists a unique (up to modification) mild solution to \eqref{generalFormSEE} satisfying the statements of the theorem.

\begin{proof}[Proof of Theorem  \ref{theoExistenceAndUniquenessMildSolutions}]

For a fixed $T>0$, as the map $\mathbb{A}$ is a contraction on $\mathcal{H}^{2}(T,\Psi'_{\beta})$ (Lemma \ref{contractionWellDefined}) and this is a complete, Hausdorff, locally convex space,  it follows from the fixed point theorem on locally convex spaces (see \cite{CainNashed:1971}, Theorem 2.2) that $\mathbb{A}$ has a unique fixed point $X^{(T)}=\{ X^{(T)}_{t} \}_{t \in [0,T]}$ in $\mathcal{H}^{2}(T,\Psi'_{\beta})$. Therefore, $X^{(T)}$ satisfies \eqref{equationMildSolution} for all $t \in [0,T]$. Moreover, Lemma \ref{contractionWellDefined} shows that there exists a continuous Hilbertian semi-norm $\rho=\rho(T)$ on $\Psi$ such that $X^{(T)}$ has a $\Psi'_{\rho}$-valued predictable version $\tilde{X}^{(T)}=\{ \tilde{X}^{(T)}_{t} \}_{t \in [0,T]}$ satisfying $\sup_{t \in [0,T]} \Exp  \, \rho'\left(\tilde{X}^{(T)}_{t}\right)^{2} < \infty$. 

Let $\{ T_{n} \}_{n \in \N}$ any sequence of positive real numbers such that $\lim_{n \rightarrow \infty} T_{n}= \infty$ and 
for each $n \in \N$ let $X^{(T)}=\{ X^{(T_{n})}_{t} \}_{t \in [0,T]}$ as above. Let $X= \{ X_{t} \}_{t \geq 0}$ be given for each $t \geq 0$ by $ X_{t}= X^{(T_{n})}_{t}$ if $T_{n-1} \leq t < T_{n}$, where we take $T_{0}=0$. Then, is easy to see that $X$ is well defined and moreover that $X$ is a $\Psi'_{\beta}$-valued, regular, predictable process satisfying \eqref{equationMildSolution} for all $t \geq 0$. Therefore, $X$ is a mild solution to \eqref{generalFormSEE} and is unique up to indistinguishable versions. 

Finally, as $X$ is a mild solution to \eqref{generalFormSEE}, and from the arguments on the proof of Lemma \ref{contractionWellDefined} one can check that the conditions of Theorem \ref{theoEquiWeakMild} are satisfied, then it follows that $X$ is also a weak solution to \eqref{generalFormSEE}. 
\end{proof}

\section{Applications to Stochastic Evolution Equations Driven by  L\'{e}vy Noise} \label{subSectionALNC}

Let $\Phi$ be a barrelled nuclear space and $\Psi$ be a quasi-complete, bornological, nuclear space such that every continuous semi-norm on $\Psi'_{\beta}$ is separable. Let $L=\{ L_{t} \}_{t \geq 0}$ be a $\Phi'_{\beta}$-valued c\`{a}dl\`{a}g L\'{e}vy process with L\'{e}vy -It\^{o} decomposition \eqref{levyItoDecomposition}. 

In this section our objective is to study the existence of weak and mild solutions to the following \emph{L\'{e}vy-driven stochastic evolution equation}:
\begin{equation} \label{levyDrivenSEE}
d X_{t}= (A'X_{t}+ B (t,X_{t})) dt + \int_{\Psi'_{\beta}}F(t,u,X_{t}) L(dt,du),
\end{equation}
for all $t \geq 0$ with initial condition $X_{0}=Z_{0}$, where $Z_{0}$ is a $\Psi'_{\beta}$-valued, regular, $\mathcal{F}_{0}$-measurable, square integrable random variable. We do this by employing the L\'{e}vy -It\^{o} decomposition \eqref{levyItoDecomposition} to $L$ to write  \eqref{levyDrivenSEE} as: 
\begin{multline}\label{levyDrivenSEELID}
 d X_{t}= (A'X_{t}+ B (t,X_{t})) dt+ F(t,0,X_{t}) dW_{t} \\
+ \int_{B_{\rho'}(1)} F(t,u,X_{t}) \tilde{N}(dt,du)+ \int_{B_{\rho'}(1)^{c}} F(t,u,X_{t}) N(dt,df), 
\end{multline}
for all $t \geq 0$ with initial condition $X_{0}=Z_{0}$. We assume $A$, $\{S(t)\}_{t \geq 0}$, $B$ and $F$ satisfy Assumption \ref{assumptionsCoefficients} (A1), (A3), (A4), Assumption \ref{assumpForDualSpaceAndDualSemigroup} (2) and Assumption  \ref{lipschitzGrowthConditions} for $U= \Phi'_{\beta}$, $\mu=\nu$ where $\nu$ is the L\'{e}vy measure of $L$, and with the family of continuous Hilbertian semi-norms  $\{ q_{r,u}: r \in \R_{+}, u \in \Phi' \}$ given by 
\begin{equation}\label{defiSemiNormsSEELevyNoise}
q_{r,u}(\phi)= \begin{cases} \mathcal{Q}(\phi), & \mbox{if } u=0, \\  \abs{u[\phi]}, & \mbox{if } u \neq 0, \end{cases}
\end{equation}
where recall that $\mathcal{Q}$ denotes the covariance functional of the Wiener process $W$. 

However, note that in \eqref{levyDrivenSEELID} there is still a difficulty to overcome because we have not defined the stochastic integral with respect to the Poisson random measure $N$. We will provide a meaning to the solutions to \eqref{levyDrivenSEELID} by setting up the problem in a way that allow us to use our theory of stochastic evolution equations developed in Section \ref{sectionSEEDNS}. We will do this by using the properties of L\'{e}vy processes in duals of nuclear spaces (Section \ref{subsectionLPDNS}), and by generalizing to our more general context some arguments from Peszat and Zabczyk (\cite{PeszatZabczyk}, Section 9.7) used to show the existence of solutions to stochastic evolution equations driven by L\'{e}vy processes  in Hilbert spaces. 

First, from Theorem \ref{theoExistenceCadlagVersionLevyProc} there exists a weaker countably Hilbertian topology $\vartheta_{L}$ on $\Phi$ such that  $\Prob \left( L_{t} \in (\widetilde{\Phi_{\theta_{L}}})'_{\beta}, \, \forall \, t \geq 0 \right)=1$. Let $\{ \rho_{n} \}_{n \in \N}$ be an increasing sequence of continuous Hilbertian semi-norms on $\Phi$ that generates the topology $\theta_{L}$. Without loss of generality we can take $\rho_{1}=\rho$, where $\rho$ is a continuous Hilbertian semi-norm on $\Phi$ such that $\rho'$ satisfies  \eqref{integrabilityPropertyLevyMeasure}. Then, because $\{ \rho_{n} \}_{n \in \N}$ is increasing it follows that (see \cite{GelfandVilenkin}, Section IV.2.2)
$$  (\widetilde{\Phi_{\theta_{L}}})'_{\beta} = \bigcup_{n \in \N} \Phi'_{\rho_{n}}= \bigcup_{n \in \N} B_{\rho'_{n}}(n),$$
where we recall $B_{\rho'_{n}}(n)=\{ f \in \Phi: \rho_{n}'(f) \leq n  \}$. Observe that $\{ B_{\rho_{n}'}(n): n \in \N \}$ is an increasing sequence of  bounded, closed, convex, balanced subsets of $\Phi'_{\beta}$. Furthermore,  
\begin{equation} \label{levyAlmostSureCountUnionBallsInDual}
\Prob \left( L_{t} \in \bigcup_{n \in \N} B_{\rho'_{n}}(n), \, \forall \, t \geq 0 \right)=1. 
\end{equation} 

For each $n \in \N$ let $U_{n}=B_{\rho'_{n}}(n)$ and define $\tau_{n}$ by 
\begin{equation} \label{defiStoppingTimesTauNLevyNoise}
 \tau_{n}(\omega) \defeq \inf \{ t \geq 0: \Delta L_{t} (\omega) \notin U_{n} \}, \quad \forall \, \omega \in \Omega.
\end{equation}  
It is clear that $\tau_{n}$ is an $\{ \mathcal{F}_{t} \}$-stopping time. Moreover, from \eqref{levyAlmostSureCountUnionBallsInDual} it follows that $\tau_{n} \rightarrow \infty$ $\Prob$-a.e. as $n \rightarrow \infty$. Furthermore, observe that for each $n \in \N$, from the definition of $U_{n}$  we have 
\begin{equation} \label{intLevyMeasuOnUnIsBounded}
\int_{U_{n}} \, \abs{u[\phi]}^{2} \nu(du) \leq \rho_{1}(\phi)^{2} \int_{B_{\rho'_{1}}(1)} \rho'_{1}(u)^{2} \nu(du)+ n^{2} \rho_{n}(\phi)^{2} \nu (B_{\rho'_{1}}(1)^{c}) < \infty, \quad \forall \, \phi \in \Phi. 
\end{equation}

Let $\mathcal{R}= \mathcal{A} \cup \{0\}$. For every $n \in \N$, let $M_{n}=(M_{n}(t,A): r \geq 0, A \in \mathcal{R})$ be the L\'{e}vy martingale-valued measure given by
\begin{equation} \label{levyMartValuedMeas}
M(t,A) = W_{t} \delta_{0}(A) + \int_{U_{n} \cap (A \backslash \{0 \})} u \widetilde{N}(t,du), \quad \mbox{ for } \, t \geq 0, \, A \in \mathcal{R}. 
\end{equation}
Note that by \eqref{intLevyMeasuOnUnIsBounded} $M_{n}$ is well-defined and from Example \ref{examLevyMartValuedMeasure} it is a cylindrical martingale-valued measure. Moreover, for $m \leq n$, from the corresponding properties of the Poisson integral it is not difficult to check that $M_{n}-M_{m}$ is again a cylindrical martingale-valued measure on $\R_{+} \times \mathcal{R}$. Furthermore, observe that on the set $\{ t \leq \tau_{m}\}$  for every $0 \leq r \leq t$, $A \in \mathcal{R}$ and $\phi \in \Phi$
we have that
\begin{equation} \label{descripMnMinusMn}
M_{n}(r,A)(\phi)-M_{m}(r,A)(\phi)
=-r \int_{(U_{n} \setminus U_{1}) \cap (A \backslash \{0 \})} u[\phi] \nu(du) + r \int_{(U_{m} \setminus U_{1}) \cap (A \backslash \{0 \})} u[\phi] \nu(du). 
\end{equation}
Now, note that for each $r \geq 0$, $g \in \Psi'_{\beta}$ the growth and Lipschitz conditions on $F$ implies that $\int_{U_{n}\setminus U_{1}} F(r,u,g) u \, \nu(du)$ defined by 
$$ \left( \int_{U_{n}\setminus U_{1}} F(r,u,g) u \, \nu(du) \right)[\psi]
= \int_{U_{n}\setminus U_{1}} u[F(r,u,g)' \psi] u \, \nu(du), \quad \forall \, \psi \in \Psi, $$
is an element of $\Psi'_{\beta}$. Moreover, if we define $B_{n}$ by 
$$B_{n}(t,g)=B(t,g) + \int_{U_{n}\setminus U_{1}} F(r,u,g) u \, \nu(du),$$
we can check that $B_{n}$ also satisfies the growth and Lipschitz conditions. Therefore, for every $n \in \N$ the following abstract Cauchy problem
\begin{equation} \label{levyAbsCauchyProbUn}
\begin{cases}
d X^{(n)}_{t}= (A'X^{(n)}_{t}+ B_{n}(t,X^{(n)}_{t})) dt+\int_{U} F(t,u,X^{(n)}_{t}) M_{n} (dt,du), \quad \mbox{for }t \geq 0, \\
X^{(n)}_{0}=Z_{0}.
\end{cases}
\end{equation}
has a unique mild solution $X^{(n)}$, that is also a weak solution,  such that for each $T>0$ there exists a continuous Hilbertian semi-norm  $\eta_{n}=\eta_{n}(T)$ on $\Psi$ such that $\sup_{t \in [0,T]} \Exp \eta'_{n}(X^{(n)}_{t})^{2} < \infty $. 

\begin{theo} \label{theoCompatLevyAbsCauchyProbUn}
For every $t \in [0,T]$ and all $m \leq n$, $X^{(n)}_{t}=X^{(m)}_{t}$ $\Prob$-a.e. on $\{ t \leq \tau_{m}\}$. Moreover, the $\Psi'_{\beta}$-valued regular, predictable process $X$ defined by $X_{t}=X^{(m)}_{t}$ for $t \leq \tau_{m}$ is a mild and a weak solution to \eqref{levyDrivenSEELID}. 
\end{theo}
\begin{prf}
Let $t \in [0,T]$. First, note that for each $m \leq n$, from the properties of the Poisson integral it is easy to check that for all $A, B \in \mathcal{R}$ and all $\phi$, $\varphi \in \Phi$, the real valued processes $\{ M_{m}(t,A)(\phi) \}_{t \geq 0}$ and $\{ M_{n}(t,B)(\varphi)-M_{m}(t,B)(\varphi) \}_{t \geq 0}$ are independent. Then, from Proposition \ref{propDecompWeakIntegralSumIndepMartValMeasu}
and because $X^{(n)}$ is a mild solution to \eqref{levyAbsCauchyProbUn}, we have $\Prob$-a.e. 
\begin{eqnarray*}
(X^{(n)}_{t} - X^{(m)}_{t})[\psi] 
& = & \int_{0}^{t} (B(r,X^{(n)}_{r})-B(r,X^{(m)}_{r}))[S(t-r)\psi] dr \\
& {} & + \int_{0}^{t} \int_{U_{n}\setminus U_{1}} u[F(r,u,X^{(n)}_{r})'S(t-r)\psi] \nu (du) dr \\
& {} & - \int_{0}^{t} \int_{U_{m}\setminus U_{1}} u[F(r,u,X^{(m)}_{r})'S(t-r)\psi] \nu (du) dr \\
& {} & + \int_{0}^{t} \int_{U} (F(r,u,X^{(n)}_{r})'-F(r,u,X^{(m)}_{r})')S(t-r)\psi M_{m}(dr,du) \\
& {} & + \int_{0}^{t} \int_{U} F(r,u,X^{(n)}_{r})'S(t-r)\psi (M_{n}-M_{m})(dr,du) 
\end{eqnarray*}
Now, by \eqref{descripMnMinusMn} we have on the set $\{ t \leq \tau_{m}\}$  that 
\begin{flalign*}
& \int_{0}^{t} \int_{U} F(r,u,X^{(n)}_{r})'S(t-r)\psi (M_{n}-M_{m})(dr,du)  \\
& = - \int_{0}^{t} \int_{U_{n}\setminus U_{1}} u[F(r,u,X^{(n)}_{r})'S(t-r)\psi] \nu (du) dr \\
& + \int_{0}^{t} \int_{U_{m}\setminus U_{1}} u[F(r,u,X^{(n)}_{r})'S(t-r)\psi] \nu (du) dr .
\end{flalign*}
Therefore, it follows from the above calculation that 
\begin{flalign*}
& (X^{(n)}_{t} - X^{(m)}_{t})[\psi] \mathbbm{1}_{\{t \leq \tau_{m}\}} \\
& =  \int_{0}^{t} (B(r,X^{(n)}_{r})-B(r,X^{(m)}_{r}))[S(t-r)\psi] dr \mathbbm{1}_{\{t \leq \tau_{m}\}} \\
& + \int_{0}^{t} \int_{U_{m}\setminus U_{1}} u[(F(r,u,X^{(n)}_{r})'-F(r,u,X^{(m)}_{r})')S(t-r)\psi] \nu (du) dr \mathbbm{1}_{\{t \leq \tau_{m}\}} \\
& + \int_{0}^{t} \int_{U} (F(r,u,X^{(n)}_{r})'-F(r,u,X^{(m)}_{r})')S(t-r)\psi M_{m}(dr,du) \mathbbm{1}_{\{t \leq \tau_{m}\}}.
\end{flalign*}
Then, from \eqref{secondMomentWeakStochasticIntegral} and \eqref{defiSemiNormsSEELevyNoise} we have
\begin{flalign*}
& \Exp \left( \abs{X^{(n)}_{t}[\psi] - X^{(m)}_{t}[\psi]}^{2} \mathbbm{1}_{\{t \leq \tau_{m}\}} \right) \\
& \leq 4 \Exp  \int_{0}^{t} \abs{(B(r,X^{(n)}_{r})-B(r,X^{(m)}_{r}))[S(t-r)\psi]}^{2}  \mathbbm{1}_{\{t \leq \tau_{m}\}} dr \\
& + 4 \Exp \int_{0}^{t} \int_{U_{m}} q_{r,u}(F(r,u,X^{(n)}_{r})'-F(r,u,X^{(m)}_{r})')S(t-r)\psi)^{2}  \mathbbm{1}_{\{t \leq \tau_{m}\}} \nu (du) dr \\
& + 4 \Exp \int_{0}^{t} \int_{U} q_{r,u}(F(r,u,X^{(n)}_{r})'-F(r,u,X^{(m)}_{r})')S(t-r)\psi)^{2}  \mathbbm{1}_{\{t \leq \tau_{m}\}} \nu (du) dr. 
\end{flalign*}  
Let $Y(\psi,t)= \Exp \left( \abs{X^{(n)}_{t}[\psi] - X^{(m)}_{t}[\psi]}^{2} \mathbbm{1}_{\{t \leq \tau_{m}\}} \right) $. Observe that $\sup_{t \in [0,T]} Y(\psi,t)< \infty$. Moreover, from the Lipschitz conditions on $B$ and $F$ we obtain for $t \leq T$:
$$ Y(\psi,t) \leq 8 M_{\psi}^{2} e^{2 \theta_{\psi}^{2} t} \int_{0}^{t} (a(\psi,r)^{2}+b(\psi,r)^{2}) Y(\psi,r) dr. $$
Therefore, from Gronwall's inequality if follows that $Y(\psi,t)=0$. Hence, we have that $X^{(n)}_{t}[\psi] = X^{(m)}_{t}[\psi]$ $\Prob$-a.s. on $\{t \leq \tau_{m}\}$. But since this is true for every $\psi \in \Psi$, because the $\Psi'_{\beta}$-valued processes $X^{(n)}$ and $X^{(m)}$ are regular, it follows from Proposition \ref{propCondiIndistingProcess} that $X^{(n)}_{t}=X^{(m)}_{t}$ $\Prob$-a.e. on $\{ t \leq \tau_{m}\}$. Finally, because $X^{(n)}$ is a mild and a weak solution to \eqref{levyDrivenSEELID} on $\{ t \leq \tau_{m}\}$, then $X$ is also a mild and a weak solution to \eqref{levyDrivenSEELID}.

\end{prf}

\textbf{Acknowledgements} { The author would like thank David Applebaum for all his helpful comments and suggestions. 
Thanks also to the The University of Costa Rica for providing financial support through the grant 820-B6-202	``Ecuaciones diferenciales en derivadas parciales en espacios de dimensi\'{o}n infinita''. Some preliminary versions of part of this work were carried out at The University of Sheffield and the author wish to express its gratitude for the financial support.   
}

\end{document}